\documentclass[10pt,preprint]{amsart}
\usepackage{amsmath,amsthm,amsfonts,amssymb}
\usepackage{mathtools} 
\usepackage{mathrsfs} 
\usepackage[active]{srcltx} 
\usepackage{amsaddr}

\usepackage{adjustbox}
\usepackage{makecell}
\usepackage{amsmath, bm}
\usepackage{subcaption}
\usepackage{gensymb}
\usepackage{longtable}

\usepackage[capitalise]{cleveref}
\usepackage{threeparttable}

\usepackage[numeric]{amsrefs}

\usepackage{adjustbox}
\usepackage[myheadings]{fullpage}
\usepackage{float}

\graphicspath{ {final_plots_new/} }

\DeclareCaptionLabelSeparator{none}{ }
\theoremstyle{plain} 

\theoremstyle{definition}

\theoremstyle{remark}

\newtheorem{theorem}{Theorem}[section]
\newtheorem{remark}{Remark}[section]
\newtheorem{problem}{Problem}[section]
\newtheorem{procedure}{Procedure}[section]

\begin{document}

\title{ Adjoint for advection schemes on the sphere in ICON model}

\author{Ramaz Botchorishvili} 
\address{Faculty of Exact and Natural Sciences, Ivane Javakhishvili Tbilisi State University, Tbilisi, Georgia} 
\email[Corresponding author]{ramaz.botchorishvili@tsu.ge (Corresponding author)} 

\author{Hendrik Elbern} 
\address{Rhenish Institute for Environmental Research, University of Cologne, Cologne, Germany;
         Institute of Energy and Climate research 8: Troposphere, Research Center J\"ulich, J\"ulich, Germany } 
\email{h.elbern@fz-juelich.de} 

\author{Tamari Janelidze} 
\address{Faculty of Exact and Natural Sciences, Ivane  Javakhishvili Tbilisi State University, Tbilisi, Georgia;
         Institute of Energy and Climate research 8: Troposphere, Research Center J\"ulich, J\"ulich, Germany } 
\email{t.janelidze@fz-juelich.de} 


\date{\today}

\begin{abstract}
		Among the most advanced and sophisticated methods for state analysis of an atmospheric 
		system is the four dimensional variational data assimilation. The numerically challenging 
		task of this approach is the development and application of the adjoint model components. 
		For tracer transport in fluid dynamics accuracy of numerical advection schemes is vital. 
		It is even more important for applications in space-time variational data assimilation with 
		adjoint model version. We propose novel straightforward and efficient approach - artificial 
		source term method - for adjoint advection solver development. It has several benefits 
		compared to traditional adjoint model building technique. One of the attractive features of 
		the new approach is that it reuses existing advection solver code, thus resulting into 
		significant reduction of time needed for adjoint solver development. The stability, accuracy 
		and convergence of the adjoint schemes are investigated. The method is implemented and evaluated 
		for the linear advection equation on the sphere in the Icosahedral Nonhydrostatic Model (ICON). 
		Adjoint solvers developed by the conventional and the methods are compared against each other on a 
		collection of standard advection test cases and variational data assimilation test cases developed 
		here. The advantages of the artificial source term method is especially obvious in case of monotonic 
		advection equation solvers, granting, e.g. absence of oscillations and nonphysical negative 
		concentrations.
\end{abstract}

\maketitle

\section{Introduction}
	Chemistry transport models (CTM) are widely used for air pollution modelling.  These models 
	are complex and they are usually solved numerically using operator splitting methods when 
	equations describing chemical reactions, horizontal advection and vertical advection  are treated separately 
	at different splitting steps, see e.g. ~\cite{ElbernSchwingerBotchorishvili:Chemical}. 
	Incorporating observations from both in situ and remote sensing devices in CTMs, e.g. satellites etc., is 
	essential for the quality of air pollution forecast. The method used for this purpose is called data 
	assimilation that first appeared in meteorology ~\cite{LeDimetTalagrand1986}. Four dimensional variational 
	data assimilation integrates into the model observational data distributed in space and finite time interval. 
	For achieving this goal a constrained minimization problem is formulated where governing equations serve 
	as constraints and a quadratic cost function measures misfits between observations and model forecast. 
	Solving this minimization problem results in optimal initial data that are consistent with observations and  
	with model dynamics. These data are used for improved forecasting on longer time interval, 
	see  ~\cite{TalagrandCourtier1987},~\cite{ElbernSchmidtEbel1997},~\cite{Wang2001},~\cite{SanduChai2011} 
	for detailed exposition on the subject, see section ~\ref{sec:variational} for the formulation of variational 
	problem for the linear advection equation. 

	A central module in CTMs as part of a 4D-var system is  the advcetion algorithm and its ajoint.  \cite{vukicevic2001properties}  
	studied the influence of linear and non--linear numerical advection algorithm properties on variational data assimilation 
	results in a 2D idealized scalar advection framework. Results suggested that  exact the same scalar advection algorithm in 
	forward and adjoint computations obtains, at lower cost, an optimal solution accuracy that is consistent with the forward 
	model accuracy.

	\cite{thuburn2001adjoints} investigated the impact of switches in non-oscillatory advection schemes for their adjoints.
	They  showed that there is no possibility of smoothing the switches in nonoscillatory advection schemes to remove
        the discontinuities while retaining an obvious and desirable scaling property. On the basis of established 
        equivalence between Eulerian backtracking or retro-transport and adjoint transport with respect to an 
        air-mass-weighted scalar product, \cite{hourdin2006eulerian} studied the question which arises as to whether it is 
        preferable to use the exact numerical adjoint, or the retro-transport model for a model that is not time-symmetric. 
        They concluded that the presence of slope limiters in the Van Leer advection scheme can produce in sonic circumstances 
        unrealistic, even negative adjoint sensitivities. The retro-transport equation, on the other hand, generally produces 
        robust and realistic results.

        In the contex of the GEOS-Chem \cite{henze2007development} tested the accuracy of the adjoint model by comparing adjoint to  
        finite difference sensitivities, which are shown to agree within acceptable tolerances. They explore the robustness 
        of these results, noting how discontinuities in the advection routine hinder, but do not entirely preclude, the use of such
        comparisons for validation of the adjoint model.  

        A  comprehensive study on the consistency of the discrete adjoints of upwind numerical schemes was provided by 
        \cite{liu2008properties}. Both linear  and nonlinear  discretizations of the one-dimensional advection equation are considered, 
        representing finite differences or finite volumes and slope or flux-limited techniques, respectively. \cite{gou2011continuous}  
        studied the effect of using discrete and continuous adjoints of the advection equation in chemical transport modeling numerically. 
        They concluded that discrete advection adjoints are more accurate in point-to-point comparisons against finite differences, 
        whereas the continuous adjoints of advection perform better as gradients for optimization in 4D-Var data assimilation. 

        Considering  a variational implementation of the RETRO-TOM model, \cite{haines2014adjoint} demonstrated a time symmetric forward 
        advection scheme with second-order moments to be efficiently exploited in the backward adjoint calculations, at least for 
        problems in which flux limiters in the advection scheme are not required. For this case the authors found the flexibility and 
        stability of a 'finite difference of adjoint' formulation with the accuracy of an 'adjoint of finite difference' formulation.

        \cite{holdaway2015assessing} examined the tangent linear and adjoint versions of NASA's Goddard Earth Observing System version 5 
        (GEOS-5). Tests exhibited that piecewise parabolic methods with flux limiters development of unrealistically large perturbations 
        within the tangent linear and adjoint models, and that using a linear third-order scheme for the linearised model produces better 
        behaviour.
		
	Iterative methods are used for minimizing the quadratic cost function. Quasi-Newton limited memory methods are 
	among most popular minimization algorithms  ~\cite{LiuNocedal1989}. A numerically challenging task in variational 
	data assimilation is the development of the adjoint model for the gradient computation that is used by quasi 
	Newton algorithm.  For the linear advection equation, as demonstrated in ~\cite{thuburn2001nonoscillatory}, 
	the results depend on the implementation of adjoint advection schemes especially in case of nonlinear schemes. 
	Several different approaches can be used for adjoint solver development, see e.g. ~\cite{LeDimetTalagrand1986}, 
	\cite{nodet2016variational} for  derivation of adjoint equations in continuous setting for which numerical solver 
	has to be developed, see e.g. ~\cite{thuburn2001adjoints}(სანახავია) for the development of adjoint in discrete 
	setting - i.e. adjoint of a given numerical advection scheme,  and see ~\cite{hascoet2013tapenade} for the 
	development of adjoint code based on powerful automated differentiation tool Tapenade.   
	
	Here we propose simple and powerful approach for the adjoint development for linear advection equation in 
	conservative form. Our approach is based on introducing artificial source term that ensures consistency of 
	the adjoint numerical scheme with the adjoint of the linear advection equation under consideration. 
	Implementation of our approach is easy and time efficient: it essentially reuses given linear advection solvers 
	with minor modifications, maintains accuracy and stability of the given parent solver and it uses the same 
	parallelization routines. The proposed method is general, though here we consider advection schemes on triangular 
	mesh of the sphere in the ICON model. See subsection \ref{sec:ICONmesh} for short description 
	and ~\cite{wan2013icon} for details, as part of broader effort on development of a next generation 4D-Var data 
	assimilation system with CTM.  
	
	The paper is organized as follows: in section \ref{sec:variational} the linear advection equation on the sphere 
	is given, the variational problem on data assimilation is formulated, and the adjoint and cost functions are given; 
	in section \ref{sec:NumSchemes} a new method for building adjoint scheme is proposed that is applied for building 
	the adjoint solver using linear advection schemes in the ICON model, the properties  of the developed adjoint schemes 
	are investigated theoretically; in section \ref{sec:NumTests} the new method is studied numerically on test problems 
	in the context of data assimilation for tracer transport models.

\section{Variational problem formulation for linear advection}
\label{sec:variational}
\subsection{Linear advection equation on the sphere}
The linear advection equation in conservative form writes:
\begin{equation}\label{eq:LinearAdvection}
\frac{\partial(\rho q)}{\partial t}+\nabla\cdot(\rho q\vec{v})=0,
\end{equation}
where  $\rho$ and $q$ are fluid density and mixing ratio, respectively. We are interested in their advection on the surface of the sphere $\Omega$. 
Therefore $\vec{v}=(v_\lambda,v_\theta)^T$ is the 2D horizontal wind vector and $\nabla\cdot$ is a spherical horizontal divergence operator given by 
\begin{equation}\label{eq:NablaHorizontal}
	\nabla\cdot v=\frac{1}{\sin\theta}\big[\frac{\partial v_\lambda}{\partial \lambda}+\frac{\partial (v_\theta\sin\theta)}{\partial\theta}\big],
\end{equation}  
with  $0\leq\lambda\leq 2\pi$ longitude and $0\leq\theta\leq \pi$ latitude. For simplicity, we take the radius $R$ of the sphere to be one. 

Throughout this paper in (\ref{eq:LinearAdvection})  $\rho=\rho(t,\lambda, \theta)>0$ and $v=v(t,\lambda, \theta)$ are sufficiently smooth given functions, 
$q=q(t,\lambda, \theta)$ is the unknown function that should be determined from the equation (\ref{eq:LinearAdvection}) subject to the initial 
condition that writes:
\begin{equation}\label{eq:InitCondition}
q(0,\lambda,\theta)=q_0(\lambda,\theta),
\end{equation}
where $q_0(\lambda,\theta)$ is a sufficiently smooth given function. 

Notice that initial condition (\ref{eq:InitCondition}) is given on the surface of the sphere. Because of this initial condition is sufficient and no 
boundary conditions are needed, the problem (\ref{eq:LinearAdvection})-(\ref{eq:InitCondition}) is well posed for any finite time interval $[0,T]$.

\subsection{Variational problem, adjoint equation, gradient of cost function}

Consider the following cost function:
\begin{equation}\label{eq:CostFunc}
J(q_0)=f^b(q_0,q^b)+\int_{0}^{T}\int_{\Omega}f^o(q,q^o,t,\lambda,\theta)d\lambda d\theta dt,
\end{equation}
where   $q_0$ is initial value function from the initial condition (\ref{eq:InitCondition}); $q^b=q^b(\lambda, \theta)$ is given, it is the so called 
background function that is also used as first guess for iteration method minimizing $J(q_0)$ with respect to $q_0$; in $q^o$ the superscript "o" stands 
for observations and is a given function; the functions $f^b$ and $f^o$ will be detailed in next subsection. Function $f^b$ measures the misfit between initial and 
background functions, while $f^o$ measures misfit between mixing ratio and observations. Here we assume $f^b$ and $f^o$ are smooth enough, we also assume that 
gradients of $f^b$ and $f^o$ with respect to $q_0$, denoted here via $\nabla_{q_0} f^b$ and $\nabla_{q_0} f^o$ respectively, can be computed analytically. 

Now we can state the variational data assimilation problem in the following way:

\begin{problem}[Variational data assimilation]
Find a initial function $q_0$ minimizing cost function (\ref{eq:CostFunc}) subject to constraints (\ref{eq:LinearAdvection}),(\ref{eq:InitCondition}). 
\end{problem}

The gradient of cost function $J(q_0)$ writes:
\begin{equation}\label{eq:GradCostFunc}
\nabla_{q_0} J(q_0)=\nabla_{q_0}f^b(q_0,q^b) - \rho(0,\lambda, \theta)q^*(0,\lambda, \theta),
\end{equation}
where $q^*(0,\lambda, \theta)$ is solution at $t=0$ of the following adjoint equation
\begin{equation}\label{eq:AdjointLinearAdvection}
\rho (\frac{\partial q^*}{\partial t}+v\cdot \nabla q^*)=\nabla_{q}f^o(q,q^o,t,\lambda,\theta),
\end{equation}
which is subject to the initial condition at $t=T$, in particular, 
\begin{equation}\label{eq:AdjointInitCondition}
q^*(T,\lambda,\theta)=0.
\end{equation}

Notice that in the adjoint equation (\ref{eq:AdjointLinearAdvection}) on the right hand side the function $q$ is the solution of the conservative linear advection 
equation (\ref{eq:LinearAdvection}) with initial condition (\ref{eq:InitCondition}). Notice also that the problem  (\ref{eq:LinearAdvection}),(\ref{eq:InitCondition}) 
is integrated forward in time and the problem (\ref{eq:AdjointLinearAdvection}),(\ref{eq:AdjointInitCondition}) is integrated backward in time. 
Because of this reason the problem  (\ref{eq:LinearAdvection}),(\ref{eq:InitCondition}) is often referred as forward problem and the problem (\ref{eq:AdjointLinearAdvection}),
(\ref{eq:AdjointInitCondition}) is referred as adjoint problem. The same applies to the numerical solvers, i.e. forward solver and adjoint solver. Using these terms we 
can easily formulate procedure for the gradient computation. In particular, we have:

\begin{procedure}[Computing the gradient of the cost function] \hspace{2cm}
\begin{enumerate}
\item Solve the forward problem (\ref{eq:LinearAdvection}),(\ref{eq:InitCondition}) and store $q(t,\lambda,\theta)$.  \\
\item Put $q(t,\lambda,\theta)$ on the right hand side of the equation (\ref{eq:AdjointLinearAdvection}) and solve the adjoint problem (\ref{eq:AdjointLinearAdvection}),
      (\ref{eq:AdjointInitCondition}), store  $q^*(0,\lambda,\theta)$.\\
\item Insert $q^*(0,\lambda,\theta)$ in (\ref{eq:GradCostFunc}) and compute the gradient of the cost function $J(q_0)$ defined by (\ref{eq:CostFunc}). \\
\end{enumerate}
\end{procedure} 

\subsection{Cost function}

Cost functions of variational data assimilation problem are mainly given in discrete or semi-discrete form in scientific literature, 
see e.g.~\cite{ElbernSchwingerBotchorishvili:Chemical, Wang2001, SanduChai2011}  for fully discrete cost functions and 
see  ~\cite{ElbernSchmidtEbel1997, elbern1999four, Wang2001} for semi-discrete cost functions.  Under semi-discrete version 
we mean cost function where continuous time integration is used. Here we establish a relationship between discrete, 
semi-discrete and continuous cost functions that is needed later for studying convergence of developed numerical schemes. 
For this purpose we have to make a discretization in space and time of dependent and independent variables. 
Since finite volume schemes are of interest of the present paper, we denote by $\Omega_j$ finite volume cell $j$, 
$ \Omega_j \subset \Omega$, $\cup_{j=1}^{N_c} \Omega_j= \Omega$, where $|\Omega_j|$ is area of the cell $\Omega_j$, 
and $N_c$ is number of cells. The average value of function q on $\Omega_j$ at time $t$ is denoted as $q_j(t)$. 
The same value at time $t=t_n$ is denoted as $q_j^n$, $t_n=n\Delta t, n=0,1,..,N_T$, where $\Delta t= T/N_T$ is 
time discretization step, and $N_T$ is number of nodal points in time. In the data assimilation context all discretized function 
values at time $t$ are interpreted as vectors, e.g. $\vec{q}(t)=(q_1(t),q_2(t),..,q_N(t))^T$, $\vec{q}^n=(q_1^n,q_2^n,..,q_{N_c}^n)^T$, 
since such notations are convenient for formulating cost function. Adopting our notations the semidiscrete cost function 
from ~\cite{elbern1999four} writes:
\begin{equation}\label{eq:CostFuncSemidiscrete}
J(\vec{q}_0))=\frac{1}{2}(\vec{q~}^b-\vec{q}_0)^TB^{-1}(\vec{q~}^b-\vec{q}_0)+ \frac{1}{2}\int_{t_0}^{t_{N_T}}(\vec{q~}^o(t)-\textit{H}[\vec{q}(t)])^TR^{-1}(\vec{q~}^o(t)-\textit{H}[\vec{q}(t)])dt,
\end{equation}
where $B$ and $R$ matrices are the background and observation error covariance matrices, respectively, 
$\vec{q~}^o(t)=(q^o_{i_1}(t),q^o_{i_2}(t),..,q^o_{i_{N_o}}(t))^T$, $1 \leq i_k \leq N_c, k=1,2,..,N_o$,  $N_o$ 
is number of observations, $N_o < N_c$. Operator $\textit{H}$ projects the model state to observation space that practically 
means the following: $\textit{H}[\vec{q}(t)]=(q_{i_1}(t),q_{i_2}(t),..,q_{i_{N_o}}(t))^T$.  
We denote by $\mathcal{O}$ the set of indices corresponding to observations, $\mathcal{O}=\{ i_1,i_2,..,i_{N_o} \}$.

The fully discrete version of the cost function from ~\cite{ElbernSchwingerBotchorishvili:Chemical} writes: 
\begin{equation}\label{eq:CostFuncDiscrete}
J(\vec{q}_0))=\frac{1}{2}[\vec{q~}^b-\vec{q}_0]^TB^{-1}[\vec{q~}^b-\vec{q}_0]+
\frac{1}{2} \sum_{n=0}^{N_T}[\vec{q~}^{on}-\textit{H}(\vec{q~}^n)]^TR^{-1}[\vec{q~}^{on}-\textit{H}(\vec{q~}^n)].
\end{equation}

If we assume that covariance matrices $B$ and $R$ are the same in (\ref{eq:CostFuncSemidiscrete}) and (\ref{eq:CostFuncDiscrete}) then 
for the compatibility of semi-discrete and discrete cost functions the following modification is necessary:
\begin{equation}\label{eq:CostFuncDiscreteM}
	J(\vec{q}_0))=\frac{1}{2}[\vec{q~}^b-\vec{q}_0]^TB^{-1}[\vec{q~}^b-\vec{q}_0]+ 
	\frac{T}{2N_T} \sum_{n=0}^{N_T}[\vec{q~}^{on}-\textit{H}(\vec{q~}^n)]^TR^{-1}[\vec{q~}^{on}-\textit{H}(\vec{q~}^n)],
\end{equation}

Notice that (\ref{eq:CostFuncSemidiscrete}) and (\ref{eq:CostFuncDiscreteM}) are compatible in the sense that 
when $N_T \longrightarrow \infty$ limit of discrete cost function (\ref{eq:CostFuncDiscreteM}) coincides with 
semidiscrete cost function (\ref{eq:CostFuncSemidiscrete}). 

Studying covariance matrices is out of scope of this paper. Therefore, without loss of generality and for the convenience 
of further exposition, we replace inverse covariance matrices by some positive definite kernels 
denoted by $K_b(\xi,\xi')$ and $K_o(\xi,\xi')$, $\xi, \xi' \in \Omega$, $\xi=(\lambda, \theta)$, $\xi'=(\lambda', \theta')$, and we 
define the function $f^b$ from (\ref{eq:CostFunc}) in the following way:
\begin{equation}\label{eq:CostFunc:f^b}
	f^b(q_0,q^b)=\frac{1}{2|\Omega|^2} \int_{\Omega} \int_{\Omega} K_b(\xi,\xi')[q_0(\xi)-q^b(\xi)] [q_0(\xi')-q^b(\xi')] d\xi d\xi'.
\end{equation} 

After standard discretization of (\ref{eq:CostFunc:f^b}) we have:
\begin{equation}\label{eq:CostFunc:J^b}
J^b =  \sum_{i=1}^{N_c} \sum_{j=1}^{N_c} \frac{|\Omega_i||\Omega_j|}{2|\Omega|^2}   K_b(\xi_i,\xi'_j)[q_{0j}-q^b_i] [q_{0j}-q^b_j]= \frac{1}{2}[\vec{q~}^b-\vec{q}_0]^TK^b[\vec{q~}^b-\vec{q}_0],
\end{equation} 
where $K^b$ is a symmetric positive definite matrix with elements 
\begin{equation*}
K^b_{ij}=\frac{|\Omega_i||\Omega_j|}{2|\Omega|^2}   K_b(\xi_i,\xi'_j), ~i,j=1,2,..,N_c.
\end{equation*}

Thus consistency is ensured between discrete and continuous versions of background terms (\ref{eq:CostFunc:f^b}) and 
(\ref{eq:CostFunc:J^b}), i.e. $\lim_{N_c \rightarrow \infty} J^b = f^b$. 
The formulation will allow us to study the impact of mesh refinement on data assimilation in section ~\ref{sec:NumTests}. 

By analogy with (\ref{eq:CostFunc:J^b}) we define $J^o$ and its continuous version: 
\begin{multline}\label{eq:CostFunc:J^o}
J^o =  \frac{T}{2N_T} \sum_{n=0}^{N_T} \sum_{i \in \mathcal{O}} \sum_{j \in \mathcal{O}} \frac{|\Omega_i||\Omega_j|}{2|\Omega^o|^2}   K_o(\xi_i,\xi'_j)[q_i^{on}-q^n_i] [q_j^{on}-q^n_j]= 
\frac{T}{2N_T} \sum_{n=0}^{N_T}[\vec{q~}^{on}-\textit{H}(\vec{q~}^n)]^TK^o[\vec{q~}^{on}-\textit{H}(\vec{q~}^n)],
\end{multline}

\begin{equation}\label{eq:f^o}
f^o(q,q^o,t, \xi)= 
\begin{cases}
\frac{[q(t,\xi)-q^o(t,\xi)]}{2|\Omega^o|^2}  \int_{\Omega^o} K_o(\xi,\xi') [q(t,\xi')-q^o(t,\xi')] d\xi', & ~\xi \in \Omega^o,\\
0,  & \text{otherwise},
\end{cases}
\end{equation} 
where 
\begin{equation*}
K^o_{ij} = \frac{|\Omega_i||\Omega_j|}{2|\Omega^o|^2}   K_o(\xi_i,\xi'_j), ~{i,j \in \mathcal{O}},~\Omega^o=\cup_{i \in \mathcal{O}} \Omega_i.
\end{equation*} 
Consequently we have 
\begin{equation}\label{eq:Nabla_qf^o}
\nabla_q f^o(q,q^o,t, \xi)= 
\begin{cases}
\frac{1}{|\Omega^o|^2}  \int_{\Omega^o} K_o(\xi,\xi') [q(t,\xi')-q^o(t,\xi')] d\xi', & ~\xi \in \Omega^o,\\
0,  & \text{otherwise}.
\end{cases}
\end{equation} 

Notice that in (\ref{eq:f^o}) we assume that observations are given on a subdomain $\Omega^o$, $|\Omega^o| >0.$ 

\section{Numerical schemes}
\label{sec:NumSchemes}
\subsection{ICON grids}
\label{sec:ICONmesh}
The name "ICON"  stands for the joint project of the Max Plank Institute for Meteorology (MPI-M) and the German Weather Service (DWD) on development 
of \textbf{ICO}sahedral \textbf{N}onhydrostatic models ~\cite{wan2013icon}. In the earlier version of ICON model triangular and hexagonal discretizations 
were tested and later only the triangular version of the icosahedral mesh is used~\cite{wan2013icon}. In this paper we use the triangular version of icosahedral 
grids from "Published list of DWD (EDZW) ICON grids"  by MPI-M and DWD given at http://icon-downloads.zmaw.de/. Meshes are referred as "$Rn_rBn_b$ grids" and 
from this name one can recover the algorithm how the grid was constructed. In partciular, starting point in any "$Rn_rBn_b$ grid" is an icosahedron with vertices 
on the sphere that is entirely projected onto the sphere, then edges along a great arc are devided in $n_r$ equal parts and then each edge of the obtained triangles 
are recursively devided $n_b$ times. See ~\cite{Sadourny1968integration, wan2013icon} for detailed description of the algorithm, for different optimization 
approaches and for grid characteristics. Some important characteristics of these meshes are given in the  Table 1 in ~\cite{wan2013icon} that shows non uniformity 
of $Rn_rBn_b$ grids in terms of triangle area ratios that can reach values as high as $1.53$ for $R2B7$ grid.  The same table also shows that the ratio is increasing 
together with mesh refinement. Here we also give some additional characteristics of $Rn_rBn_b$ grids in the  ~\cref{tab:R2BnGrids}. 
When grids are refined  the edge ratio is also  increasing and its maximum can reach $1.34$ for $R2B7$ grid. Though inside each triangle the edge ratio is 
almost constant and it increases very slowly, for example in particular, from $1.1734$ for $R2B2$ grid to $1.1761$ for $R2B7$ grid. Close to 1 edge ratio also 
means that triangular mesh changes smoothly on the surface of the sphere. $R2B7$ is the finest grid we consider in this paper since. 
Notice that our cost function accounts for the variability of grid.  

\begin{table}[h]
	\centering
	\caption{ICON grids, $R2Bn_b$ }	
	\begin{adjustbox}{width=1\textwidth}		
		\begin{tabular}{r r r r r r r r}
			\hline
			Grid & Number of  & Number of & Min triangle             & Max:min       & Max:min        & Min edge \\
			     & triangular & triangle  & cell area, $\text{km}^2$ & edge length   & edge length    & length, km \\
			     & cells      & edges     &                          & ratio, global & ratio, triangle & \\
			\hline
							
				R2B0  & 80      & 120     & 6010381.55 & 1.14 & 1.1350 & 3526.95 \\
				R2B1  & 320     & 480     & 1440873.32 & 1.17 & 1.1660 & 1737.06 \\
				R2B2  & 1280    & 1920    & 333434.84  & 1.21 & 1.1734 & 836.96  \\
				R2B3  & 5120    & 7680    & 78835.01   & 1.24 & 1.1750 & 407.12  \\
				R2B4  & 20480   & 30720   & 18777.28   & 1.27 & 1.1754 & 198.71  \\
				R2B5  & 81920   & 122880  & 4507.50    & 1.30 & 1.1755 & 97.36   \\
				R2B6  & 327680  & 491520  & 1089.56    & 1.32 & 1.1756 & 47.87   \\
				R2B7  & 1310720 & 1966080 & 265.08     & 1.34 & 1.1761 & 23.61   \\
				\hline
				
		\end{tabular}
	\end{adjustbox}	
        \label{tab:R2BnGrids}
\end{table}

\subsection{Linear advection schemes in ICON model}
\label{sec:ICONschemes}
Tracer transport schemes in the ICON triangular version are based on ~\cite{miura2007upwind} for finite volume discretization 
and on ~\cite{ollivier2002high} for high order reconstruction, both methods adapted to triangular grids on a sphere. The scheme is 
referred as ICON-FFSL (ICON-Flux Form Semi-Lagrangean) in ~\cite{lauritzen2014geoscientific} for its second order version. 
In this paper we will also use the same abbreviation for ICON tracer transport schemes. These schemes can be used with or without limiters. 
Here we build adjoint schemes for both cases. Adjoint without limiter is given in the section \ref{sec:Adjoint:IntegrByParts} and the adjoint 
with limiter is introduced in the section \ref{sec:Adjoint:ArtSource}.  Limiters in ICON-FFSL 
schemes are based on  Zalesak's \textbf{F}lux \textbf{C}orrected \textbf{T}ransport (FCT) by ~\cite{zalesak1979fully} the limiter and its 
positive definite modifications by  ~\cite{schar1996synchronous} and ~\cite{harris2011flux}.  Details including numerical results on 
ICON tracer transport schemes  are given in  http://www.cgd.ucar.edu/cms/pel/transport-workshop/2011/16-Reinert.pdf, 
~\cite{wan2013icon} and ~\cite{lauritzen2014geoscientific}. For the convenience of further exposition on constructing adjoint schemes the above 
finite volume advection schemes ICON-FFSL can be written in the following flux form: 
\begin{equation} \label{fv:FFSL4LA}
	\frac{\rho_j^{n+1} q_j^{n+1}-\rho_j^n q_j^n}{\Delta t}+\frac{1}{| \Omega_j |} \sum_{i\in I_j} F_{ji}(\{\bar{\rho}_k^n, \bar{v}_k^n,q_k^n\}_{k\in K_{ji}},\vec{n}_{ji}, l_{ji})=0,
\end{equation} 
where $I_j$ is a set of reference numbers to cell interfaces surrounding cell $\Omega_j$, i.e. $I_j$ consists of three elements in  
case of triangular cell.  $F_{ji}$ is a numerical flux function on the cell interface between cells $\Omega_j$ and $\Omega_i$, 
set $K_{ji}$ contains reference numbers to cells that are used for computing the numerical flux function $F_{ji}$, $\vec{n}_{ji}$ is a unit 
outward normal of the cell interface, $l_{ji}$ is length of the edge shared by triangles $j$ and $i$, $\bar{\rho}_k^n$ and $\bar{v}_k^n$ 
stands for time average of $\rho_k$ and $\vec{v}_k$ on $[]t_n,t_{n+1}]$. Notice that $K_{ji}=K_{ij}$, $l_{ji}=l_{ij}$, $\vec{n}_{ji}=-\vec{n}_{ij}$, $F_{ji}=-F_{ij}$. 
The latter ensures mass conservation of the scheme (\ref{fv:FFSL4LA}). 

\subsection{Integration by parts method for adjoint schemes}
\label{sec:Adjoint:IntegrByParts}
Here we consider adjoint scheme for ICON-FFSL with numerical flux function without limiters. In this case ICON-FFSL scheme (\ref{fv:FFSL4LA}) can 
be equivalently written in the following form:

\begin{equation} \label{fv:FFSL:NonFluxForm}
\frac{\rho_j^{n+1} q_j^{n+1}-\rho_j^n q_j^n}{\Delta t} +
\frac{1}{| \Omega_j |} \sum_{i\in S_j} \alpha_{ji}(\{\bar{\rho}_k^n, \bar{v}_k^n\}_{k \in K_{ji}})q_{i}^n
=0, ~~S_j=\cup_{i\in I_j} K_{ji},
\end{equation}
where $S_j$ contains a reference number to cells on the stencil of the scheme, see ~\cref{fig:stencil}. The scheme (\ref{fv:FFSL:NonFluxForm}) 
is linear with respect to $\vec{q~}^n$ and it is the starting point for building the adjoint, i.e. for deriving the scheme that is consistent 
with equation (\ref{eq:AdjointLinearAdvection}). In the latter equation the approximation of the time derivative and the right hand side 
is straightforward. The spatial derivatives can be approximated, e.g. by means of rewriting second term in (\ref{fv:FFSL:NonFluxForm}) in matrix 
form and then using transpose matrix. Equivalent but easy to use approach is using discrete analogue of integration 
by parts formula, i.e. discrete analogue of the approach which was used for the derivation of (\ref{eq:AdjointLinearAdvection}). 
For the convenience of further exposition we set: $\alpha_{ji}=0, ~\text{if}~ i \notin S_j$. Thus we have

\begin{multline} \label{fv:SumByParts:NonFluxForm}
\sum_{j} | \Omega_j | q^{*,n}_j \frac{1}{| \Omega_j |} \sum_{i\in S_j} \alpha_{ji}(\{\bar{\rho}_k^n, \bar{v}_k^n\}_{k \in K_{ji}})q_{i}^n
= 
\sum_{j}  \sum_{i} \alpha_{ji}(\{\bar{\rho}_k^n, \bar{v}_k^n\}_{k \in K_{ji}})q_{i}^nq^{*,n}_j 
= \\
 \sum_{i}  \sum_{j} \alpha_{ji}(\{\bar{\rho}_k^n, \bar{v}_k^n\}_{k \in K_{ji}})q_{i}^nq^{*,n}_j =
\sum_{i} | \Omega_i | q^{n}_i \frac{1}{| \Omega_i |} \sum_{j\in S^*_i} \alpha_{ji}(\{\bar{\rho}_k^n, \bar{v}_k^n\}_{k \in K_{ji}})q^{*,n}_{i},
\end{multline}
where $ S^*_i$ is set of indexes for which $\alpha_{ji} \neq 0$. On account of (\ref{fv:SumByParts:NonFluxForm}), (\ref{fv:FFSL:NonFluxForm}) 
and (\ref{eq:AdjointLinearAdvection}) adjoint numerical scheme writes:
\begin{equation} \label{fv:Adjoint:NonFluxForm}
\bar{\rho}_j^{n}\frac{ q^{*,n+1}_j- q^{*,n}_j}{\Delta t} +
\frac{1}{| \Omega_j |} \sum_{i\in S^*_j} \alpha_{ij}(\{\bar{\rho}_k^n, \bar{v}_k^n\}_{k \in K_{ij}})q^{*,n}_i
=\nabla_{q}f^o(q^{*,n}_j,q^{o,n}_j,t_n,\xi_j). 
\end{equation}
Coefficients $\alpha_{ji}$ are given in the appendix \ref{A:Coef4Adjoint}.

\subsection{Artificial source term method for adjoint schemes}
\label{sec:Adjoint:ArtSource}
In this section we introduce a simple and efficient new method for developing the adjoint scheme and we apply it for building the 
adjoint based on ICON-FFSL. The method can be used when the numerical flux function is  limited with or without flux limiters. 
Our starting point is ICON-FFSL in flux form (\ref{fv:FFSL4LA}). Assuming the scheme is consistent with equation (\ref{eq:LinearAdvection}) 
the goal is constructing a numerical scheme that is consistent with the equation (\ref{eq:AdjointLinearAdvection}). The latter can be equivalently written 
\begin{equation}\label{eq:LA+ArtSource}
\rho \frac{\partial q^*}{\partial t}+
\nabla\cdot(\rho q^*\vec{v})=
q^*\nabla\cdot(\rho \vec{v}) +
\nabla_{q}f^o(q,q^o,t,\xi), 
\end{equation}
where $\nabla_{q}f^o(q,q^o,t,\xi)$ is defined by (\ref{eq:Nabla_qf^o}). 
We will use this equivalent formulation of the adjoint equation as starting point of our method. The equation in the 
form  (\ref{eq:LA+ArtSource}) is a good choice because we can effortlessly reuse discretization schemes developed 
for the linear advection equation in conservative form (\ref{eq:LinearAdvection}). In particular, the spatial discretization 
of ICON-FFSL (\ref{fv:FFSL4LA}) can be reused for the second and first term respectively, both in the left hand side and in the 
right hand side of the equation (\ref{eq:LA+ArtSource}). As a result we have the following numerical scheme for the gradient computation
\begin{multline} \label{fv:adjointFFSL4LA}
\bar{\rho}_j^n \frac{q^{*,n+1}_j- q^{*,n}_j}{\Delta t}+\frac{1}{| \Omega_j |} \sum_{i\in I_j} F_{ji}(\{\bar{\rho}_j^n, \bar{v}_k^n,q^{*,n}_k\}_{k\in K_{ji}},\vec{n}_{ji}, l_{ji})= \\
\frac{1}{| \Omega_j |} \sum_{i\in I_j} F_{ji}(\{\bar{\rho}_j^n, \bar{v}_k^n,q^{*,n}_j\}_{k\in K_{ji}},\vec{n}_{ji}, l_{ji}) + \nabla_{q}f^o(q^{n}_j,q^{o,n}_j,t_n,\xi_j).
\end{multline} 
Notice that the first sum in the numerical scheme (\ref{fv:adjointFFSL4LA}) is exactly the same as the sum in (\ref{fv:FFSL4LA}). The 
second sum in (\ref{fv:adjointFFSL4LA}) is in principle the same with $q^{*,n}_j$ instead of $q^{*,n}_k$ that makes the computation of 
the sum even easier. Therefore the particularly interesting feature of the numerical scheme (\ref{fv:adjointFFSL4LA}) is the 
following: it makes easy reusing of the source code of the parent scheme, in our case ICON-FFSL scheme, and that dramatically reduces 
time needed for adjoint code development, e.g. it took just a couple of days in our case.  

\subsection{Properties of the adjoint scheme with artificial source term}
 
It is expected that the adjoint scheme constructed by artificial source term method will have good properties in the case if the parent 
scheme has the same good properties as well. According to Lax's equivalence theorem the minimal set of properties of the schemes ensuring 
convergence are consistency and stability. Therefore we will assume that parent scheme has these properties and then we prove that its 
descendant adjoint scheme also enjoys the same properties. 

\begin{theorem}[Consistency theorem]\label{Th:Consistency}
	Suppose $q,\rho, \vec{v}, q^*$ are sufficiently smooth and numerical scheme (\ref{fv:FFSL4LA}) is consistent with the linear 
	advection equation (\ref{eq:LinearAdvection}) in the sense of local truncation error. Then adjoint scheme 
	with artificial source term (\ref{fv:adjointFFSL4LA}) is also consistent with adjoint equation (\ref{eq:LA+ArtSource}).
\end{theorem}
\begin{proof}
	From the consistency requirement of numerical scheme (\ref{fv:FFSL4LA}) with equation  (\ref{eq:LinearAdvection}) we easily 
	obtain consistency between corresponding terms in the scheme and equation. In particular it is clear that the expression 
\begin{equation}\label{consistency1}
\frac{\rho_j^{n+1} q_j^{n+1}-\rho_j^n q_j^n}{\Delta t}
														~\text{is consistent with}~
\frac{\partial(\rho q)}{\partial t}
\end{equation} 
at $(t_n, \xi_j)$ and therefore the expression 
\begin{equation}\label{consistency2}
\frac{1}{| \Omega_j |} \sum_{i\in I_j} F_{ji}(\{\bar{\rho}_k^n, \bar{v}_k^n,q_k^n\}_{k\in K_{ji}},\vec{n}_{ji}, l_{ji})
~\text{is consistent with}~
\nabla\cdot(\rho q\vec{v}).
\end{equation} 
at $(t_n, \xi_j)$. Assuming in the above expression $q_k=q_j$ means that $q$ is constant on the stencil and therefore 
\begin{equation}\label{consistency3}
\frac{1}{| \Omega_j |} \sum_{i\in I_j} F_{ji}(\{\bar{\rho}_k^n, \bar{v}_k^n,q_j^n\}_{k\in K_{ji}},\vec{n}_{ji}, l_{ji})
~\text{is consistent with}~
q_j \nabla\cdot(\rho \vec{v}).
\end{equation} 
Similarly with (\ref{consistency1}) it is evident that the first and fourth term in the numerical scheme with artificial source 
term (\ref{fv:adjointFFSL4LA}) are consistent at $(t_n, \xi_j)$ with first and fourth term in the adjoint equation  (\ref{eq:LA+ArtSource}). 
Substituting $q_j^n$ with $q_j^{*n}$ in (\ref{consistency2}) and (\ref{consistency3}) yields consistency at $(t_n, \xi_j)$ of the 
second and third terms in the numerical scheme (\ref{fv:adjointFFSL4LA}) with  the second and third term in the equation  (\ref{eq:LA+ArtSource}) 
that concludes the proof.
\end{proof}
\begin{theorem}[Stability theorem]\label{Th:Stability}
Suppose
\begin{enumerate}
	\item $\rho$ is time independent.
	\item $\rho \geq \rho_{min} >0$.
	\item $\rho, \vec{v},F, q, \tilde{q}$ are smooth enough.
	\item $\vec{q^n}, \vec{\tilde{q}}^{n}$ are numerical solutions constructed by numerical scheme (\ref{fv:FFSL4LA}) with initial values $\vec{q_0}, \vec{\tilde{q}}_{0}$ respectively.
	\item Under Courant-Friedrichs-Levy (CFL) condition with CFL  number $CFL_0$  the numerical scheme (\ref{fv:FFSL4LA}) is stable in some norm: 
	\begin{equation}\label{stab:FFSL4LA}
\|  \vec{q}^{~n+1} - \vec{\tilde{q}}^{~n+1} \| \leq (1+ C_0 \Delta t) \|  \vec{q}^{~n} - \vec{\tilde{q}}^{~n} \|, ~~0\leq n < N_t,
	\end{equation}
	constant $C_0$ is independent of $n$.
\end{enumerate}
If $1-5$ are valid the numerical scheme  with artificial source term (\ref{fv:adjointFFSL4LA}) is also stable under CFL condition with CFL 
number $CFL_*=2CFL_0$ in the same norm,
 	\begin{equation}\label{stab:adjointFFSL4LA}
 	\|  \vec{q}^{~*,n} - \vec{\tilde{q}}^{~*,n} \| \leq C_1 \|  \vec{q}^{~*}_0 - \vec{\tilde{q}}^{~*}_{0} \|, ~~0<n\leq N_t,
 	\end{equation}
where 	$C_1$ is some constant independent of $n$.
\end{theorem}
\begin{proof}
	Numerical scheme (\ref{fv:adjointFFSL4LA}) equivalently writes: 
\begin{multline} \label{fv:adjointFFSL4LA:split1}
\frac{1}{2}\bar{\rho}_j^n \frac{q^{*,1,n+1}_j- q^{*,n}_j}{\Delta t}+\frac{1}{| \Omega_j |} \sum_{i\in I_j} F_{ji}(\{\bar{\rho}_j^n, \bar{v}_k^n,q^{*,n}_k\}_{k\in K_{ji}},\vec{n}_{ji}, l_{ji})= 0,\\
\frac{1}{2}\bar{\rho}_j^n \frac{q^{*,2,n+1}_j- q^{*,n}_j}{\Delta t}=
\frac{1}{| \Omega_j |} \sum_{i\in I_j} F_{ji}(\{\bar{\rho}_j^n, \bar{v}_k^n,q^{*,n}_j\}_{k\in K_{ji}},\vec{n}_{ji}, l_{ji}) + \nabla_{q}f^o(q^{n}_j,q^{o,n}_j,t_n,\xi_j),\\
q^{*,n+1}_j=\frac{1}{2}(q^{*,1,n+1}_j+q^{*,2,n+1}_j).
\end{multline} 
The first numerical scheme in (\ref{fv:adjointFFSL4LA:split1}) can be obtained from (\ref{fv:adjointFFSL4LA}) by replacing $\bar{\rho}_j^{n+1}$ 
with $\bar{\rho}_j^n$, see requirement $1$ of the theorem, and by replacing $\Delta t$ with $2\Delta t$. 
Therefore, similar to (\ref{stab:FFSL4LA}), the estimate is valid for $q^{*,1,n+1}$ under the CFL condition with CFL number $CFL_*$ and we have:
\begin{equation}\label{stab:FFSL4LA:2}
\|  \vec{q}^{~*,1,n+1} - \vec{\tilde{q}}^{~*,1,n+1} \| \leq (1+ 2C_0 \Delta t) \|  \vec{q}^{~*,n} - \vec{\tilde{q}}^{~*,n} \|, ~~0\leq n < N_t.
\end{equation}
The first term in the right hand side of the second equation in (\ref{fv:adjointFFSL4LA:split1}) is consistent with $q_j \nabla\cdot(\rho \vec{v})$ 
according to theorem ~\ref{Th:Consistency}. Therefore, on account of the requirements of theorem ~\ref{Th:Stability}, 
we can assume that there exists a constant $C_2$ independent of $n$ such that the following inequality holds true: 
\begin{multline}\label{fv:adjointFFSL4LA:split2}
 \frac{2}{\bar{\rho}_j^n|\Omega_j |} | \sum_{i\in I_j}[ 
F_{ji}(\{\bar{\rho}_j^n, \bar{v}_k^n,q^{*,n}_j\}_{k\in K_{ji}},\vec{n}_{ji}, l_{ji})- 
F_{ji}(\{\bar{\rho}_j^n, \bar{v}_k^n,\tilde{q}^{*,n}_j\}_{k\in K_{ji}},\vec{n}_{ji}, l_{ji})] | \leq 
2C_2 |  q_j^{~*,n} -  \tilde{q}_j ^{~*,n} |.
\end{multline}
The second term on the right hand side of the second equation in (\ref{fv:adjointFFSL4LA:split1}) does not depend on $q^*$. Therefore, 
by analogy with explicit Euler time integration scheme on account of (\ref{fv:adjointFFSL4LA:split2}) we obtain the following stability inequality:
\begin{equation}\label{fv:adjointFFSL4LA:split3}
\|  \vec{q}^{~*,2,n+1} - \vec{\tilde{q}}^{~*,2,n+1} \| \leq (1+ 2C_2 \Delta t) \|  \vec{q}^{~*,n} - \vec{\tilde{q}}^{~*,n} \|, ~~0\leq n < N_t.
\end{equation}
Putting together the third equation in (\ref{fv:adjointFFSL4LA:split1}), (\ref{stab:adjointFFSL4LA}) and  (\ref{fv:adjointFFSL4LA:split3}) 
we obtain the following  stability estimate for $q*$ for one time step: 
\begin{equation}\label{fv:adjointFFSL4LA:split4}
\|  \vec{q}^{~*,n+1} - \vec{\tilde{q}}^{~*,n+1} \| \leq (1+ [C_0+C_2] \Delta t) \|  \vec{q}^{~*,n} - \vec{\tilde{q}}^{~*,n} \|, ~~0\leq n < N_t.
\end{equation}
From (\ref{fv:adjointFFSL4LA:split4}) we arrive to (\ref{stab:adjointFFSL4LA}) with $C_1=\exp((C_0+C_2)T)$ which concludes the proof.
\end{proof}
\begin{remark} 
	$CFL_*=2CFL_0$ is needed for the theoretical estimates only. In practice when calculating numerical tests the adjoint numerical 
	scheme with artificial source term is stable under the same CFL condition as a parent scheme. 
\end{remark}
\begin{remark} 
	Theorem \ref{Th:Stability} deals with time independent $\rho$. The case with time dependent $\rho$  can be also be easily treated 
	by a different equivalent formulation of the adjoint scheme with artificial source term. In particular a numerical scheme of the form 
\begin{equation*}
\bar{\rho}_j^n \frac{q^{*,n+1}_j- q^{*,n}_j}{\Delta t}=RHS^n
\end{equation*}
can be equivalently written as the following two step scheme:
\begin{equation*}
 \frac{\bar{\rho}_j^{n+1}q^{*,1,n+1}_j- \bar{\rho}_j^nq^{*,n}_j}{\Delta t}=RHS^n, ~~
 \frac{q^{*,2,n+1}_j- q^{*,1,n+1}_j}{\Delta t}= \frac{1}{\bar{\rho}_j^n}\frac{\bar{\rho}_j^{n+1}-\bar{\rho}_j^n}{\Delta t} q^{*,2,n+1}_j.
\end{equation*}
The second scheme is an implicit Euler scheme and obtaining stability estimate is trivial. For obtaining stability estimate for the first scheme 
theorem ~\ref{Th:Stability} can be invoked by similar arguments.
\end{remark}
\begin{remark}
	Theorems \ref{Th:Consistency} and \ref{Th:Stability} are sufficient for concluding convergence of artificial source term scheme since 
	consistency and stability ensure convergence according to Lax equivalence theorem.  
\end{remark}

\section{Numerical tests}
\label{sec:NumTests}
\subsection{Input for tests}

Here we consider two different types of tests: linear advection test cases and data assimilation test cases. Linear advection test 
cases are used for comparing adjoint schemes with their parent scheme. Therefore numerical calculations are done with three schemes: 
ICON-FFSL, adjoint scheme obtained by integration by parts formula,  and an adjoint scheme by introducing artificial source term. 
Two different flux limiters are used together with ICON-FFSL and with the adjoint with artificial source term. In particular flux 
limiters from ~\cite{zalesak1979fully}, \cite{schar1996synchronous} and from ~\cite{zalesak1979fully}, \cite{harris2011flux} are used.  
Data assimilation test cases are used for comparing two adjoint schemes presented in this paper. In both test cases the true solution 
is compared with numerical solutions computed with the above mentioned numerical schemes. Errors in numerical solutions are measured 
in three different absolute and relative  norms. In  particular the following norms are used:
\begin{multline*}
l_{1,rel}=\frac{\sum_{i=1}^{N_c}|\Omega_i||q_i-q_i^{true}|}{\sum_{i=1}^{N_c}|\Omega_i||q_i^{true}|}, ~~~~ 
l_{1,abs}=\sum_{i=1}^{N_c}|q_i-q_i^{true}|, \\
l_{2,rel}=\frac{\sqrt{\sum_{i=1}^{N_c}|\Omega_i|{(q_i-q_i^{true})}^2}}{\sqrt{\sum_{i=1}^{N_c}|\Omega_i|(q_i^{true})^2}}, ~~~~l_{2,abs}=\sqrt{\sum_{i=1}^{N_c}{(q_i-q_i^{true})}^2},\\
l_{\infty,rel}=\frac{\text{max}_{i=\overline{1,N_c}}|q_i-q_i^{true}|}{\text{max}_{i=\overline{1,N_c}}|q_i^{true}|}, ~~~~ 
l_{\infty,abs}=\text{max}_{i=\overline{1,N_c}}|q_i-q_i^{true}|.
\end{multline*}

Notice that $l_1$ and $l_2$ measure integral characteristics and $l_{\infty}$ measures the maximum local deviation from the true 
solution, but none of them measures either oscillations in the numerical solution directly or shape preservation. Therefore 
where appropriate these two criteria will be also used.  We also consider magnitude of cost function and  norms of its gradient as most 
important criteria for evaluating performance of numerical schemes in case of data assimilation tests cases that are given in the 
subsection ~\ref{sec:assimilation}.

Standard test cases are available in scientific literature, such as ~\cite{nair2008moving}, \cite{nair2010class}, \cite{nair2002mass}, 
\cite{williamson1992standard}, \cite{zalesak1979fully}. We collected them in the appendix \ref{A:Tests:Wind,Shape}, 
in  ~\cref{tab:InitCond}, and in ~\cref{tab:Velocity} as collection of initial conditions and velocity fields. 
In particular we define five different initial scalar fields with references in  ~\cref{tab:InitCond} and four different velocity 
vector fields as given in ~\cref{tab:Velocity}. Combination of in initial scalar fields and velocity vector fields are used for 
defining different test cases below. 

\subsection{Linear advection test cases}
\subsubsection{ICON-FFSL, standard adjoint and adjoint with artificial source term }
First we compare the parent scheme ICON-FFSL with  two derived adjoint schemes presented in this paper. Notice that ICON-FFSL 
is consistent with equation (\ref{eq:LinearAdvection}) and adjoint schemes are consistent with equation (\ref{eq:AdjointLinearAdvection}). 
These two equations coincide with each other if $\rho=const$, $\nabla_q f^o=0$ and  $div~\vec{v}=0$. The latter condition is 
satisfied by wind fields 1,2 and 4 in the ~\cref{tab:Velocity}. We also set $\rho=1$. With this selection of parameters we 
can compare all three numerical schemes against each other on solutions of the problem (\ref{eq:LinearAdvection}),(\ref{eq:InitCondition}). 
In particular we consider the following test problems
\begin{itemize}
	\item Solid body rotation with cosine bell. 
	\item Solid body rotation with slotted cylinder.
	\item Deformational flow with two cosine bells.
	\item Deformational flow with two slotted cylinders.
\end{itemize}

Numerical results for all these schemes are summarized in  ~\cref{tab:SolidBodyRotationcos,tab:SolidBodyRotationcyl,tab:DeformationalFlow1cos,tab:DeformationalFlow1cyl}  
and  on figures ~\cref{fig:CosineBellContours,fig:CosineBellCurve,fig:SlottedCylinderContours,fig:SlottedCylinderCurve}.  
From the numerical results we conclude the following:
\begin{enumerate}
	\item All three schemes produce almost similar numerical results if flux limiters are not used in ICON-FFSL and in the adjoint 
	      scheme with artificial source term.  
	\item If flux limiters are used then ICON-FFSL scheme and  adjoint scheme with artificial source term produce practically similar 
	      numerical results. Using different flux limiters has almost  no affect on accuracy of computation. 
	\item The standard adjoint scheme produces larger errors compared to the ICON-FFSL scheme and the adjoint scheme with artificial 
	      source term, if flux limiters are used, see ~\cref{fig:CosineBellCurve} and ~\cref{fig:SlottedCylinderCurve}.
	\item Tests are calculated on R2B4 grid with 20480 nodal points. If flux limiters are not applied then the maximum principle is 
	      violated in almost in a half of the nodal points, overshoots and undershoots result in nonphysical negative numerical values. 
	      ICON-FFSL and the adjoint with artificial source are free of these drawbacks when flux limiters are applied. 
\end{enumerate}

\subsubsection{Standard adjoint and adjoint with artificial source term }
In this subsection the adjoint scheme with artificial source term with flux limiters is compared with the standard adjoint scheme.  
Test problems featuring  deformational flow and moving vortices are considered. Numerical results are given 
in ~\cref{tab:DeformationalFlow2cos,tab:DeformationalFlow2cyl,tab:MovingVortices}, and 
in ~\cref{fig:DefFl1:2CosBellContour,fig:DefFl1:2CosBellCurve,fig:DefFl1:2SlottedCylinderContour,fig:DefFl1:2SlottedCylinderCurve,fig:DefFl2:2CosBellContour,fig:DefFl2:2CosBellCurve,fig:DefFl2:2SlottedCylinderContour,fig:DefFl2:2SlottedCylinderCurve,fig:MovingVortexContour,fig:MovingVortexAtT,fig:MovingVortexAtT/2}. 
The analysis of the numerical results suggests that the adjoint scheme with artificial source term is much more accurate then 
the standard adjoint scheme. In particular we observe the following properties:
\begin{enumerate}
	\item Standard adjoint scheme does not maintain the shape of the contours while the adjoint with artificial source is almost 
	      indistinguishable from the exact solution, see e.g. ~\cref{fig:MovingVortexContour}.
	\item The standard adjoint scheme produces larger errors compared to adjoint scheme with artificial source term. 
	\item Maximum principle is violated by the standard adjoint scheme resulting in overshoots and undershoots and negative 
	      concentrations in almost half of nodal points of the ICON grid.
	\item For some test problems the adjoint scheme with artificial source term with flux limiter ~ \cite{zalesak1979fully}, 
	      \cite{schar1996synchronous} also gives negative concentrations in almost a third of nodal points of 
	      the ICON grid, though negative values are of $10^{-8}$ magnitude at most. When flux limiter ~\cite{zalesak1979fully}, 
	      \cite{harris2011flux} is used then negative concentrations appear just in 14 nodal points and it's  magnitude 
	      is very small - $10^{-15}$. 
	\end{enumerate}

\subsection{Data assimilation test cases}
\label{sec:assimilation}
\subsubsection{Setup for tests}
\label{sec:setup}
In this subsection we study established test methods in the framework of passive tracer data assimilation. 
We compare against each other standard adjoint, artificial source term adjoint with limiter from ~\cite{zalesak1979fully}, 
\cite{schar1996synchronous} and without limiter. Numerical results of advection tests given in previous subsections have 
shown that standard adjoint and artificial source term adjoint without limiter produce almost similar results. Therefore 
for some tests one case of them only will be considered; artificial source term method with limiter is considered in all 
test cases.  The following three combinations of initial scalar fields and velocity vectors are selected:
\begin{enumerate}
\item Moving vortices that correspond to initial scalar field 3 from ~\cref{tab:InitCond} and to the vector field 
4 from ~\cref{tab:Velocity}. 
\item Deformational flow with cosine bells   that correspond to initial scalar field 4 from ~\cref{tab:InitCond} and 
      to the vector field 3 from  ~\cref{tab:Velocity}. 
\item Deformational flow with slotted cylinder that correspond to initial scalar field 5 from ~\cref{tab:InitCond} and 
      to the vector field 3 from ~\cref{tab:Velocity}. 
\end{enumerate}

Notice that for selected test problems with moving vortices $ div(\vec{v})=0$ and for the deformational 
flow $ div(\vec{v}) \neq 0$. 

Cost function ~(\ref{eq:CostFunc}) with discretization of the background term according to ~(\ref{eq:CostFunc:J^b}) and with 
discretization of the observation term corresponding to ~(\ref{eq:CostFunc:J^o}) is used in numerical tests. 
In this paper we do not study covariance models and the goal is development of adjoint solvers. 
Therefore operators $K_b$ and $K_o$ are set to identity operators in numerical tests.  The cost function also contains the 
functions $q_b$ and $q_o$ as input. Therefore for finalizing data assimilation related test problems we define observations 
and background initial condition. For different test problems this is done differently, in particular, observations $q_o$ 
are defined as follows:
\begin{itemize}
	\item For numerical tests with moving vortices exact solution is known for an arbitrary time moment $t$ and 
	therefore the exact solution is used for defining observations in selected points in space and in time, 
	i.e.  $q_o(t_k, {\xi}_i)=u_{exact}(t_k,{\xi}_i), ~k=1,2,..N_T,~i\in \mathcal{O}$ . 
	\item For numerical tests with deformational the flow exact solution is not known for an arbitrary time 
	moment $t$ and therefore ICON-FFSL will be used for computing a reference solution which is used for 
	defining observations in selected points in space 
	and in time, i.e.  $q_o(t_k, x_i)=u_{ICON-FFSL}(t_k,x_i), ~k=1,2,..N_T,~i\in \mathcal{O}$. 
\end{itemize}

The number of observation points and their location is different for different numerical tests and therefore they are presented 
together with concrete numerical tests in the next subsection.

The background initial condition is constructed as an error in the true initial condition. This is done in the following way:
\begin{itemize}
	\item For numerical tests with moving vortices the background initial condition is defined as the 
	true initial condition plus 	10 $\%$ of error in each nodal point. 
	\item For numerical tests with deformational flow in the true initial condition the error is introduced 
	in one half of computational domain only. For the second half of computational domain background initial 
	condition and true initial condition coincide with each other. These domains are selected such that one 
	cosine bell or slotted cylinder remains unchanged in the background initial condition. In those nodal 
	points where the error is introduced and where the true initial condition is not zero background initial condition 
	is defined as true initial condition plus 10 $\%$ of error. For the rest of nodal points of the half of 
	computational domain 1$\%$ of maximum of the true initial condition is added. 
\end{itemize}

The above procedure results in differences between true and background initial condition, details of which are 
given in the ~\cref{tab:Datass_norms_weights_0_1}.

For minimizing the cost function the quasi-Newton method LBFGS ~\cite{LiuNocedal1989} is used here. The developed solvers 
are applied for supplying gradient of the cost function in LBFGS. ICON-FFSL is used for forward 
runs in order to supply needed input data to adjoint solvers. Magnitude of the cost function is used as most 
suitable measure for comparing different adjoint solvers against each other. The difference between the true initial 
condition and the one found by minimization procedure can also be used where appropriate. 
Details on results of numerical tests are given in the next subsections. 

\subsubsection{Convergence of iteration process}
Here we study behavior of developed adjoint solvers by means of performing 300 LBFGS iterations. We suppose 300 iterations 
will be enough for demonstrating convergence of iteration process. Calculations are done on ICON R2B4 grid with 
5120 observation points. 
Observation points are distributed on the grid evenly. In case of 5120 observation points, every fourth point is observed.
For the test problem with moving vortices the initial cost function is around $2.5\cdot10^6$ and after 
final iterations the cost is reduced to $131.9$ for standard adjoint, to $30.7$ for artificial source term 
adjoint with limiter and to $19.2$ for artificial source term adjoint without limiter. For all three methods the progress of minimisation 
is illustrated on ~\cref{fig:MV5120_cost_it300_1-10,fig:MV5120_cost_it300_11-20,fig:MV5120_cost_it300_295-300}.  
For different number of iterations situation is different: for 10 iterations best reduction of the cost function is given 
by standard adjoint, for 100 iterations best result is given by artificial source term adjoint and for 300 iterations best 
result is given by artificial source term adjoint without limiter. 
Oscillations visible in ~\cref{fig:MV5120_cost_it300_11-20} are due to the restart in LBFGS method. 
Restart in minimisation is done when LBFGS algorithm reaches 5 attempts to find $\alpha$ step satisfying Wolfe 
conditions ~\cite{LiuNocedal1989}. After the restart background condition is updated by the last 
optimized initial condition and the minimisation process starts again. 
The figure shows that after restart the cost the function increases first and then it decreases again. Similar behavior is observed in case 
of all adjoint solvers considered. 

Background and observation costs are given in 
~\cref{fig:MV5120_obscost_it300_295-300} and ~\cref{fig:MV5120_bcost_it300_295-300} 
respectively for the iterations 295-300.  We see that the standard adjoint solver is better at minimizing 
the observation term, while the artificial source term adjoint solvers are much better in minimizing background 
term of the cost function. This example shows the importance of the right adjoint solvers on the 
minimization process: gradient defines descent direction and it is calculated using adjoint solver; therefore  
different adjoint solvers can lead to different priorities, e.g. which term to be reduced in the cost function. 

Standard adjoint and artificial source term adjoint  are numerically investigated on convergence  on test problem 
of deformational flow with two cosine bells on ICON R2B4 grid with 5120 observation points.  Initial cost function is 
around $3.5\cdot10^4$ and after 300 iterations the cost is reduced to $2.2$ for standard adjoint solver 
and to $4\cdot10^{-2}$ for artificial source term adjoint solver. Both methods ensure convergence, the process is 
illustrated on  ~\cref{fig:DF3cos5120_cost_it300_1-10,fig:DF3cos5120_cost_it300_11-20,fig:DF3cos5120_cost_it300_295-300}. 
Before 30 iterations standard adjoint gives better results and after 30 iterations cost function corresponding to 
artificial source term adjoint is smaller. In this test problem we do not observe a similar effect as in previous test 
problem behavior in the quest for prioritisation of the observation or background term. Both methods minimize 
background and observation terms of the cost function in the similar way. 

These numerical results suggest that for the selected test problems all considered adjoint solvers ensure convergence of 
iteration process, initial cost function is reduced approximately $10^5$ times after 300 iterations and 
those adjoint solvers that give best reduction of the cost function after 300 iterations are different from those 
which give better results for smaller number, e.g. 10 iterations. 

\subsubsection{Impact of number of observations}
We study numerically three test problems given in the subsection \ref{sec:setup}. 
Calculations are done on ICON R2B4 grid with 2560, 5120, 10240 and 20480 observation points for 50 LBFGS iterations. 
Notice that R2B4 grid contains 20480 nodal points, i.e. we also consider the idealized case when observations are given in 
all nodal points of the grid. Numerical results are given in ~\cref{fig:MV_R2B4_cost_it50} for tests with moving vortexes, 
in the ~\cref{fig:DF3cos_R2B4_cost_it50} for tests with deformational flow and cosine bells, 
in the  ~\cref{fig:DF3cyl_R2B4_cost_it50} for tests with deformational flow and slotted cylinders. These  results suggest 
the following: 
\begin{itemize}
	\item  In case of using standard adjoint solver the cost function monotonically increases together with the number of 
	observation points for tests with deformational flow. For tests with moving vortices cost function first decreases 
	for 5120 observation points and then it increases monotonically together with number of observation points.
	\item In case of using artificial source term adjoint solver with limiter the cost function decreases together with 
	the number of observation points for tests with deformational flow and cosine bells. For a deformational flow 
	with slotted cylinders the cost function increases for 20480 observation points though the growth is not as dramatic 
	as in case of using standard adjoint solver. For tests with moving vortices cost function also decreases though we 
	observe oscillations and peaks for 10240 observation points. 
\end{itemize}

Numerical results suggest that the efficiency of iterations decreases when increasing number of observation points, i.e. 
for the same number of iterations we get a larger cost function in case of more observation points. Using different adjoint 
solvers influences minimization process differently and the adjoint solver with artificial source term offers more reliable 
behavior when increasing number of observation points. 

\subsubsection{Effect of mesh refinement}
Mesh refinement is a standard procedure for studying numerically convergence of numerical methods. As a result of mesh 
refinement accuracy increases for convergent methods. Namely ICON FFSL and its descendant adjoint solvers should produce 
more accurate solutions on refined meshes. Here we study the effect of mesh refinement on selected test problems in the 
context of variational data assimilation. In numerical tests given in previous subsections the $R2B4$ grid was used. 
Here we consider sequence of grids $R2B4,R2B5,R2B6,R2B7$ with 20480 observation points. 
Corresponding cost functions are given in \cref{tab:Datass_mesh_ref}. We observe that the final cost 
function increases together with the mesh refinement since we keep same number of observation points for all meshes. 
Notice that for each test problem the starting value of cost function is almost the same for all meshes. 
Though after 50 iterations we observe large difference in cost function evolution as impact of using different adjoint solvers. 
Namely for the same number of iterations the artificial source term adjoint solver with limiter yields cost function values which are 
from 8 to 660  times smaller than cost functions obtained with the adjoint solver without limiter.
  
\subsubsection{Manipulating wieghts of background and observation terms}
Weights can be used for prioritizing the background or observation term in the cost function. In numerical tests 
considered in previous subsections  the weight was set to $0.5$ for both terms. Here in numerical tests  we gradually increase 
the relative 
weight of observation term up to $1$ and at the same time we decrease weight of background term down to $0$. Corresponding 
cost functions on $R2B4$ grid for $50$ LBFGS iterations and $5120$ observation points are 
given on ~\cref{fig:MV_weights,fig:DF3cos_weights,fig:DF3cyl_weights}. These figures suggest the following:
\begin{itemize}
\item For tests with moving vortices cost functions decrease when weight of observation term increases.  
Cost function values calculated by artificial source term solver without limiter is 
smaller then cost function with the same solver using a limiter, namely when background term disappears, i.e. weight of 
observation term is $1$, we have  $J_{o,Art.S.NoLimiter}/J_{o,Art.S.WithLimiter}= 0.248777957$.
\item  For tests with deformational flow and cosine bells cost function valuess decrease when the weight of observation 
term increases.  Cost function values calculated by artificial source term solver with limiter is smaller then cost 
function produced by standard adjoint solver, we have  $J_{o,Std.Adjoint}/J_{o,Art.S.WithLimiter}= 1.878180425$.
\item For tests with deformational flow and slotted cylinders cost functions increase together with weight of observation 
term. If the standard adjoint solver is used the growth of the cost function 
is faster and  its value is also bigger compared to the case when the artificial source term adjoint with limiter 
is used, $J_{o,Std.Adjoint}/J_{o,Art.S.WithLimiter}= 58.210774364$.
\end{itemize} 
In case when the background term disappears in the cost function the true initial scalar field can be compared with those 
found by data assimilation. The results are given in the 
 ~\cref{tab:Datass_norms_weights_0_1}. They suggest the following
\begin{itemize}
	\item For tests with moving vortices some error norms are approximately two times smaller in case of using 
	artificial source term adjoint without limiter compared to the same solver with limiter. 
	\item For tests with deformational flow and cosine bells the standard adjoint solver gives better result compared 
	to artificial source term adjoint solver with limiter, 
	namely  in $l_1,l_{1,rel},l_2,l_{2,rel}$ norms error is slightly smaller and for $l_{\infty},l_{\infty,rel}$ norms 
	the error is approximately 3 times smaller. 
	\item For tests with deformational flow and slotted cylinders using the adjoint solver with artificial source term 
	with limiter gives approximately 7 times smaller error norms compared 
	to the case when standard adjoint solver is used.  
\end{itemize}
Putting together ~\cref{fig:MV_weights,fig:DF3cos_weights,fig:DF3cyl_weights} and \cref{tab:Datass_norms_weights_0_1} we 
conclude that value of cost function is a good 
indicator for comparing adjoint solvers in the sense that if the cost function  is smaller, then the solution of the variational 
data assimilation better approximates the true initial condition. We observe this behavior in case of  
numerical tests considered with moving vortices and slotted cylinders. For the test problem with cosine bells this conclusion is 
not valid for 5120 observations and it is true for 20480 observations.

\section{Conclusions}
We considered two approaches for developing adjoints of ICON-FFSL schemes for the linear advection equation in the 
ICON model. The first approach is standard and it is used for building the adjoint of ICON-FFSL scheme without flux 
limiters. Another approach is new and it is used for constructing the adjoint scheme of ICON-FFSL with or without 
flux limiters. The new approach is based on rewriting the adjoint equation in flux form using an artificial source term and 
then using the same discretization method as its parent ICON-FFSL scheme.  Because of this feature the development 
of adjoint solver is easy and very fast if,as it is typically the case, the parent forward solver is available. 
Another advantage is that flux limiters can be applied in our artificial source term method in a straightforward way 
thus ensuring 
stability regardless of smoothness of the solution. Stability and consistency of the adjoint scheme is proved 
under assumption that the parent scheme has these properties. ICON-FFSL and its decsendant adjoint solvers are compared 
against each other on a collection of standard advection test cases and variational data assimilation test cases 
developed here. For some test cases when the solution is smooth, the adjoint solvers without limiter can provide competitive 
and even smaller errors. For other and more realistic of test cases the artificial source term solver with limiter is the 
clear winner ensuring reduction of the cost function even in cases where adjoint solvers without limiter fail to 
perform minimization.

\section{Appendicies}
\subsection{Input for tests: initial condition and wind field}
\label{A:Tests:Wind,Shape}
\begin{table}[h!] 
\centering
\caption{Initial scalar fields}\label{tab:InitCond}
\begin{adjustbox}{width=1\textwidth}
	\begin{tabular}{|l|l|c|c|} 
	\hline 
	 \textbf{No} &  \textbf{Initial scalar field} & \textbf{Parameters} & \textbf{Source}	\\
	\hline 
	1 & Cosine bell &  & \\
 	  &  $\begin{array}{lcl}
             q(\lambda,\theta, t_0)=\begin{cases}
	                         \frac{h_{max}}{2}(1+\cos\frac{\pi r}{\tilde{r}}), & \text{if } r<\tilde{r}\\	
	                         0,              & \text{otherwise}
	                           \end{cases} 
             \end{array}$ 
          & \makecell{ $r=arccos(\sin\theta_c\sin\theta+\cos\theta_c\cos\theta\cos(\lambda-\lambda_c))$
	\\$h_{max}=1$ \ \  $\tilde{r}=1/3$ \ \ $(\lambda_c,\theta_c)=(3\pi/2,0)$ } 
          &  \cite{williamson1992standard}  \\
	\hline 
	2 & Slotted cylinder &  & \\
	  & $\begin{array}{lcl}
	q(t_0,\lambda, \theta)=\begin{cases}
	c, & \text{if } r\leq\tilde{r}, |\lambda-\lambda_c|\geq\frac{\tilde{r}}{6}, \\
	c, & \text{if } r\leq\tilde{r}, |\lambda-\lambda_c|<\frac{\tilde{r}}{6},\ \  \theta-\theta_c<\frac{2}{3}\tilde{r}, \\
	b,              & \text{otherwise}
	\end{cases}
	\end{array}$ & \makecell{ $r=arccos(\sin\theta_c\sin\theta+\cos\theta_c\cos\theta\cos(\lambda-\lambda_c))$ 
	               \\ $\tilde{r}=1/2$ \\ $c=1$ \ \ $b=0$ \ \ $(\lambda_c,\theta_c)=(3\pi/2,0)$} & \cite{zalesak1979fully}  \\
	\hline
	3 & Vortex &  & \\
	  & $\begin{array}{lcl}
	         q(t_0,\lambda^{'}, \theta^{'})=1-tanh\big[\frac{\tilde{\rho}}{\gamma}\sin(\lambda^{'}-\omega(\theta^{'}) t)\big]
	     \end{array}$ 
	  & \makecell{ $\lambda^{'}(\lambda,\theta)=\text{arctan}\big[\frac{\cos\theta\sin(\lambda-\lambda_p)}
	                                                                   {\cos\theta\sin\theta_p\cos(\lambda-\lambda_p)-\cos\theta_p\sin\theta}\big]$,\\ 
	               $\theta^{'}(\lambda,\theta)=\text{arcsin}[\sin\theta\sin\theta_p+\cos\theta\cos\theta_p\cos(\lambda-\lambda_p)]$,\\
	               $\begin{array}{lcl}
	                  \omega(\theta^{'})=\begin{cases}
                                                V/(R\tilde{\rho}), & \text{if } \tilde{\rho}\neq0, \\
                                                0, & \text{otherwise}.
                                             \end{cases}
	                \end{array}$, \\
	               $V=v_0\frac{3\sqrt{3}}{2}sech^2(\tilde{\rho})tanh(\tilde{\rho})$ \\ 
	               $v_0=2\pi R/T$ \ \ $\tilde{\rho}=\tilde{\rho_0}\cos\theta^{'}$ \\
	               $R=6.371229\times10^6 [m]$ \ \  $T=1036800[s]$  \\
	               $(\lambda_p,\theta_p)=(\pi-0.8+\pi/4,\pi/4.8)$ \\
	               $\gamma=5$ \ \ $\tilde{\rho_0}=3$ } 
	  & \cite{nair2002mass}  \\
	\hline
	4 & Two cosine bells &  & \\
 	  & $\begin{array}{lcl}
         q(t_0,\lambda,\theta)=\begin{cases}
         b+ch_1(\lambda, \theta), & \text{if } r_1<\tilde{r}, \\
         b+ch_2(\lambda, \theta), & \text{if } r_2<\tilde{r}, \\
         b, & \text{otherwise,}
         \end{cases}\end{array}$ & \makecell{ $\begin{array} {lcl}
	h_i(\lambda, \theta)=\begin{cases}
	\frac{h_{max}}{2}(1+\cos\frac{\pi r}{\tilde{r}}), & \text{if } r_i<\tilde{r}\\	
	0,              & \text{otherwise}
	\end{cases} 
	\end{array}$  \\ $r_i=arccos(\sin\theta_{c,i}\sin\theta+\cos\theta_{c,i}\cos\theta\cos(\lambda-\lambda_{c,i}))$
	\\$i=1,2$ \ \ $h_{max}=1$ $\tilde{r}=1/2$ \ \ $c=1$ \ \ $b=0$ \\ 
	$(\lambda_{c,1},\theta_{c,1})=(3\pi/4,0)$ \ \ $(\lambda_{c,2},\theta_{c,2})=(5\pi/4,0)$ } 
          &  \cite{williamson1992standard}  \\
 	\hline
 	5 & Two slotted cylinders &  & \\
 	  & $\begin{array}{lcl}
 	     q(t_0,\lambda, \theta)=\begin{cases}
c, & \text{if } r_i\leq\tilde{r}, |\lambda-\lambda_{c,i}|\geq\frac{\tilde{r}}{6} \ \ \text{for } i=1,2, \\
c, & \text{if } r_1\leq\tilde{r}, |\lambda-\lambda_{c,1}|<\frac{\tilde{r}}{6},\ \  \theta-\theta_{c,1}<-\frac{5}{12}\tilde{r}, \\
c, & \text{if } r_2\leq\tilde{r}, |\lambda-\lambda_{c,2}|<\frac{\tilde{r}}{6},\ \  \theta-\theta_{c,2}>\frac{5}{12}\tilde{r}, \\
b,              & \text{otherwise},
\end{cases}
 	    \end{array}$ &\makecell{$r_i=arccos(\sin\theta_{c,i}\sin\theta+\cos\theta_{c,i}\cos\theta\cos(\lambda-\lambda_{c,i}))$ \\
 	    $i=1,2$ \ \ $\tilde{r}=1/2$ \ \ $c=1$ \ \ $b=0$ \\ 
 	    $(\lambda_{c,1},\theta_{c,1})=(3\pi/4,0)$ \ \ $(\lambda_{c,2},\theta_{c,2})=(5\pi/4,0)$ } 
 	    & \cite{nair2010class} \\
 	    \hline
	\end{tabular}	
\end{adjustbox}	
\end{table}
\begin{table}[h!]
\centering
\caption{Velocity vector}\label{tab:Velocity}
\begin{adjustbox}{width=1\textwidth}	
	\begin{tabular}{|l|l|c|c|}
	\hline
	\textbf{No} & \textbf{Velocity vector} & \textbf{Parameters} & \textbf{Source}	\\
	\hline
	1 & \multicolumn{2}{l|}{Solid body rotation: $\nabla\cdot \vec{v}=0$} &  \\ 
	  & \multicolumn{2}{l|}{Exact solution: after complete revolution, initial field reaches starting position} &  \\
        \hline
	  & $\begin{array} {lcl}
	v_\lambda(\lambda,\theta)=u_0(\cos\theta\cos\alpha+\sin\theta\cos\lambda\sin\alpha),\\
	v_\theta(\lambda,\theta)=-u_0\sin\lambda\sin\alpha
	\end{array}$  & \makecell{ $u_0=2\pi R/T$ \\ $R=6.371229\times10^6 [m]$ \ \ $T=1036800[s]$ \\ $\alpha=0^\circ$ \ \  $0\leq t\leq T$ } & \cite{williamson1992standard}  \\
	\hline
	2 & \multicolumn{2}{l|}{Deformational flow: $\nabla\cdot \vec{v}=0$} &  \\ 
	  & \multicolumn{2}{l|}{Exact solution: after complete revolution, initial field reaches starting position} &  \\
        \hline	  
	  & $\begin{array}{lcl}
	v_\lambda(\lambda,\theta, t)=k\sin^2(\lambda/2)\sin(2\theta)\cos(\pi t/T),\\
	v_\theta(\lambda,\theta, t)=\frac{k}{2}\sin\lambda\cos\theta\cos(\pi t/T)
	\end{array}$ &  $k=2.4$  \ \ $0\leq t\leq T$ & \cite{nair2010class}  \\
	\hline
	3 & \multicolumn{2}{l|}{Deformational flow: $\nabla\cdot \vec{v}\neq0$} &  \\ 
	  & \multicolumn{2}{l|}{Exact solution: after complete revolution, initial field reaches starting position} &  \\
        \hline	  
	  & $\begin{array}{lcl}
	v_\lambda(\lambda,\theta, t)=-k\sin^2(\lambda/2)\sin(2\theta)\cos^2\theta\cos(\pi t/T),\\
	v_\theta(\lambda,\theta, t)=\frac{k}{2}\sin\lambda\cos^3\theta\cos(\pi t/T)
	\end{array}$ & $k=1$ \ \ $0\leq t\leq T$  & \cite{nair2010class}  \\
	\hline
	4 & \multicolumn{2}{l|}{Moving vortices: $\nabla\cdot \vec{v}=0$} &  \\ 
	  & \multicolumn{2}{l|}{Exact solution: Forward- $q(t,\lambda,\theta)=1-tanh[\frac{\tilde{\rho}}{\gamma}\sin(\lambda^{'}-\omega t)]$, 
	                                       Backward- $q(t,\lambda,\theta)=1-tanh[\frac{\tilde{\rho}}{\gamma}\sin(\lambda^{'}+\omega t)]$ } &  \\
        \hline	  
	  & $\begin{array}{lcl}
	     v_\lambda(t,\lambda,\theta)=u_0(\cos\theta\cos\alpha+\sin\theta\cos\lambda\sin\alpha)+\\
	                                 R\omega(\theta^{'})(\sin\theta_c\cos\theta-\cos\theta_c\cos(\lambda-\lambda_c)\sin\theta), \\
	     v_\theta(t,\lambda,\theta)=-u_0\sin\lambda\sin\alpha,+R\omega(\theta^{'})(\cos\theta_c\sin(\lambda-\lambda_c)).
	     \end{array}$  
	 & \makecell{ $V=v_0\frac{3\sqrt{3}}{2}sech^2(\tilde{\rho})tanh(\tilde{\rho})$ \\ 
	              $v_0=2\pi R/T$\ \ $\tilde{\rho}=\tilde{\rho_0}\cos\theta^{'}$ \ \ $\tilde{\rho_0}=3$\\
	              $R=6.371229\times10^6 [m]$ \ \ $T=1036800[s]$ \\
                      $\begin{array}{lcl}
                         \omega(\theta^{'})=\begin{cases}
                                              V/(R\tilde{\rho}), & \text{if } \tilde{\rho}\neq0, \\
                                              0, & \text{otherwise}.
                                            \end{cases} 
                       \end{array}$  \\ 
                      $(\lambda_c,\theta_c)=(\lambda_p+\omega t_n, \theta_p)$ \\
                      $(\lambda_p,\theta_p)=(\pi-0.8+\pi/4,\pi/4.8)$ \\ 
                      $t_n=n\Delta t$ \ \  $\Delta t=600[s]$ \ \
                      $0\leq t\leq T$ } 
         & \cite{nair2008moving}  \\
	\hline		  
	\end{tabular}
\end{adjustbox}	
	\label{tab2}	
\end{table} 

   \clearpage
\subsection{Numerical results of advection tests}
\label{A:AdvecNumerical}
    
\subsubsection{Figures for advection tests}
\label{A:AdvTestPlots}

\begin{figure}[b]   
\begin{subfigure}{0.5\textwidth}
  \centering
  \includegraphics[width=\textwidth]{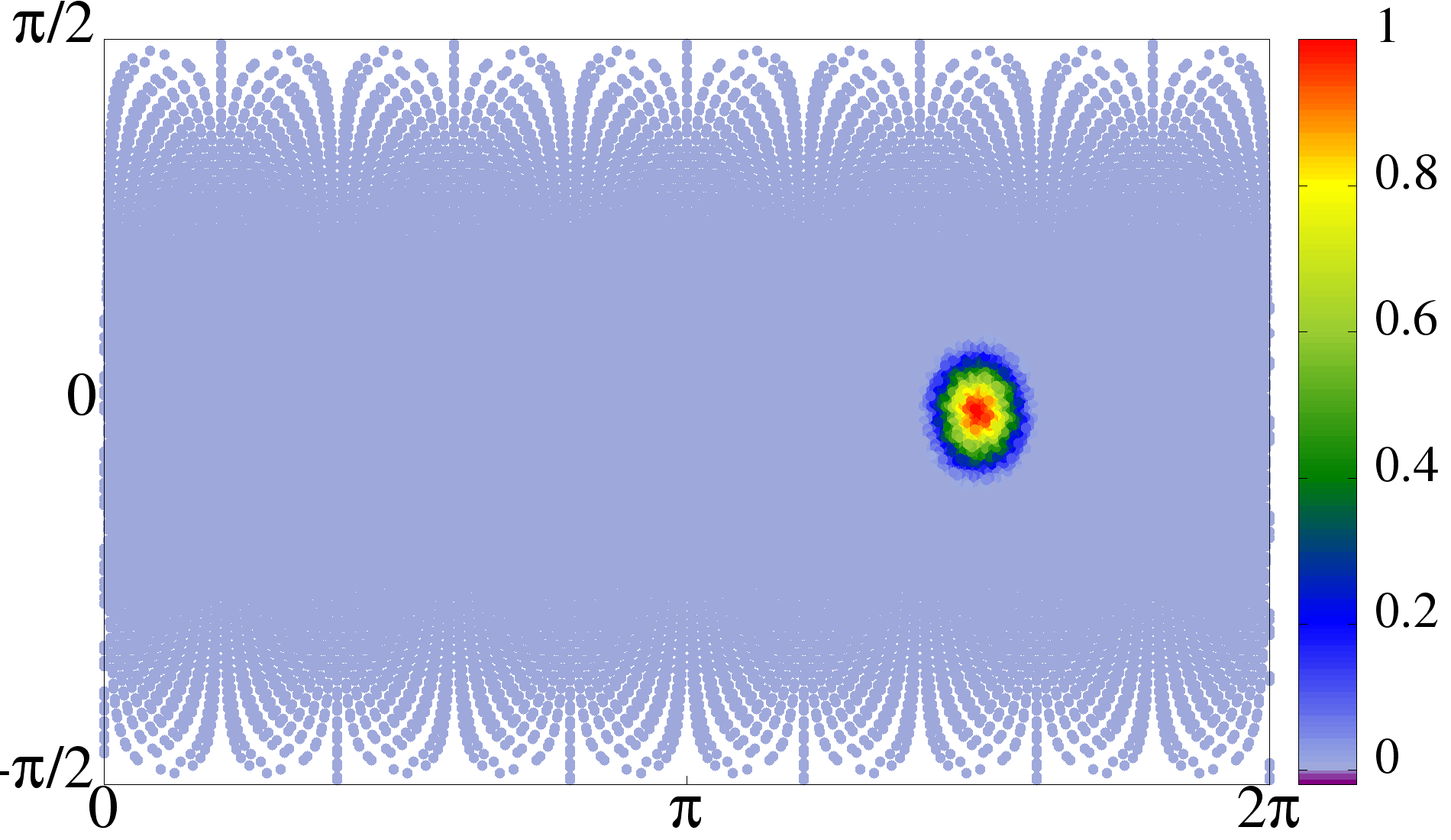}
  \caption{}
  \label{fig:PACosineExact}
\end{subfigure}%
\begin{subfigure}{0.5\textwidth}
  \centering
  \includegraphics[width=\textwidth]{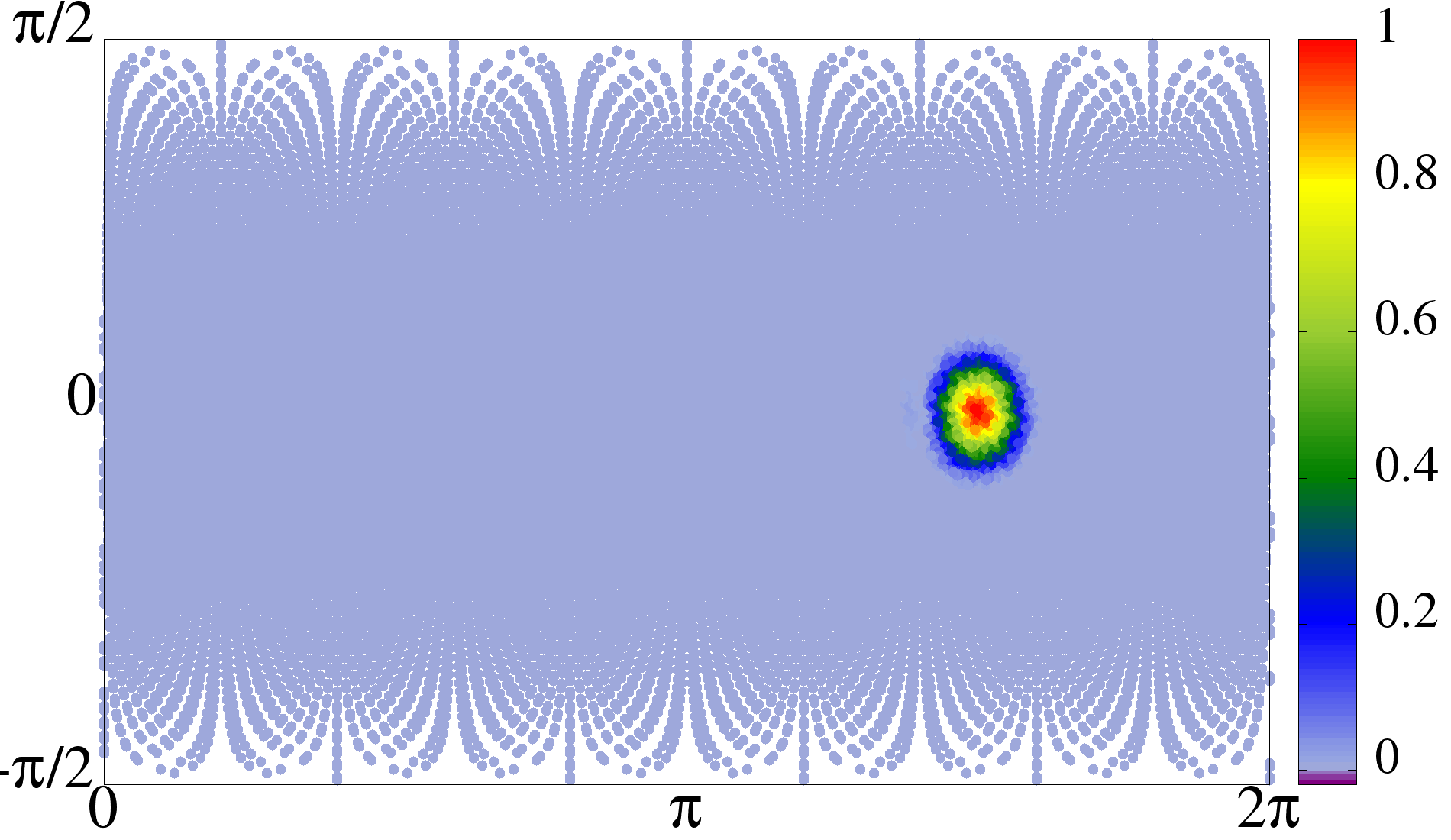} 
  \caption{}
  \label{fig:PACosineICON}
\end{subfigure}
\begin{subfigure}{0.5\textwidth}
  \centering
  \includegraphics[width=\textwidth]{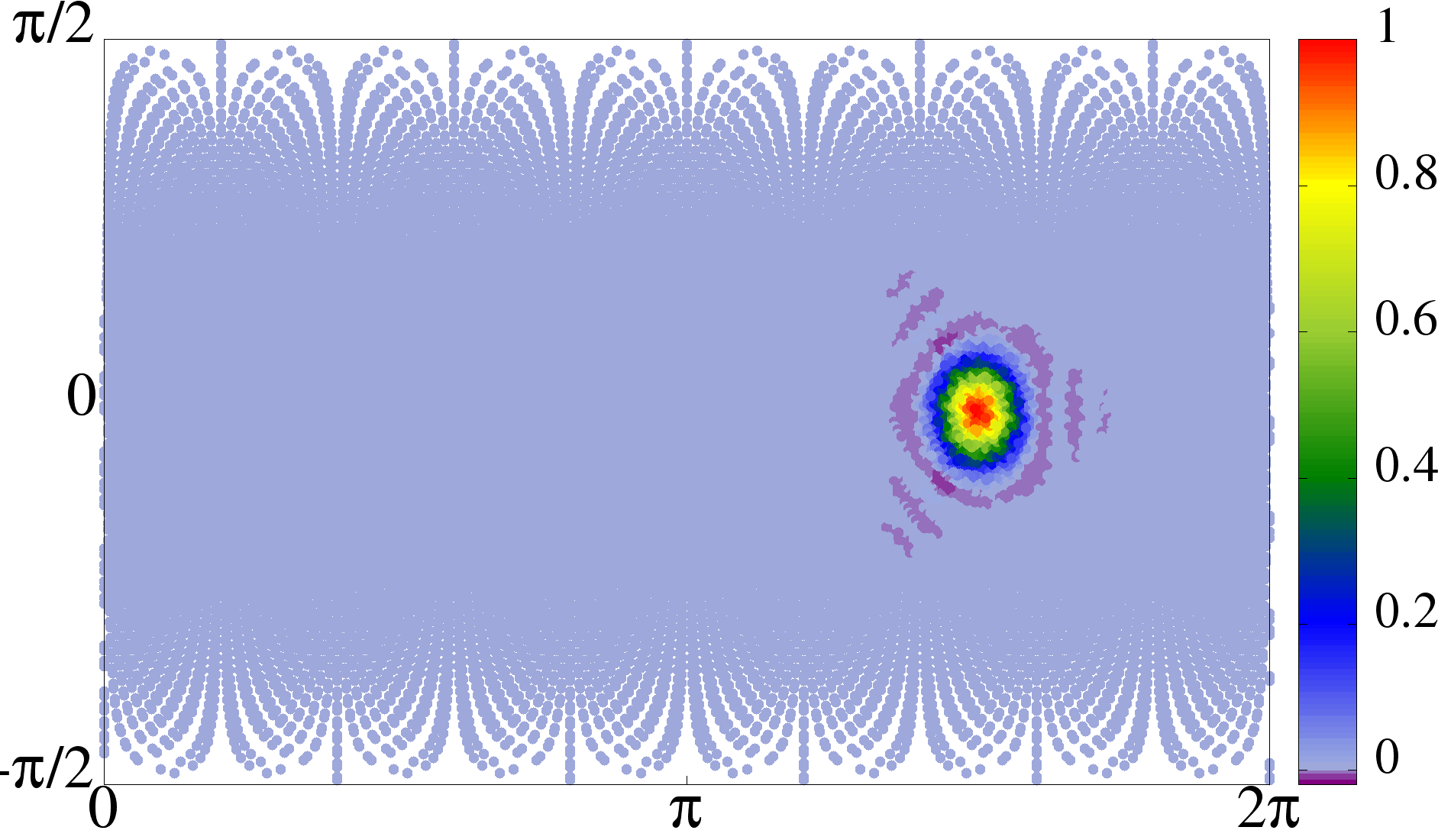}
  \caption{}
  \label{fig:PACosineAdjoint}
\end{subfigure}%
\begin{subfigure}{.5\textwidth}
  \centering
  \includegraphics[width=\textwidth]{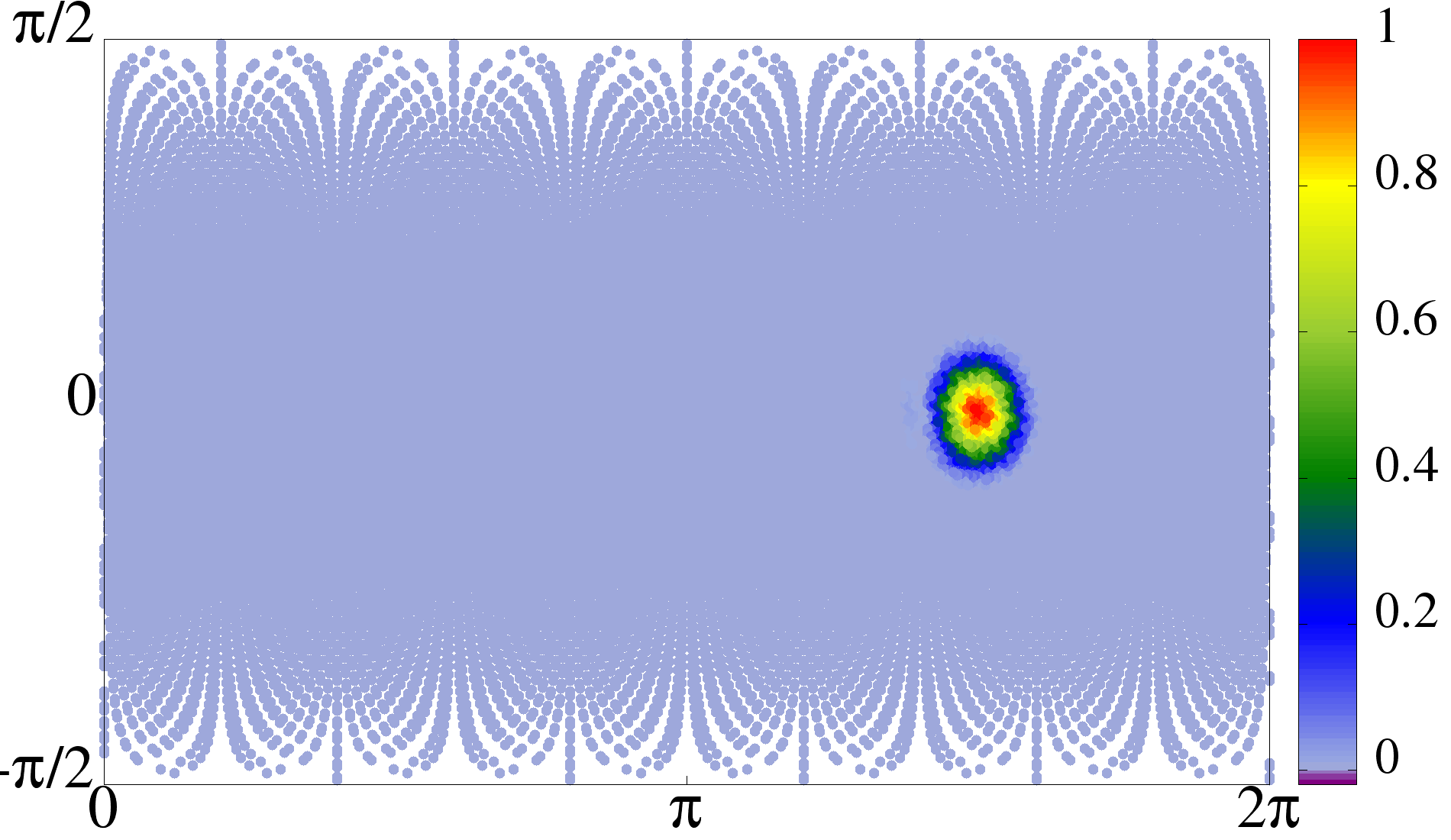}
  \caption{}
  \label{fig:PACosineSource}
\end{subfigure}
\caption{Solid body rotation, cosine bell, a)-d) - contour plots with color bar: 
                                           a) exact solution 
                                           b) ICON-FFSL, limiter ~\cite{zalesak1979fully}, \cite{harris2011flux}
                                           c) standard adjoint
                                           d) art. source adjoint, limiter ~\cite{zalesak1979fully}, \cite{harris2011flux}
                                           }
\label{fig:CosineBellContours}
\end{figure}

\begin{figure}[h]
\begin{subfigure}{.5\textwidth}
  \centering
  \includegraphics[width=\textwidth]{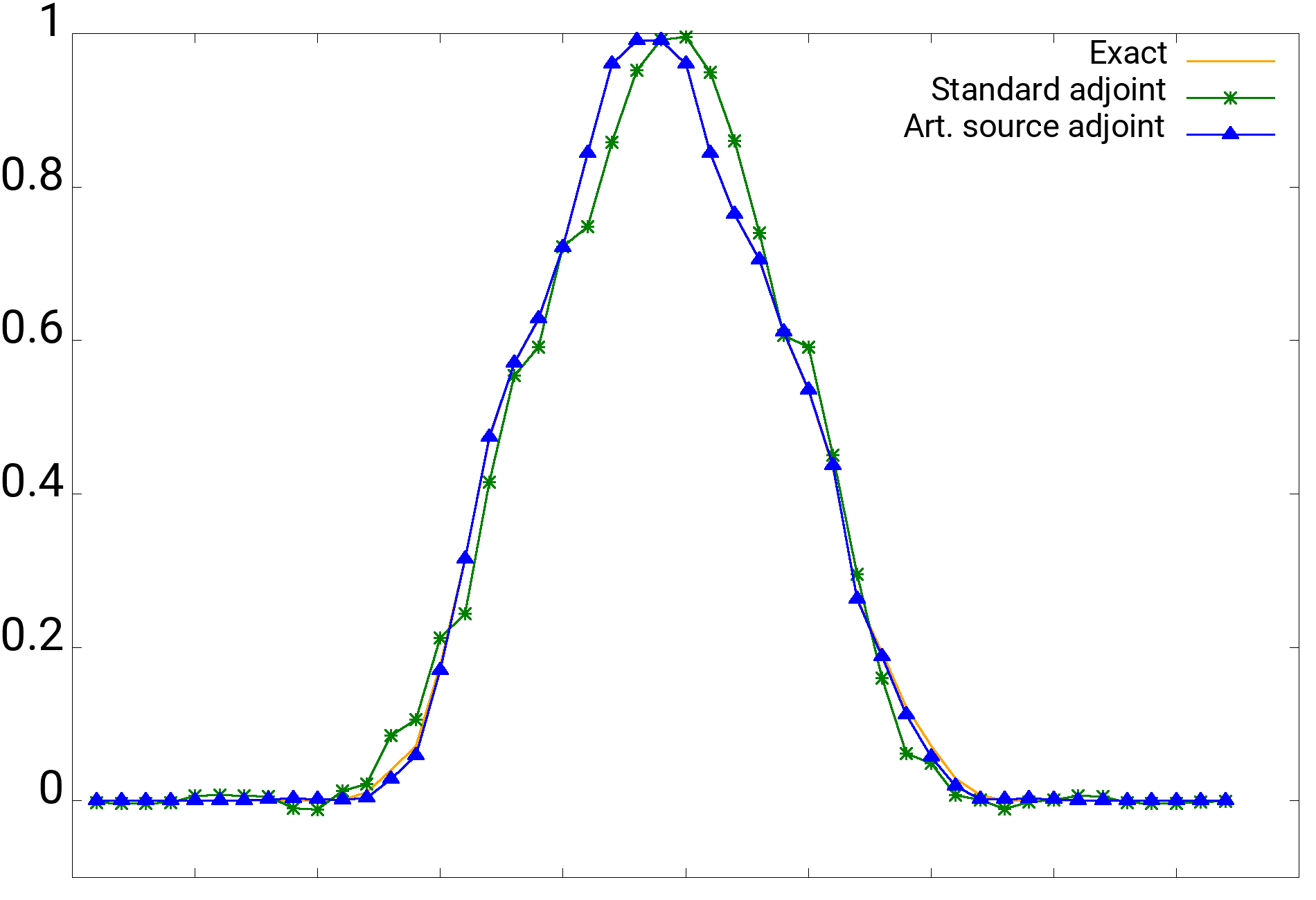}
  \caption{}
  \label{fig:PACosineAdjointSourceLim3}
\end{subfigure}%
\begin{subfigure}{.5\textwidth}
  \centering
  \includegraphics[width=\textwidth]{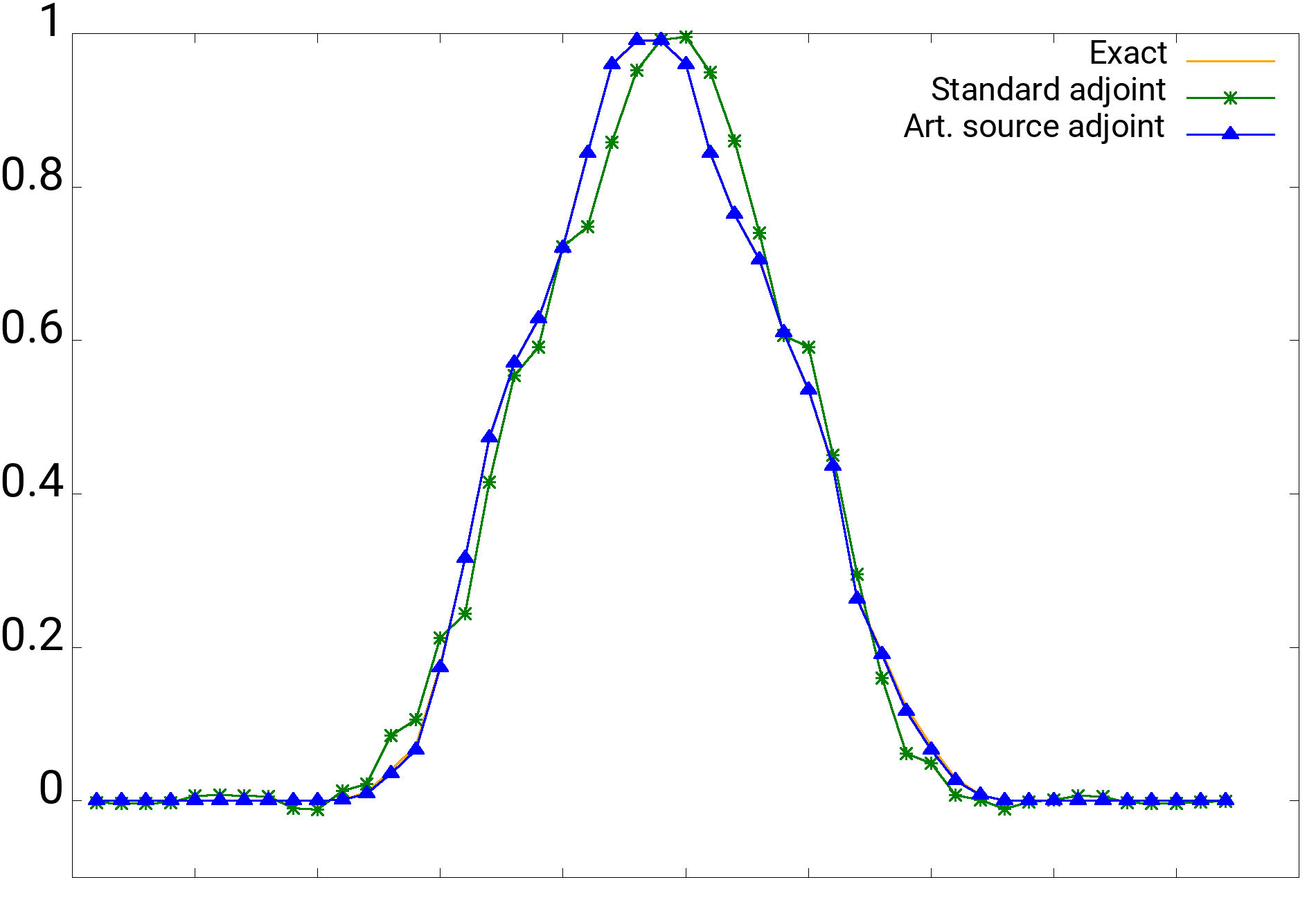}
  \caption{}
  \label{fig:PACosineAdjointSourceLim4}
\end{subfigure}
\begin{subfigure}{.5\textwidth}
  \centering
  \includegraphics[width=\textwidth]{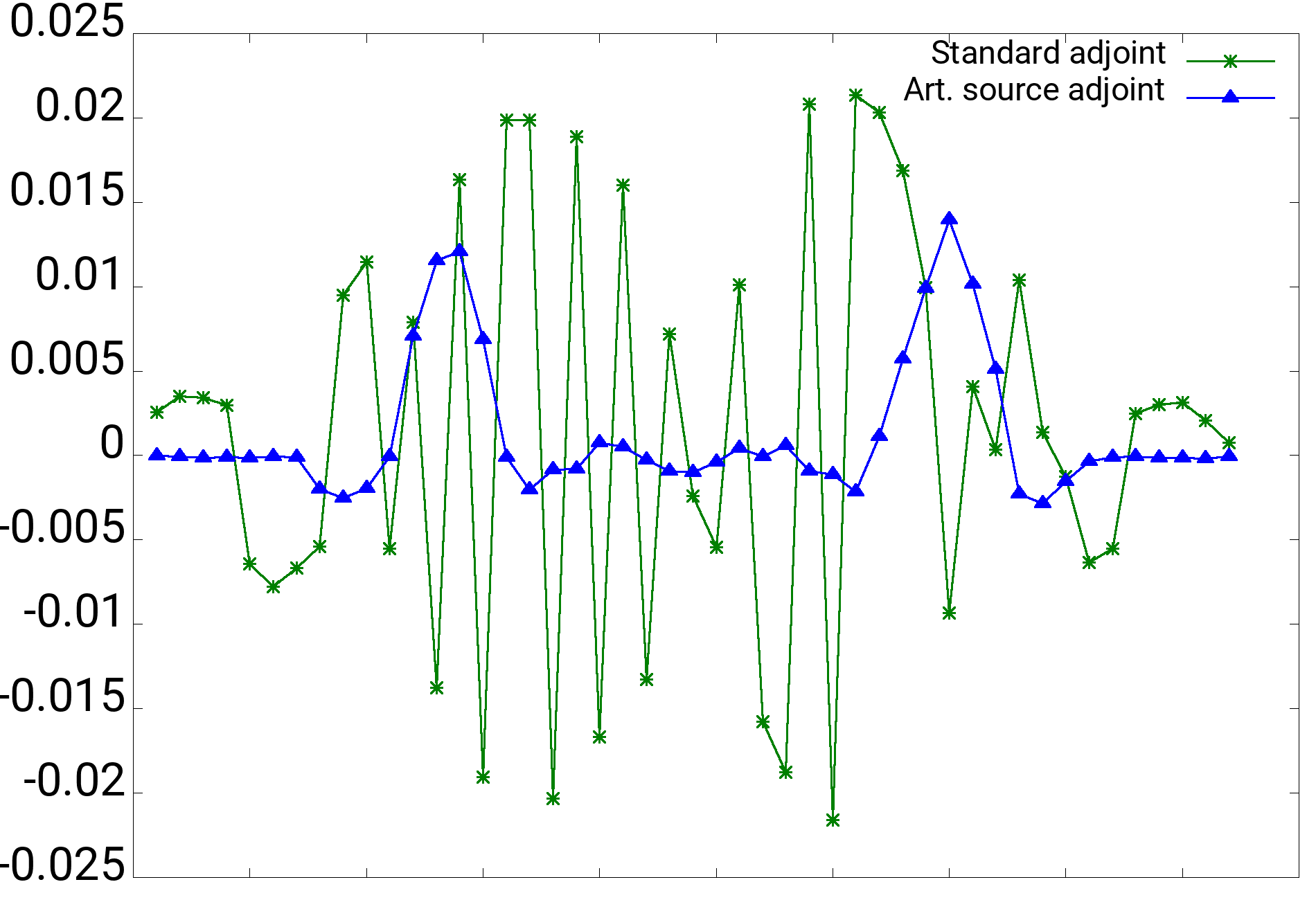}
  \caption{}
  \label{fig:PACosineAdjointSourceDiffLim3}
\end{subfigure}%
\begin{subfigure}{.5\textwidth}
  \centering
  \includegraphics[width=\textwidth]{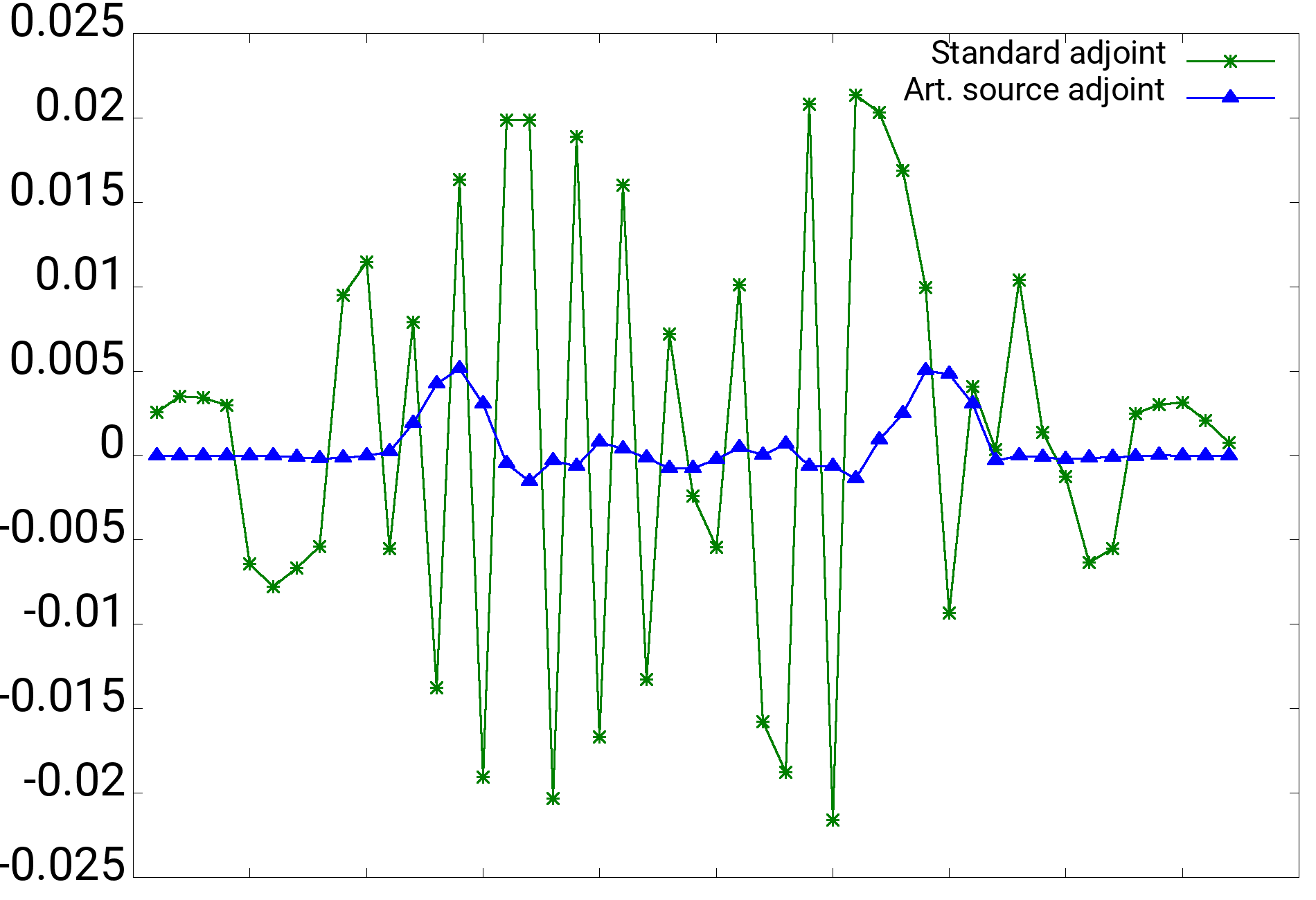}
  \caption{}
  \label{fig:PACosineAdjointSourceDiffLim4}
\end{subfigure}
\caption{Solid body rotation, cosine bell, a)-d) - along the curve exact vs standard adjoint vs art. source adjoint :
                                           a) solutions, art. source with limiter ~\cite{zalesak1979fully}, \cite{schar1996synchronous} 
                                           b) solutions, art. source with limiter ~\cite{zalesak1979fully}, \cite{harris2011flux} 
                                           c) errors, art. source with limiter ~\cite{zalesak1979fully}, \cite{schar1996synchronous}  
                                           d) errors, art. source with limiter ~\cite{zalesak1979fully}, \cite{harris2011flux}   
                                           }                                            
\label{fig:CosineBellCurve}                                    
\end{figure}

\begin{figure}
\begin{subfigure}{0.5\textwidth}
  \centering
  \includegraphics[width=\textwidth]{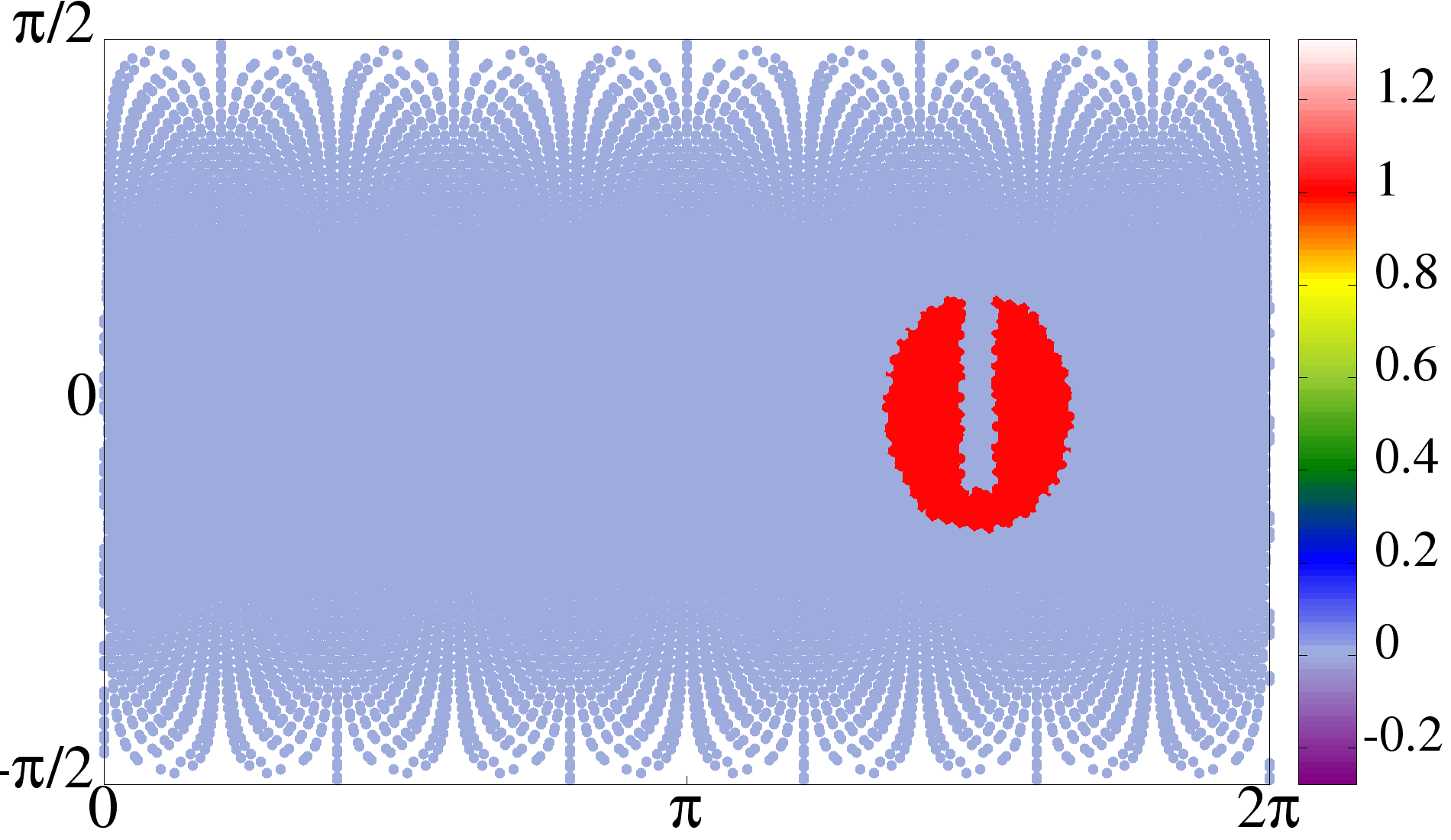}
  \caption{}
  \label{fig:PACylinderExact}
\end{subfigure}%
\begin{subfigure}{0.5\textwidth}
  \centering
  \includegraphics[width=\textwidth]{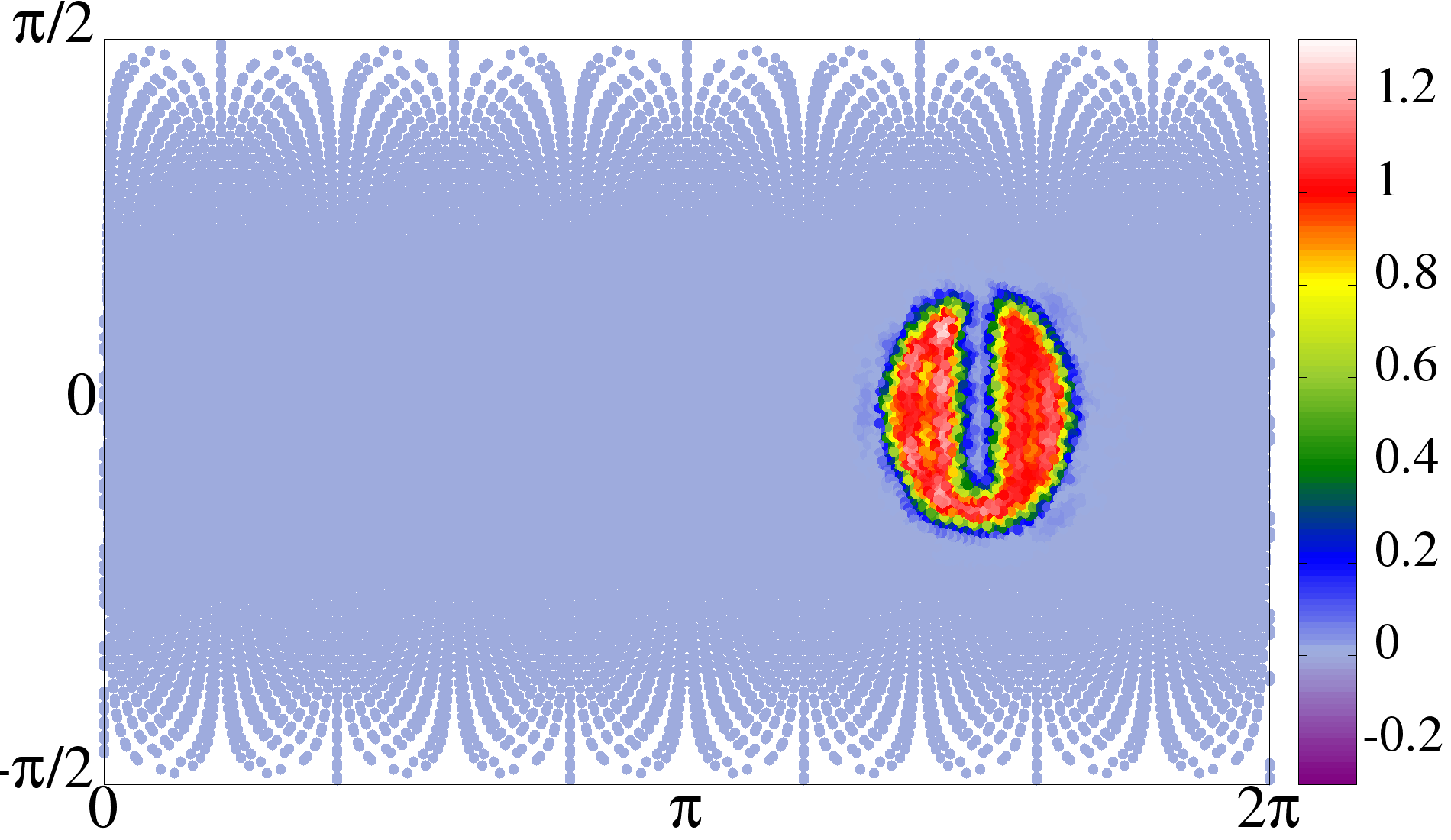}
  \caption{}
  \label{fig:PACylinderIcon}
\end{subfigure}
\begin{subfigure}{0.5\textwidth}
  \centering
  \includegraphics[width=\textwidth]{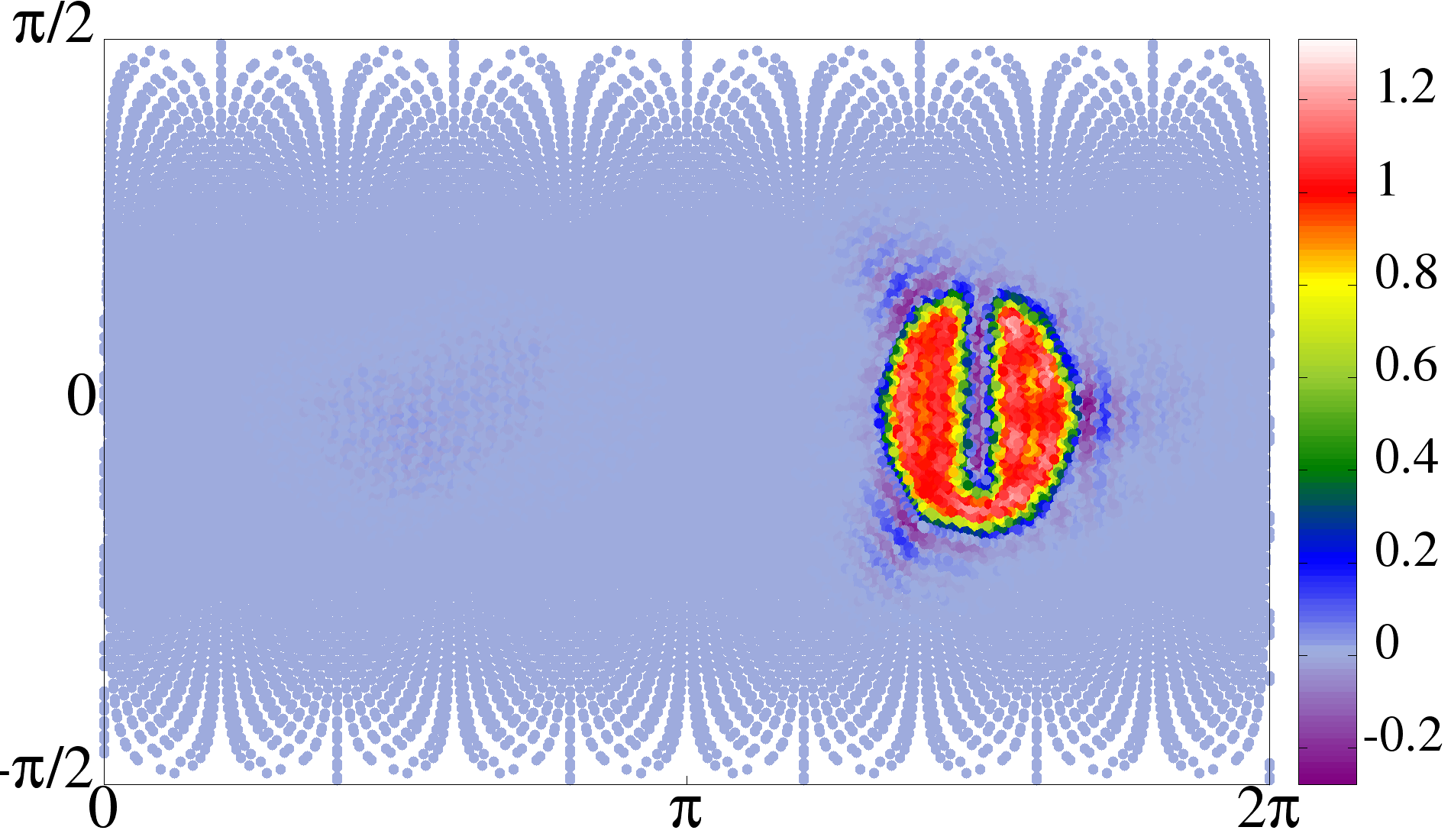}
  \caption{}
  \label{fig:PACylinderAdjoint}
\end{subfigure}%
\begin{subfigure}{.5\textwidth}
  \centering
  \includegraphics[width=\textwidth]{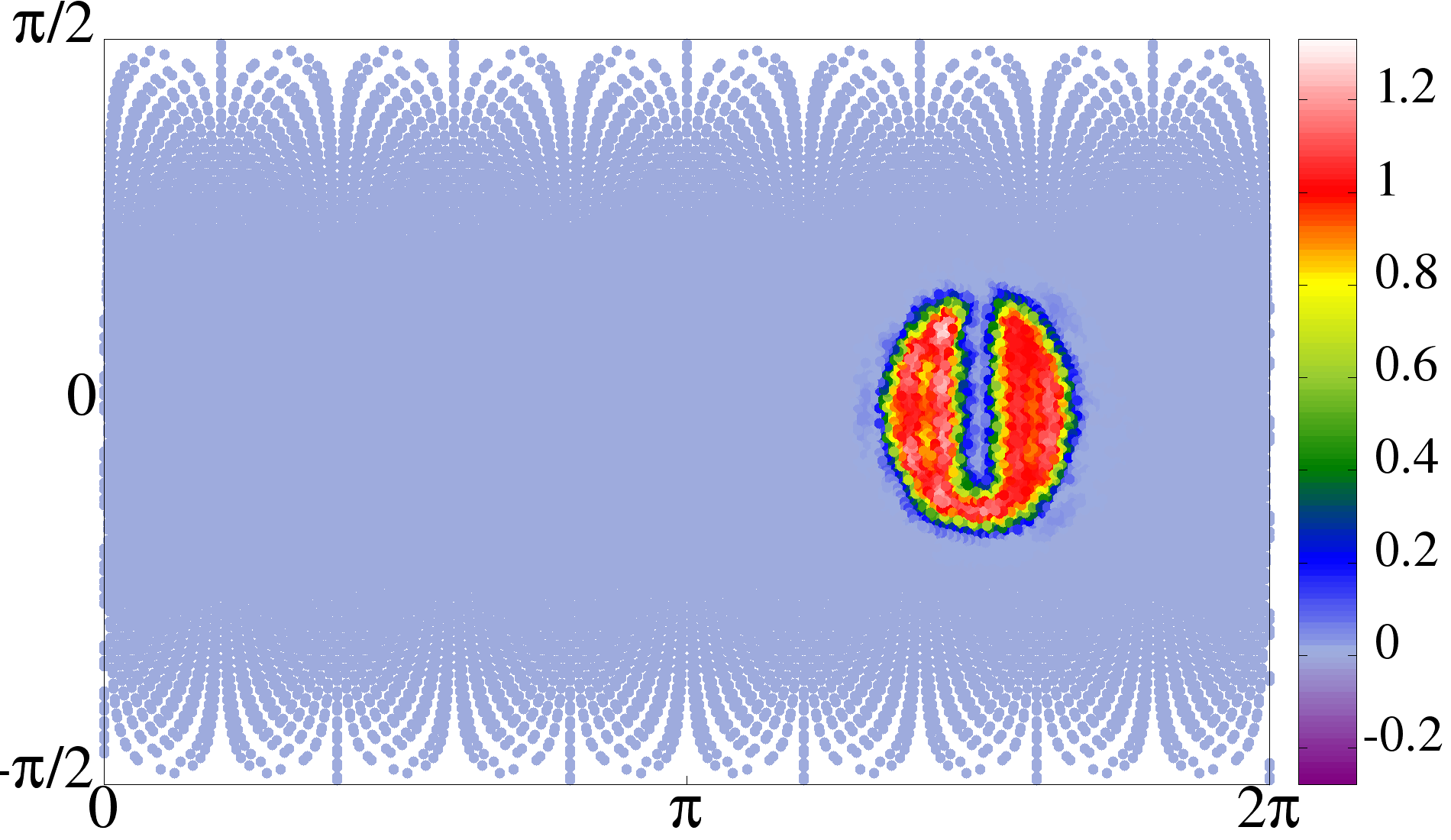}
  \caption{}
  \label{fig:PACylinderSource}
\end{subfigure}
\caption{Solid body rotation, slotted cylinder, a)-d) - contour plots with color bar: 
                                           a) exact solution 
                                           b) ICON-FFSL, limiter ~\cite{zalesak1979fully}, \cite{schar1996synchronous} 
                                           c) standard adjoint
                                           d) art. source adjoint, limiter ~\cite{zalesak1979fully}, \cite{schar1996synchronous}    
                                           }
\label{fig:SlottedCylinderContours}
\end{figure}

\begin{figure}[h]
\begin{subfigure}{.5\textwidth}
  \centering
  \includegraphics[width=\textwidth]{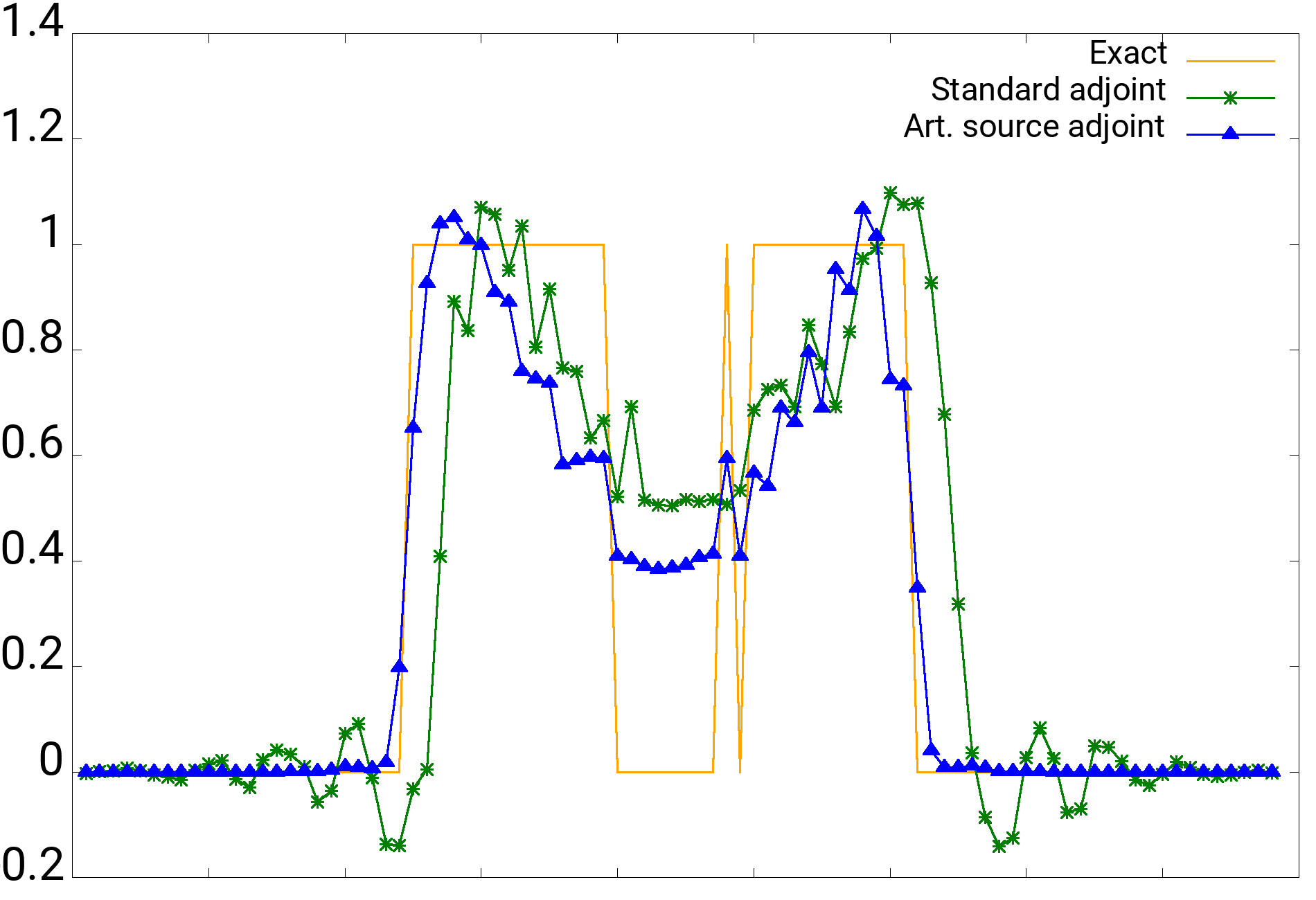}
  \caption{}
  \label{fig:PACylinderAdjointSourceLim3}
\end{subfigure}%
\begin{subfigure}{.5\textwidth}
  \centering
  \includegraphics[width=\textwidth]{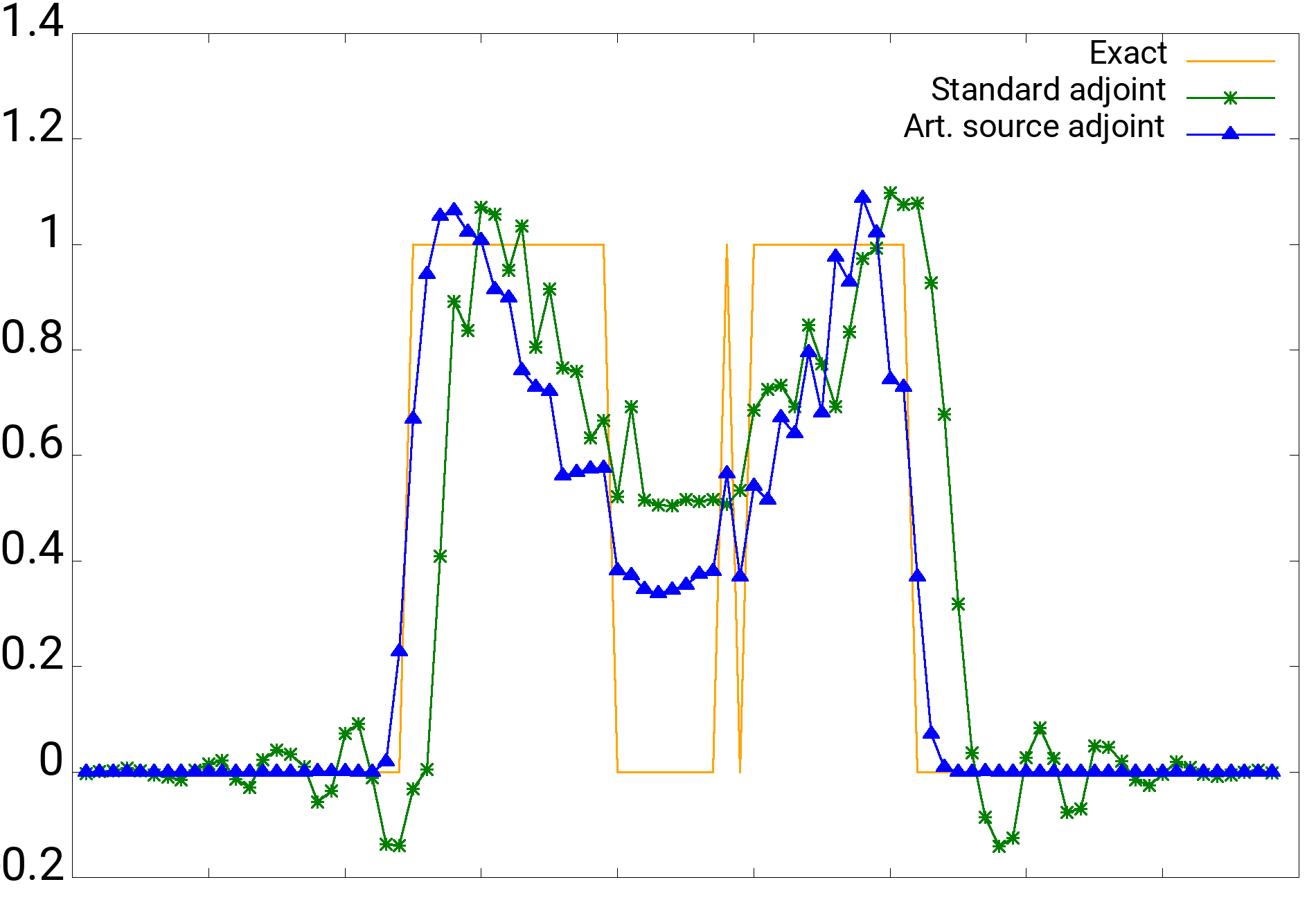}
  \caption{}
  \label{fig:PACylinderAdjointSourceLim4}
\end{subfigure}
\begin{subfigure}{.5\textwidth}
  \centering
  \includegraphics[width=\textwidth]{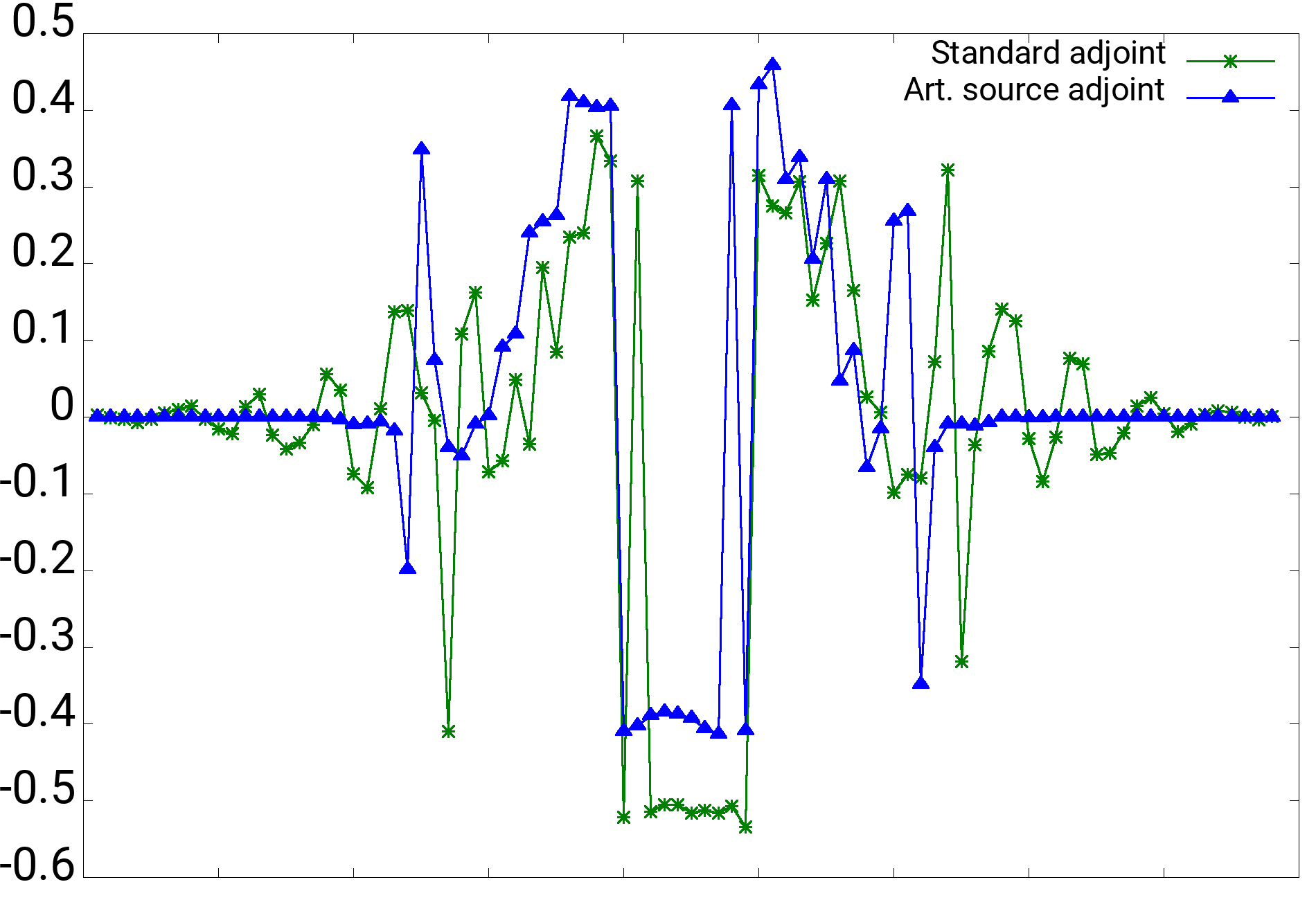}
  \caption{}
  \label{fig:PACylinderAdjointSourceDiffLim3}
\end{subfigure}%
\begin{subfigure}{.5\textwidth}
  \centering
  \includegraphics[width=\textwidth]{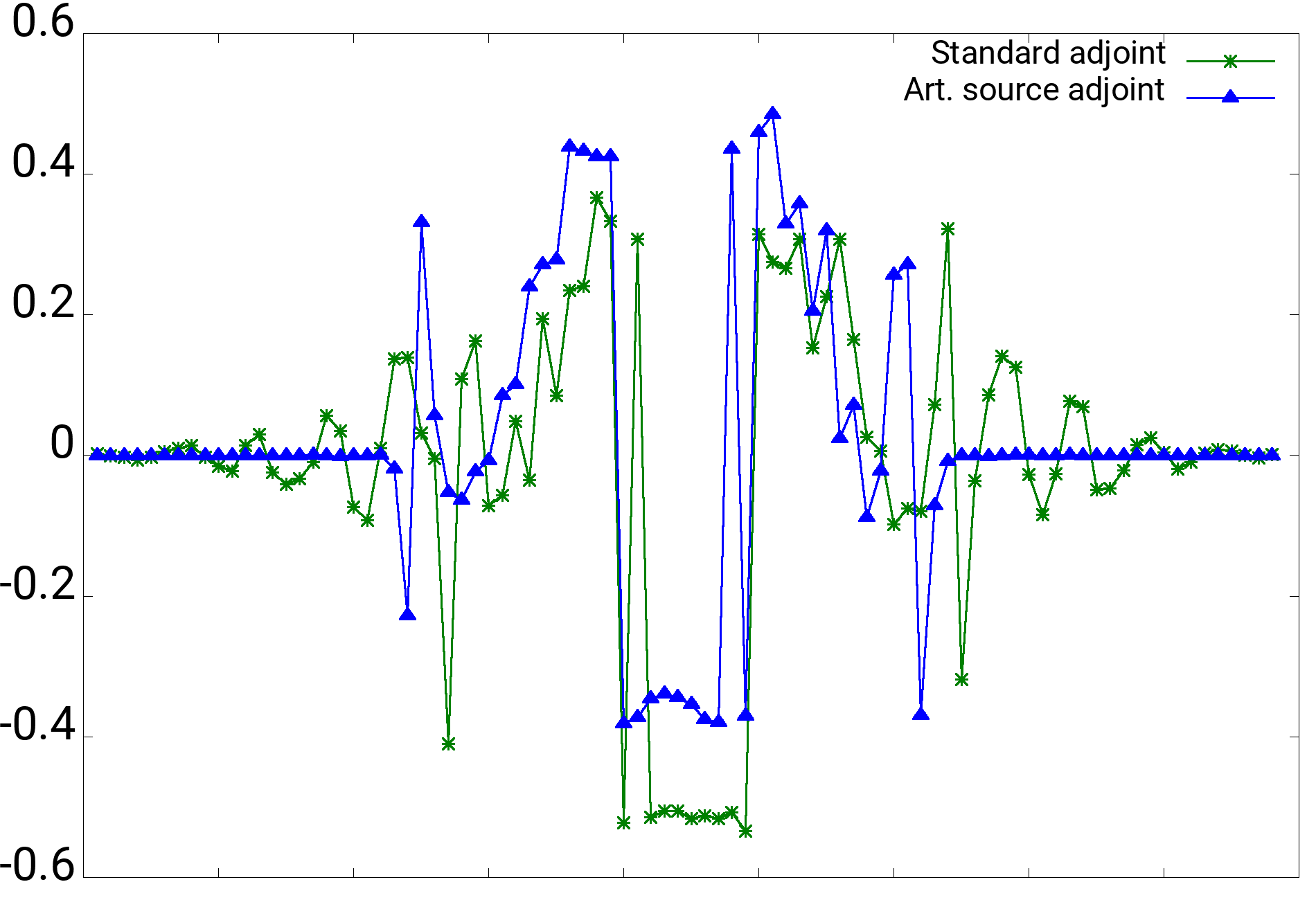}
  \caption{}
  \label{fig:PACylinderAdjointSourceDiffLim4}
\end{subfigure}
\caption{Solid body rotation, slotted cylinder, a)-d) - along the curve exact vs standard adjoint vs art. source adjoint:
                                                 a) solutions, art. source with limiter ~\cite{zalesak1979fully}, \cite{schar1996synchronous} 
                                                 b) solutions, art. source with limiter ~\cite{zalesak1979fully}, \cite{harris2011flux} 
                                                 c) errors, art. source with limiter ~ \cite{zalesak1979fully}, \cite{schar1996synchronous}  
                                                 d) errors, art. source with limiter ~ \cite{zalesak1979fully}, \cite{harris2011flux}
                                                 } 
\label{fig:SlottedCylinderCurve}
\end{figure}
%
\begin{figure}
\begin{subfigure}{.5\textwidth}
  \centering
  \includegraphics[width=\textwidth]{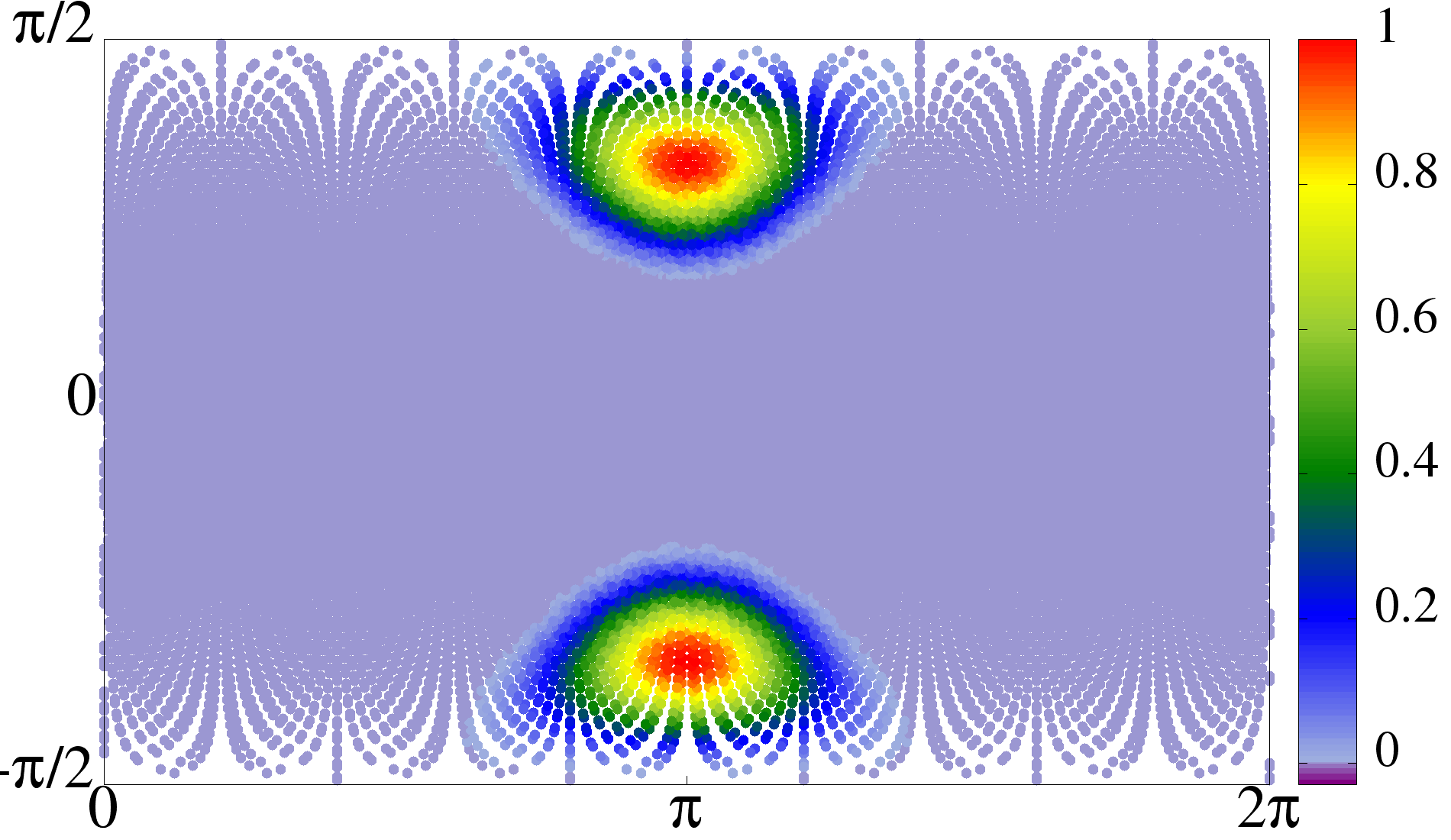}
  \caption{}
  \label{fig:DF1CosineExact}
\end{subfigure}%
\begin{subfigure}{.5\textwidth}
  \centering
  \includegraphics[width=\textwidth]{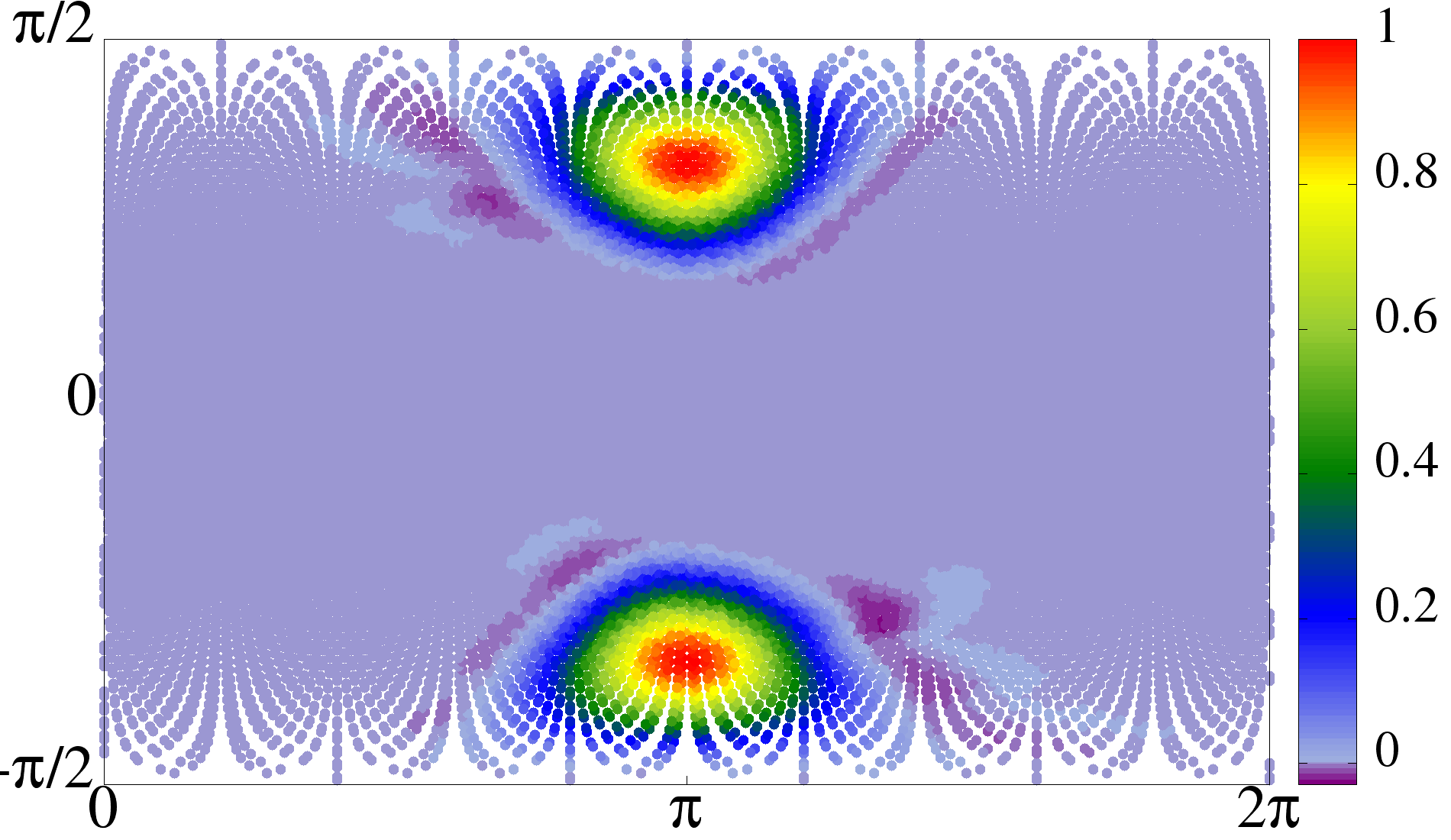}
  \caption{}
  \label{fig:DF1CosineAdjoint}
\end{subfigure}
\begin{subfigure}{.5\textwidth}
  \centering
  \includegraphics[width=\textwidth]{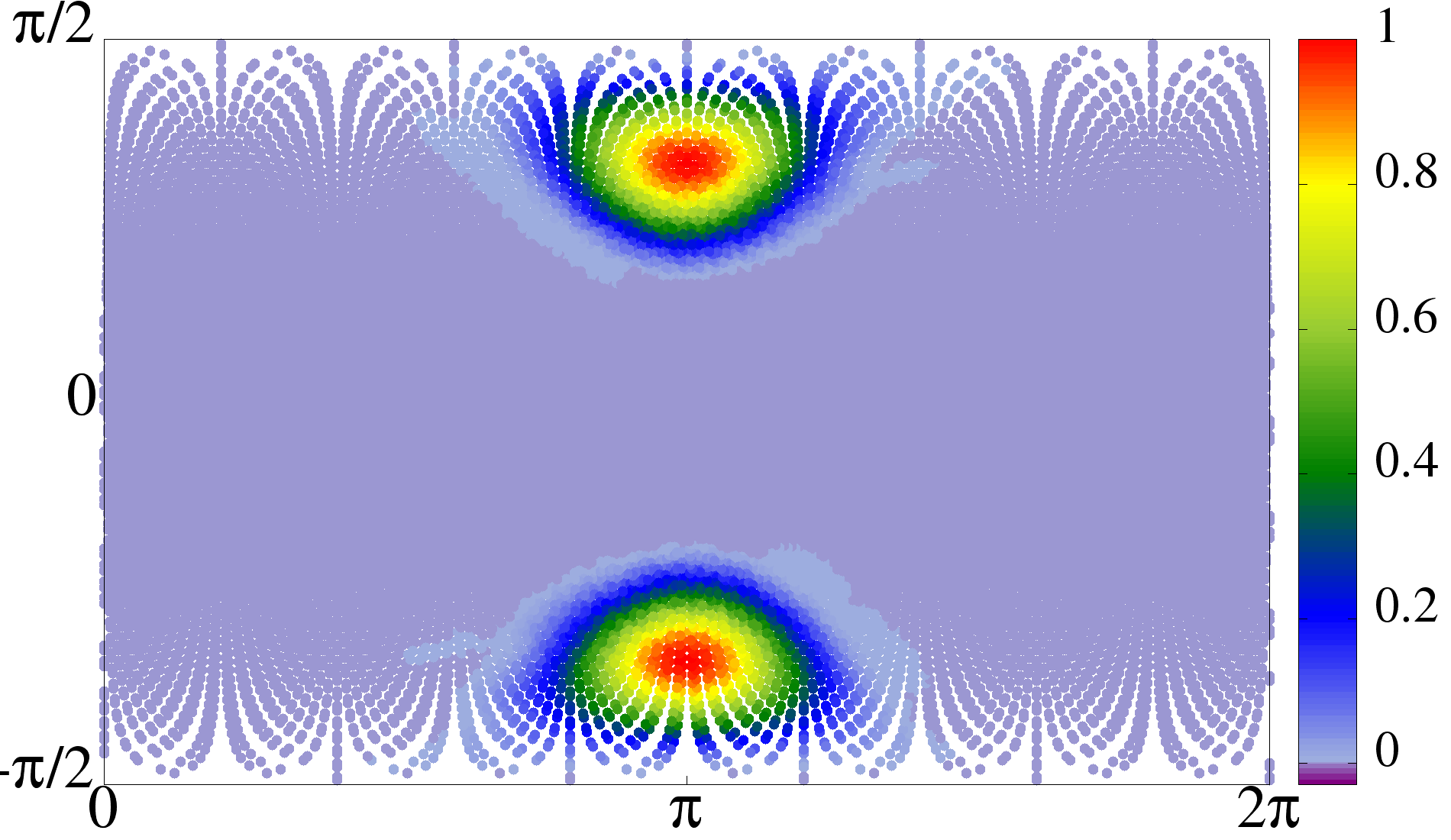}
  \caption{}
  \label{fig:DF1CosineSourceLim3}
\end{subfigure}%
\begin{subfigure}{.5\textwidth}
  \centering
  \includegraphics[width=\textwidth]{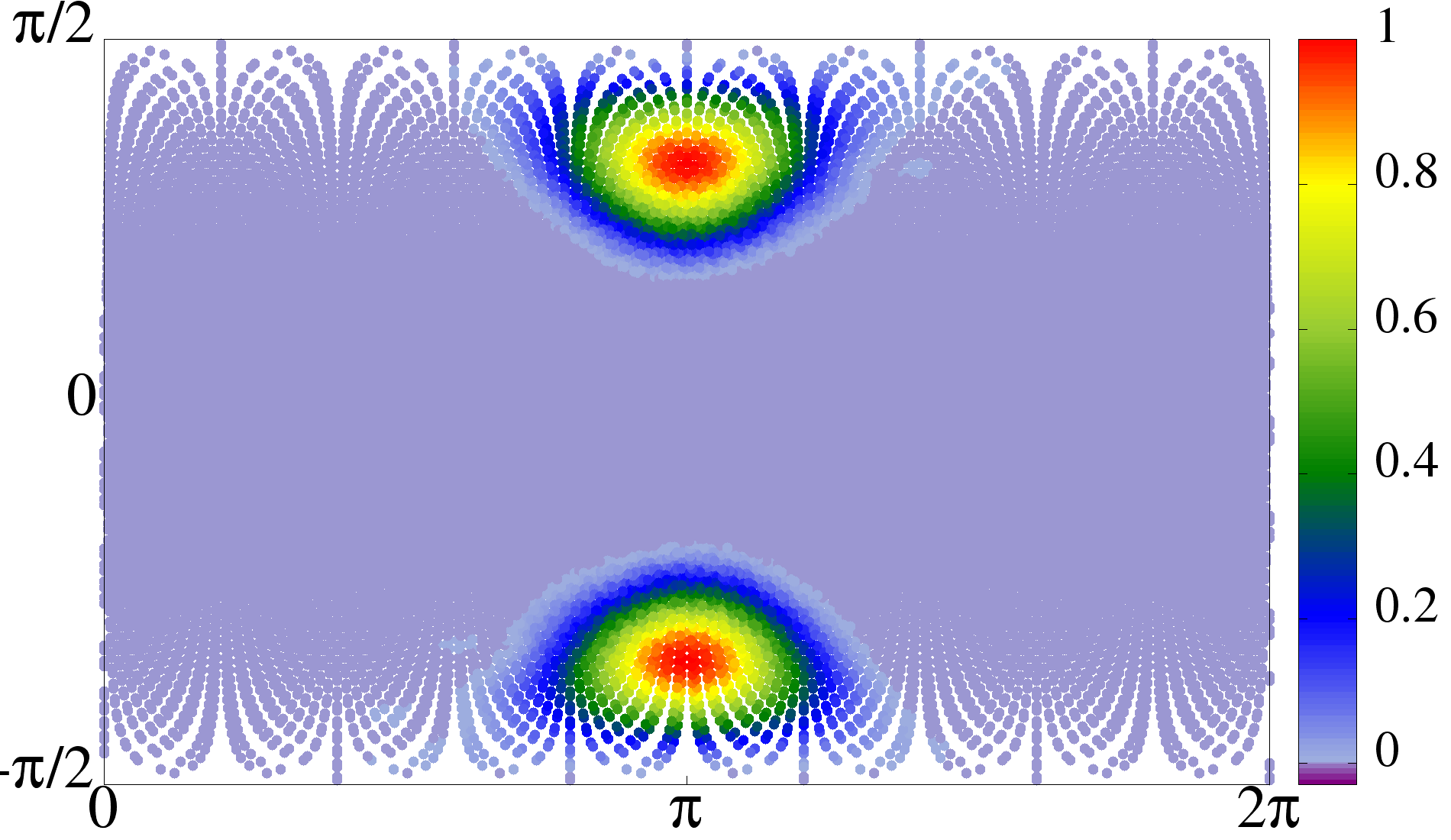} 
  \caption{}
  \label{fig:DF1CosineSourceLim4}
\end{subfigure}
\caption{Deformational flow, $\nabla \vec{v}=0$, cosine bells, a)-d) - contour plots with color bar: 
                                          a) exact solution
                                          b) standard adjoint, 
                                          c) art. source adjoint, limiter ~ \cite{zalesak1979fully}, \cite{schar1996synchronous}
                                          d) art. source adjoint, limiter ~\cite{zalesak1979fully}, \cite{harris2011flux}
                                          }
\label{fig:DefFl1:2CosBellContour}
\end{figure}
%

\begin{figure}[h]
\begin{subfigure}{.5\textwidth}
  \centering
  \includegraphics[width=\textwidth]{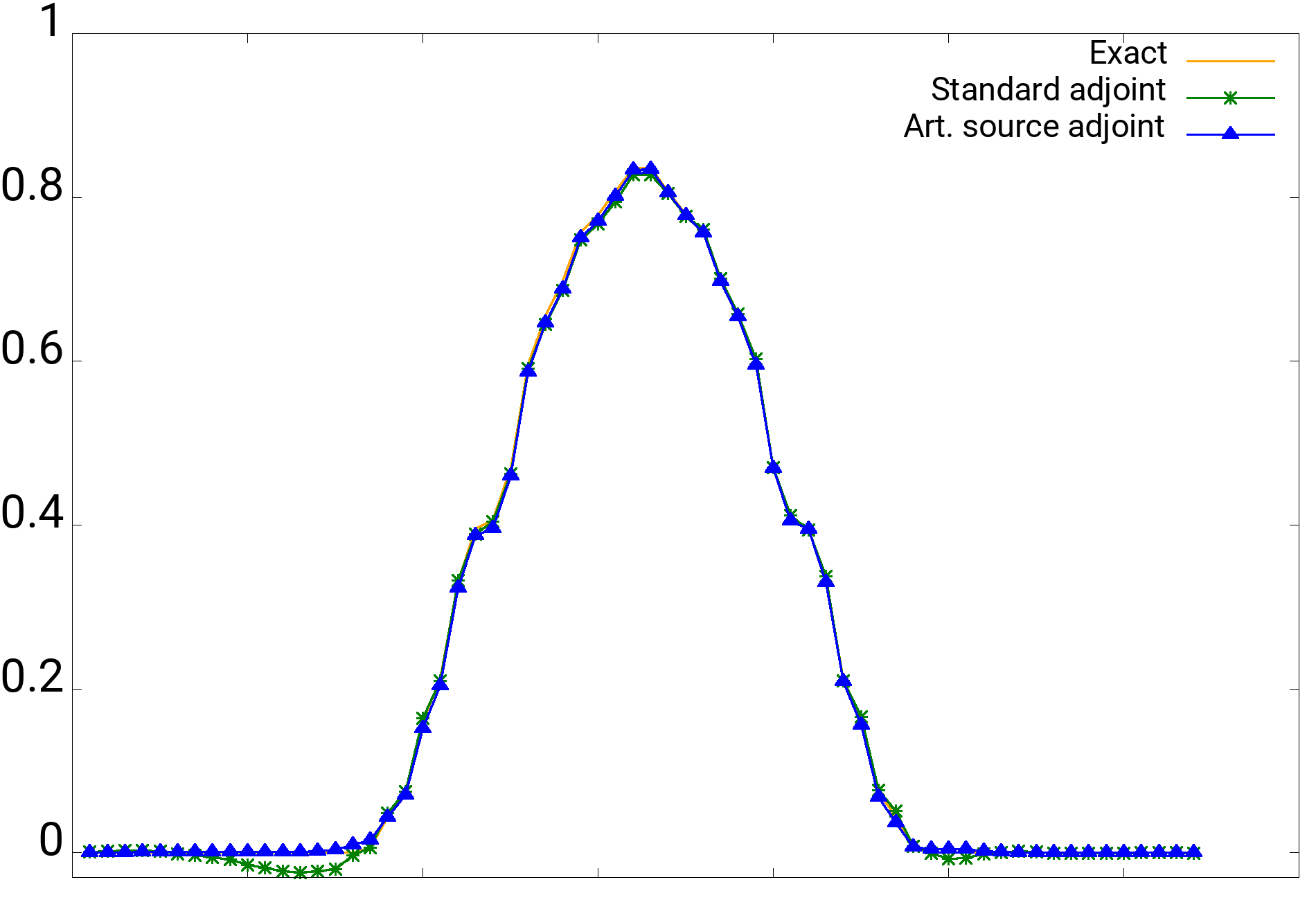}
  \caption{}
  \label{fig:DF1CosineAdjointSourceLim3}
\end{subfigure}%
\begin{subfigure}{.5\textwidth}
  \centering
  \includegraphics[width=\textwidth]{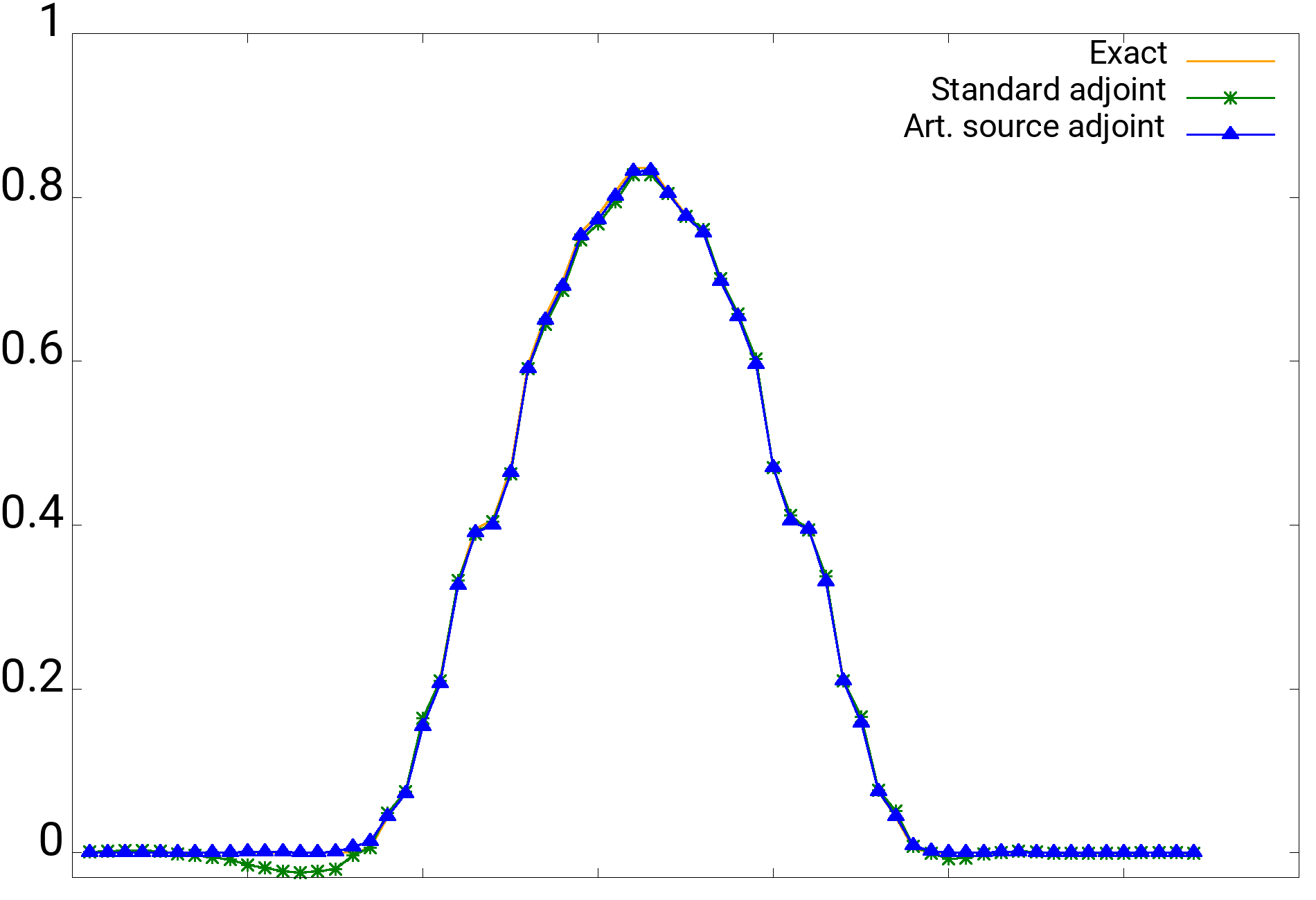}
  \caption{}
  \label{fig:DF1CosineAdjointSourceLim4}
\end{subfigure}
\begin{subfigure}{.5\textwidth}
  \centering
  \includegraphics[width=\textwidth]{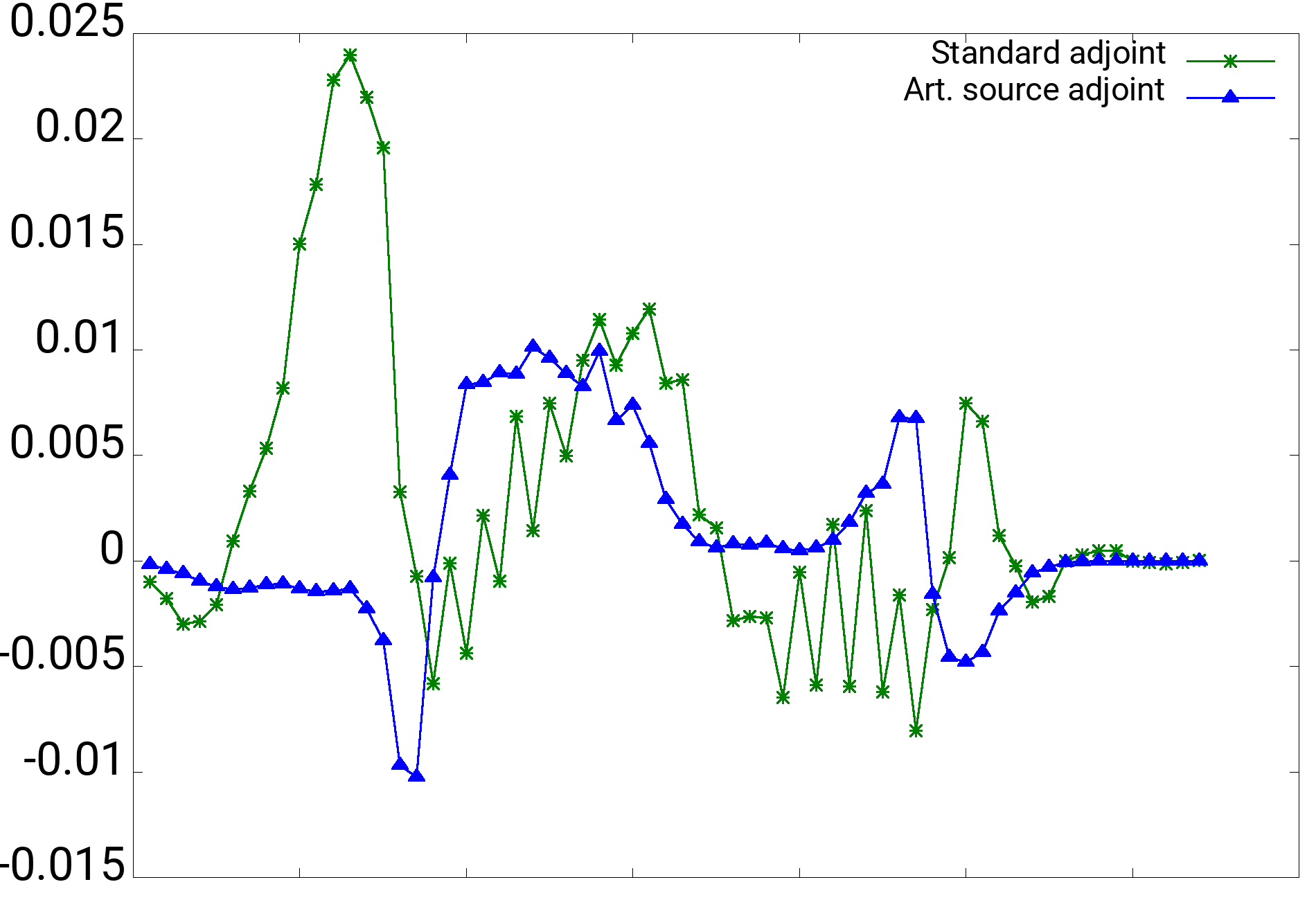}
  \caption{}
  \label{fig:DF1CosineAdjointSourceDiffLim3}
\end{subfigure}%
\begin{subfigure}{.5\textwidth}
  \centering
  \includegraphics[width=\textwidth]{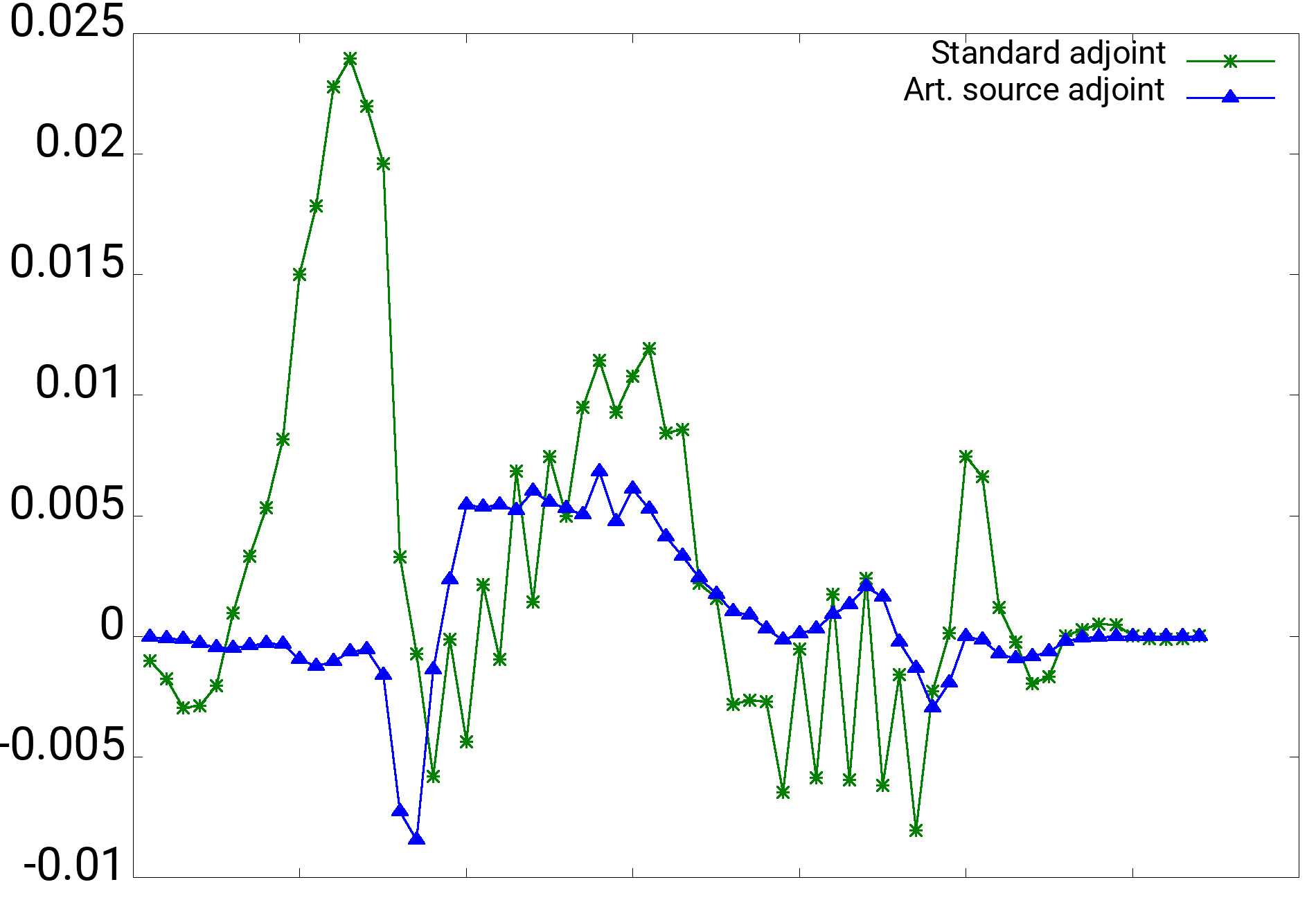}
  \caption{}
  \label{fig:DF1CosineAdjointSourceLim4}
\end{subfigure}
\caption{Deformational flow, $\nabla \vec{v}=0$, cosine bells, a)-d) - along the curve  exact vs standard adjoint vs art. source adjoint:
                                                                       a) solutions, art. source with limiter ~ \cite{zalesak1979fully}, \cite{schar1996synchronous} 
                                                                       b) solutions, art. source with limiter ~ \cite{zalesak1979fully}, \cite{harris2011flux} 
                                                                       c) errors, art. source with limiter ~ \cite{zalesak1979fully}, \cite{schar1996synchronous}  
                                                                       d) errors, art. source with limiter ~ \cite{zalesak1979fully}, \cite{harris2011flux}   
                                                                       } 
\label{fig:DefFl1:2CosBellCurve}
\end{figure}

\begin{figure}
\begin{subfigure}{.5\textwidth}
  \centering
  \includegraphics[width=\textwidth]{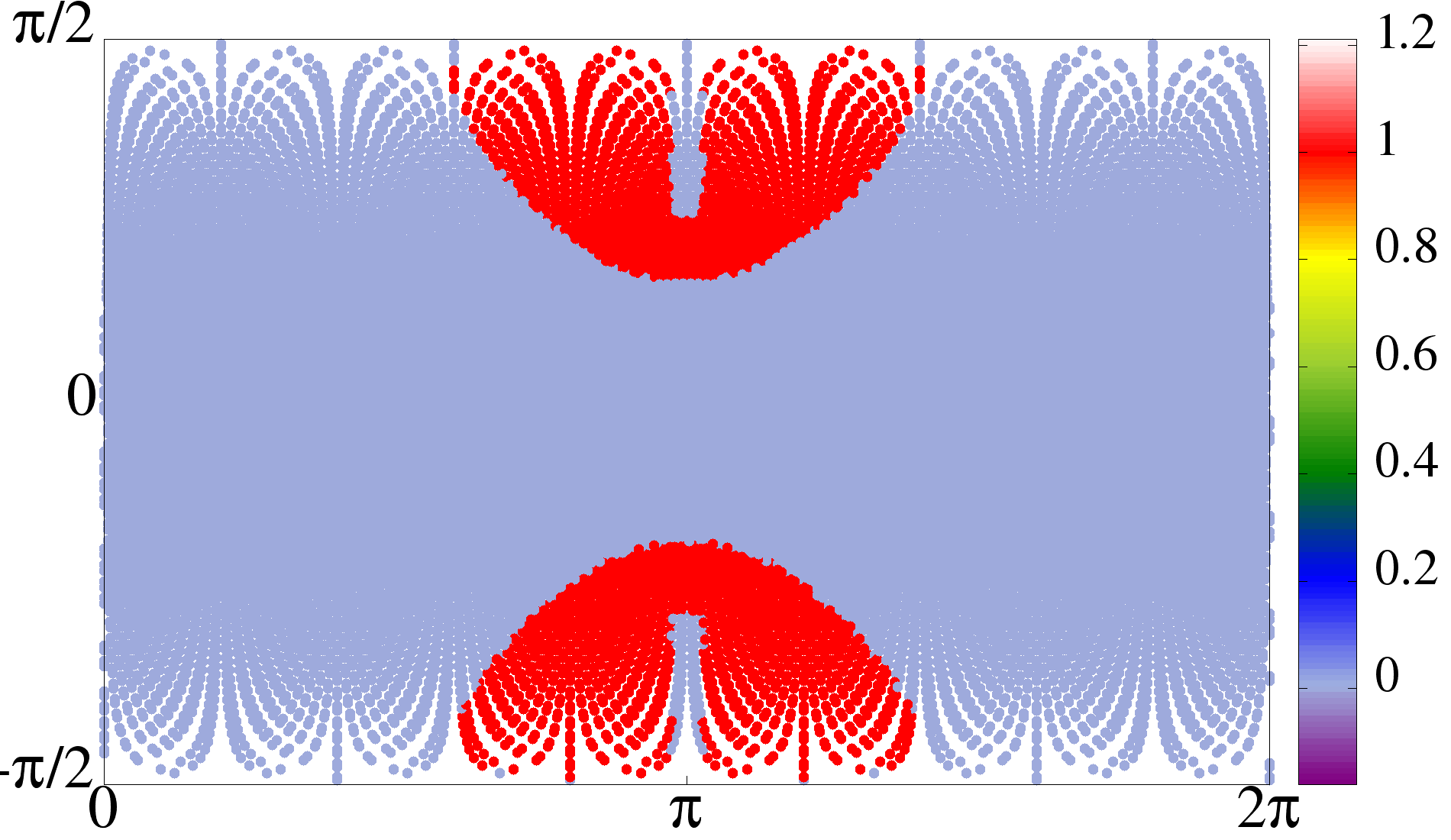}
  \caption{}
  \label{fig:DF1CylinderExact}
\end{subfigure}%
\begin{subfigure}{.5\textwidth}
  \centering
  \includegraphics[width=\textwidth]{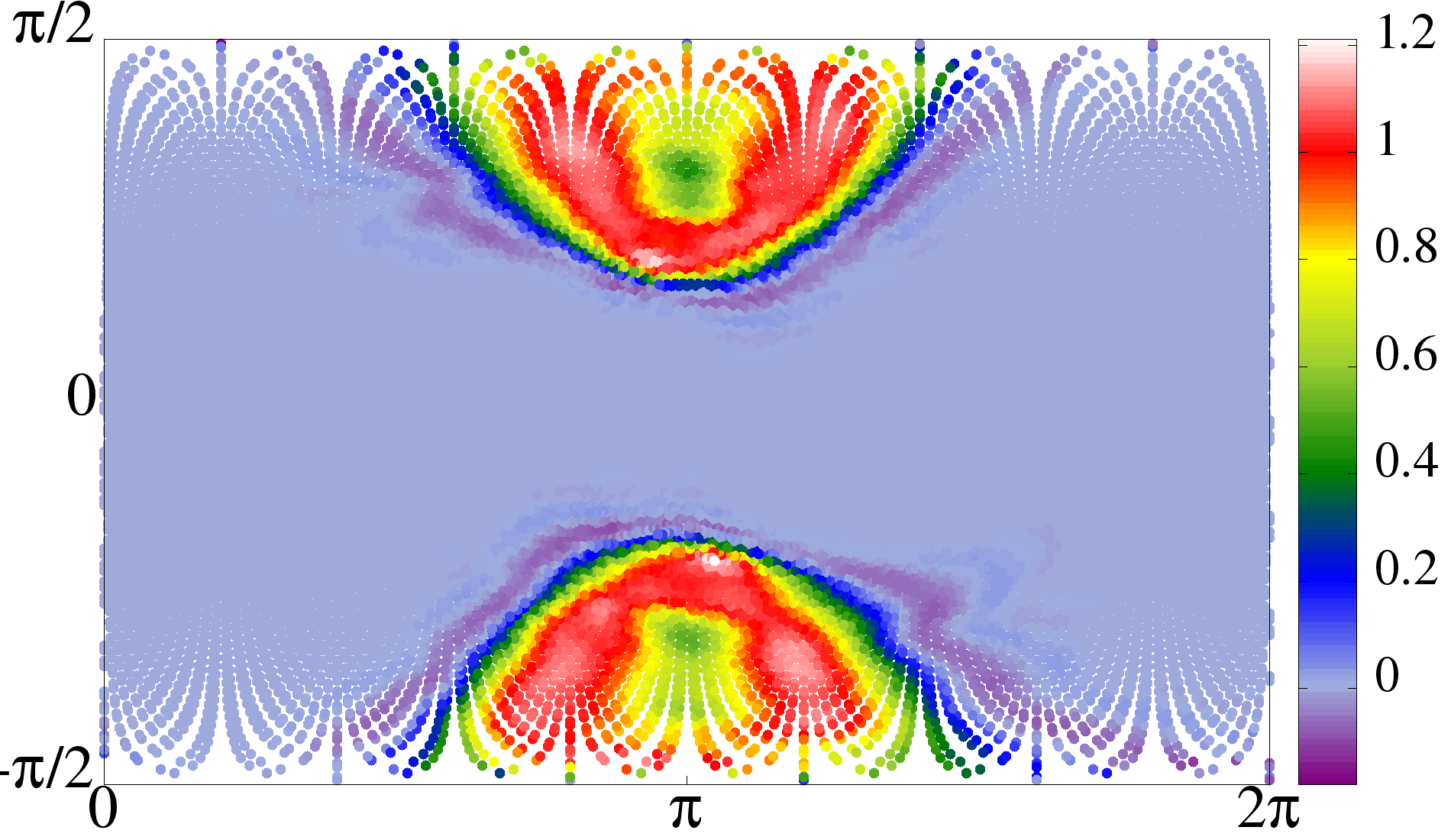}
  \caption{}
  \label{fig:DF1CylinderAdjoint}
\end{subfigure}
\begin{subfigure}{.5\textwidth}
  \centering
  \includegraphics[width=\textwidth]{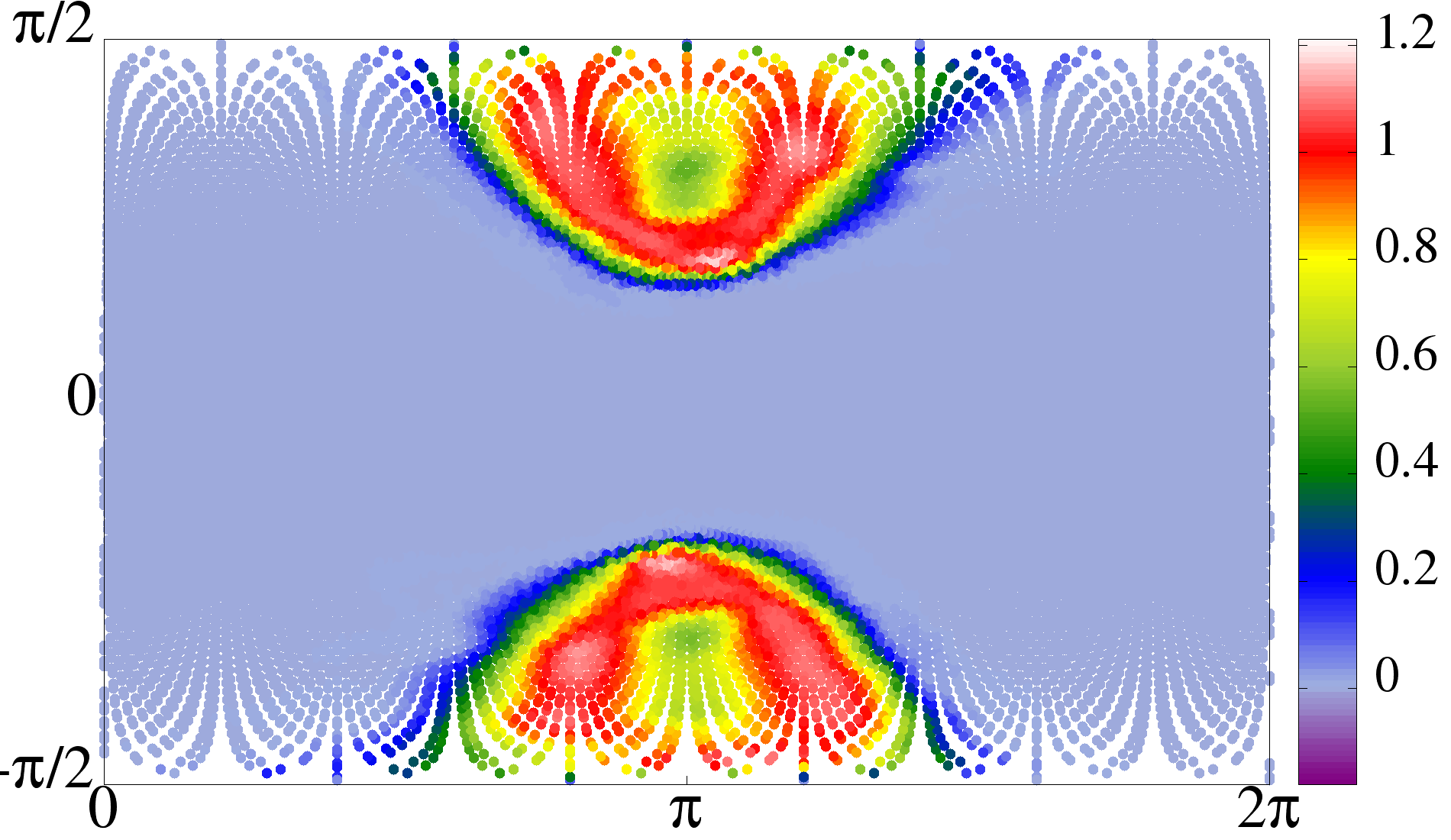}
  \caption{}
  \label{fig:DF1CylinderSourceLim3}
\end{subfigure}%
\begin{subfigure}{.5\textwidth}
  \centering
  \includegraphics[width=\textwidth]{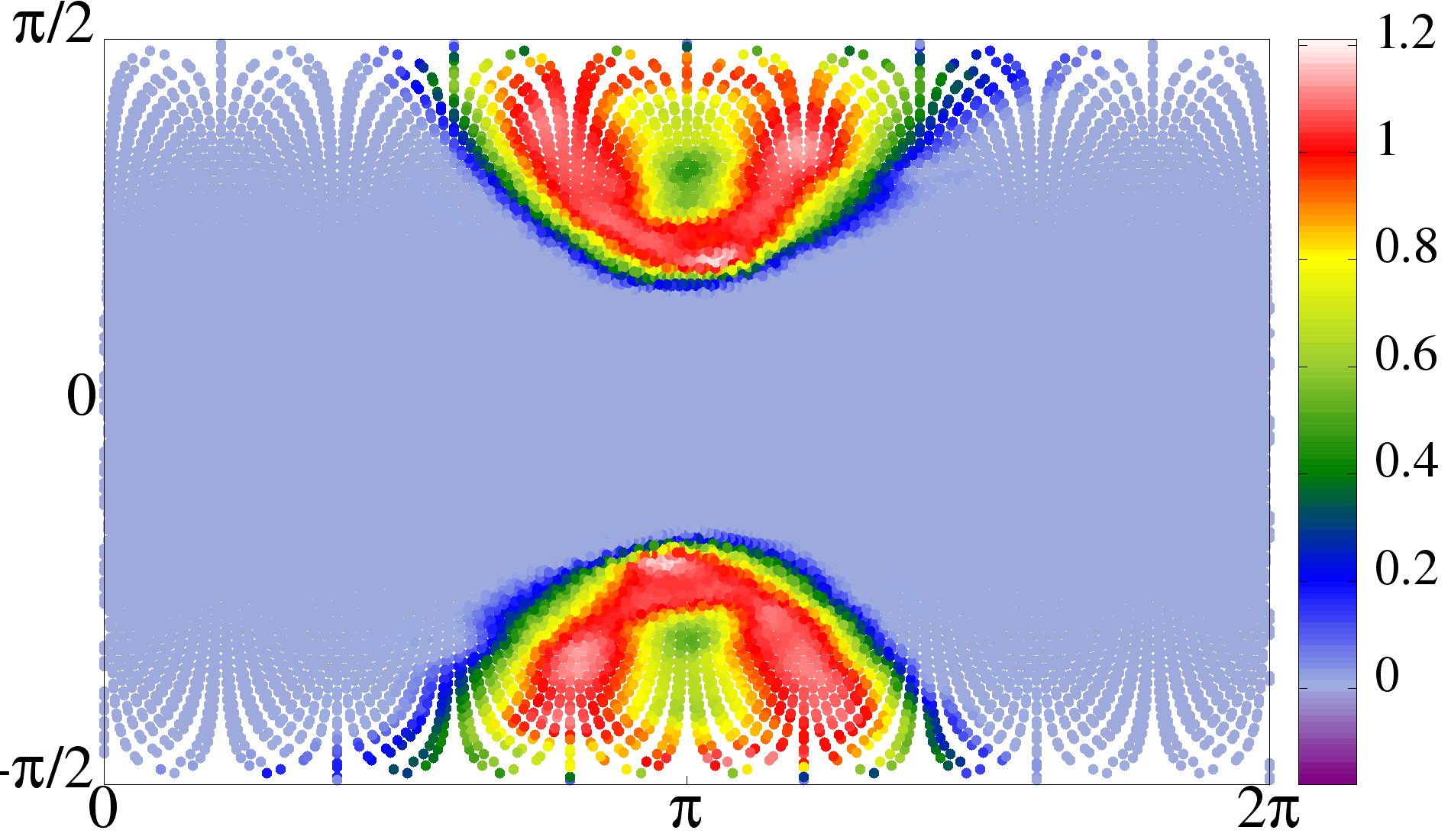}
  \caption{}
  \label{fig:DF1CylinderLim4}
\end{subfigure}
\caption{Deformational flow, $\nabla \vec{v}=0$, slotted cylinders, a)-d) - contour plots with color bar: 
                                          a) exact solution
                                          b) standard adjoint, 
                                          c) art. source adjoint, limiter ~ \cite{zalesak1979fully}, \cite{schar1996synchronous}
                                          d) art. source adjoint, limiter ~\cite{zalesak1979fully}, \cite{harris2011flux}     
                                          }
\label{fig:DefFl1:2SlottedCylinderContour}
\end{figure}

\begin{figure}[h]
\begin{subfigure}{.5\textwidth}
  \centering
  \includegraphics[width=\textwidth]{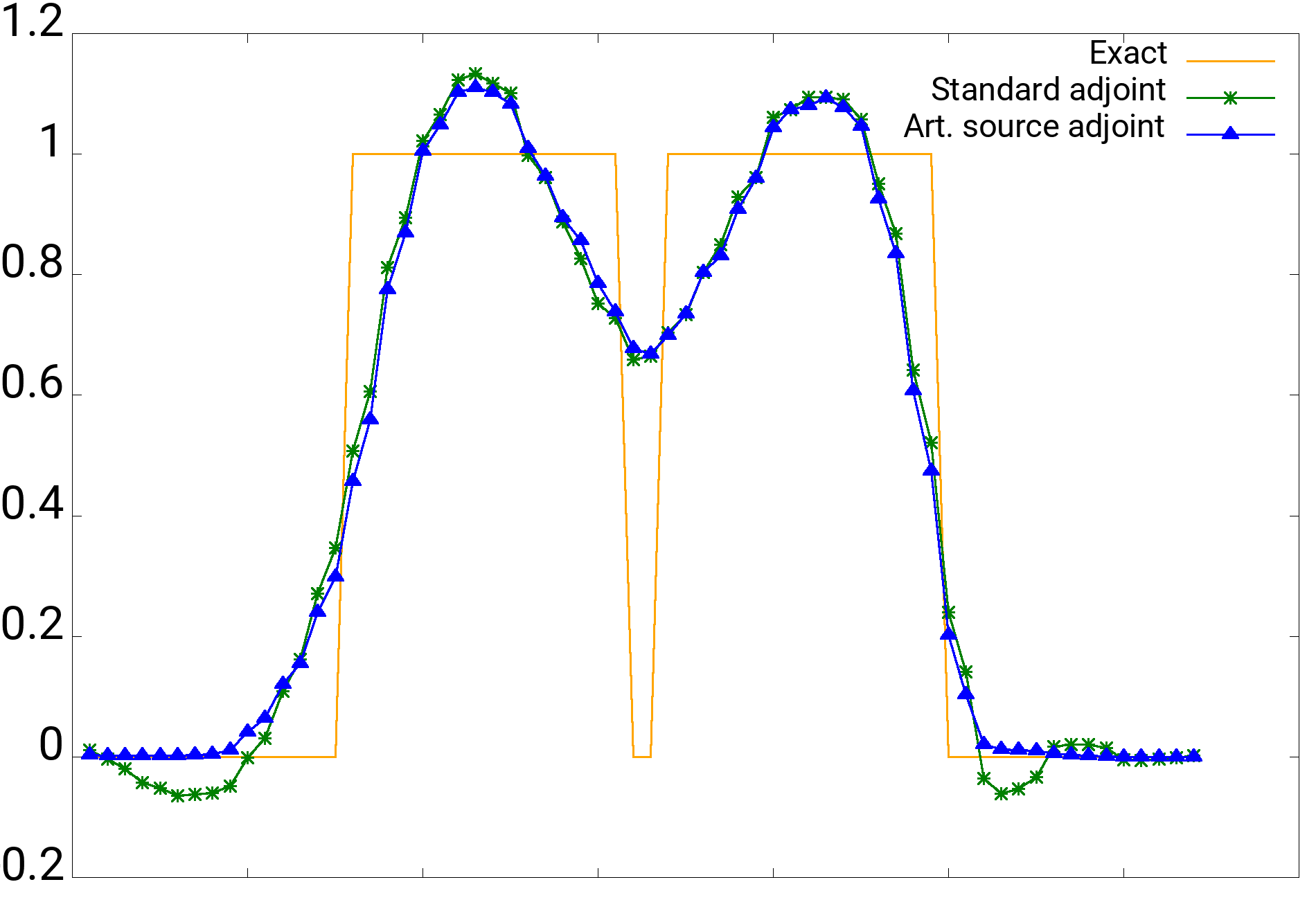}
  \caption{}
  \label{fig:DF1CylinderAdjoinrSourceLim3}
\end{subfigure}%
\begin{subfigure}{.5\textwidth}
  \centering
  \includegraphics[width=\textwidth]{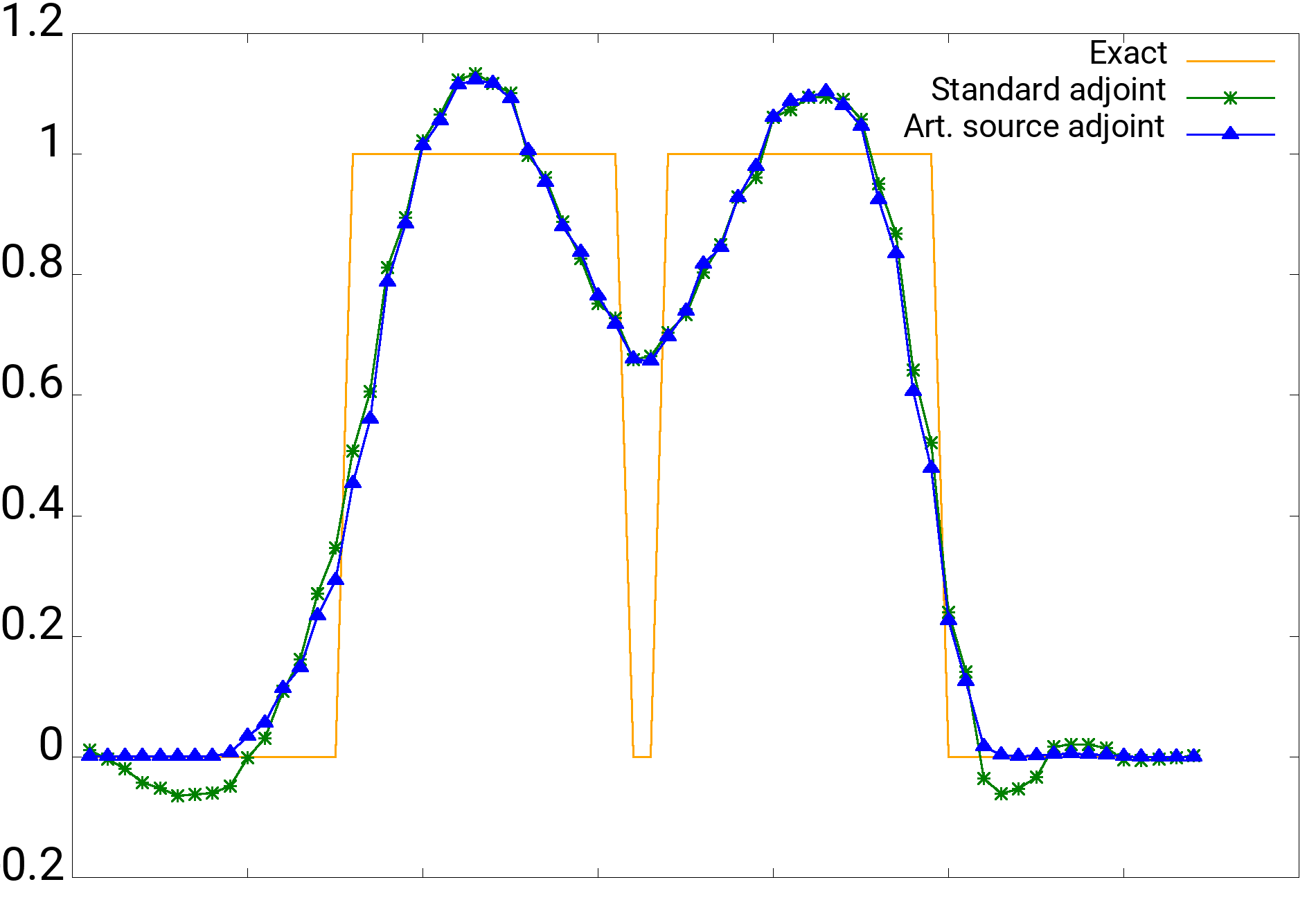}
  \caption{}
  \label{fig:DF1CylinderAdjointSourceLim4}
\end{subfigure}
\begin{subfigure}{.5\textwidth}
  \centering
  \includegraphics[width=\textwidth]{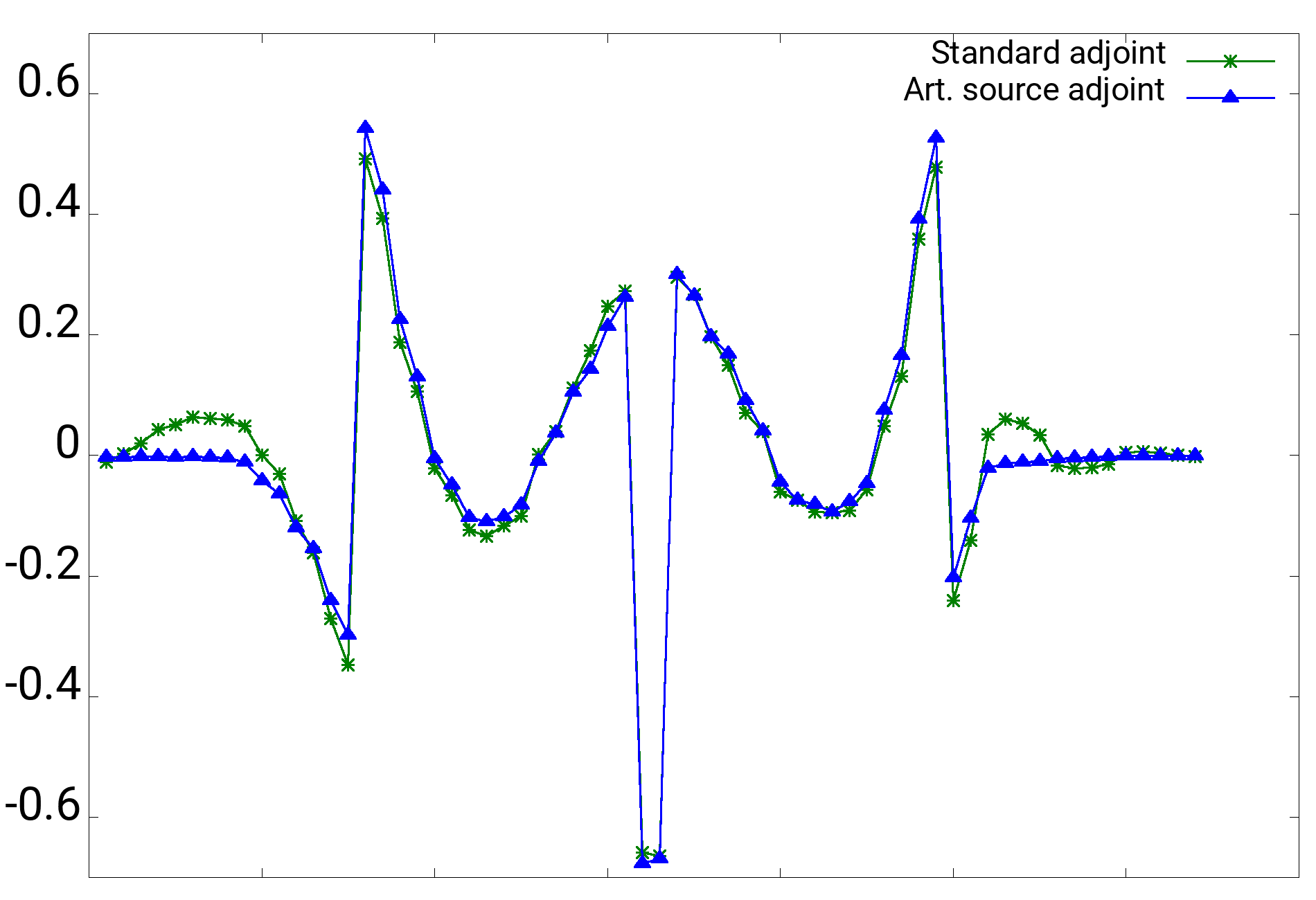}
  \caption{}
  \label{fig:DF1CylinderAdjointSourceDiffLim3}
\end{subfigure}%
\begin{subfigure}{.5\textwidth}
  \centering
  \includegraphics[width=\textwidth]{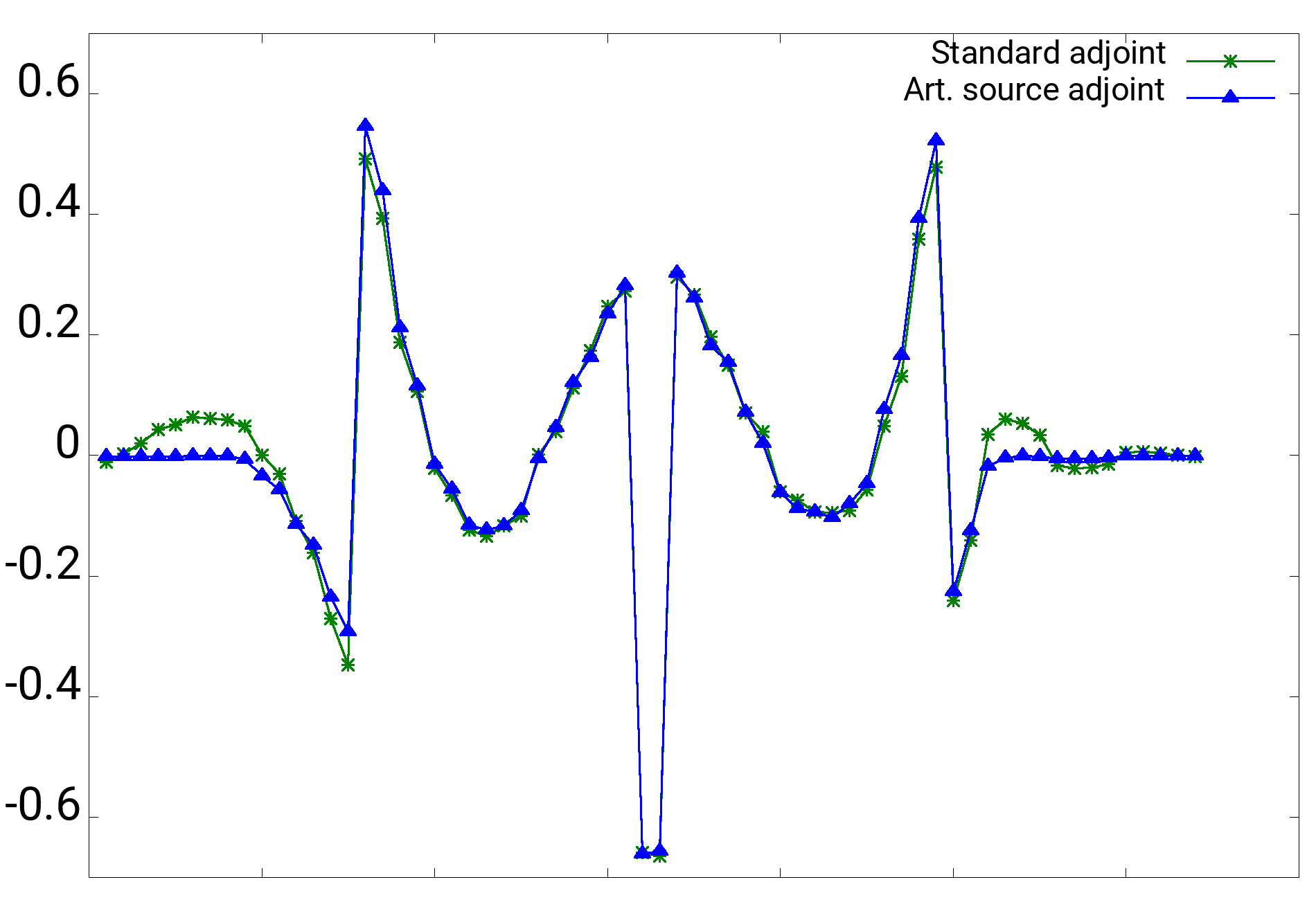}
  \caption{}
  \label{fig:DF1CylinerAdjointSourceDiffLim4}
\end{subfigure}
\caption{Deformational flow, $\nabla \vec{v}=0$, slotted cylinders, a)-d) - along the curve exact vs standard adjoint vs art. source adjoint: 
                                                                       a) solutions, art. source with limiter ~ \cite{zalesak1979fully}, \cite{schar1996synchronous} 
                                                                       b) solutions, art. source with limiter ~ \cite{zalesak1979fully}, \cite{harris2011flux} 
                                                                       c) errors, art. source with limiter ~ \cite{zalesak1979fully}, \cite{schar1996synchronous}  
                                                                       d) errors, art. source with limiter ~ \cite{zalesak1979fully}, \cite{harris2011flux}   
                                                                       } 
\label{fig:DefFl1:2SlottedCylinderCurve}
\end{figure}

\begin{figure}
\begin{subfigure}{.5\textwidth}
  \centering
  \includegraphics[width=\textwidth]{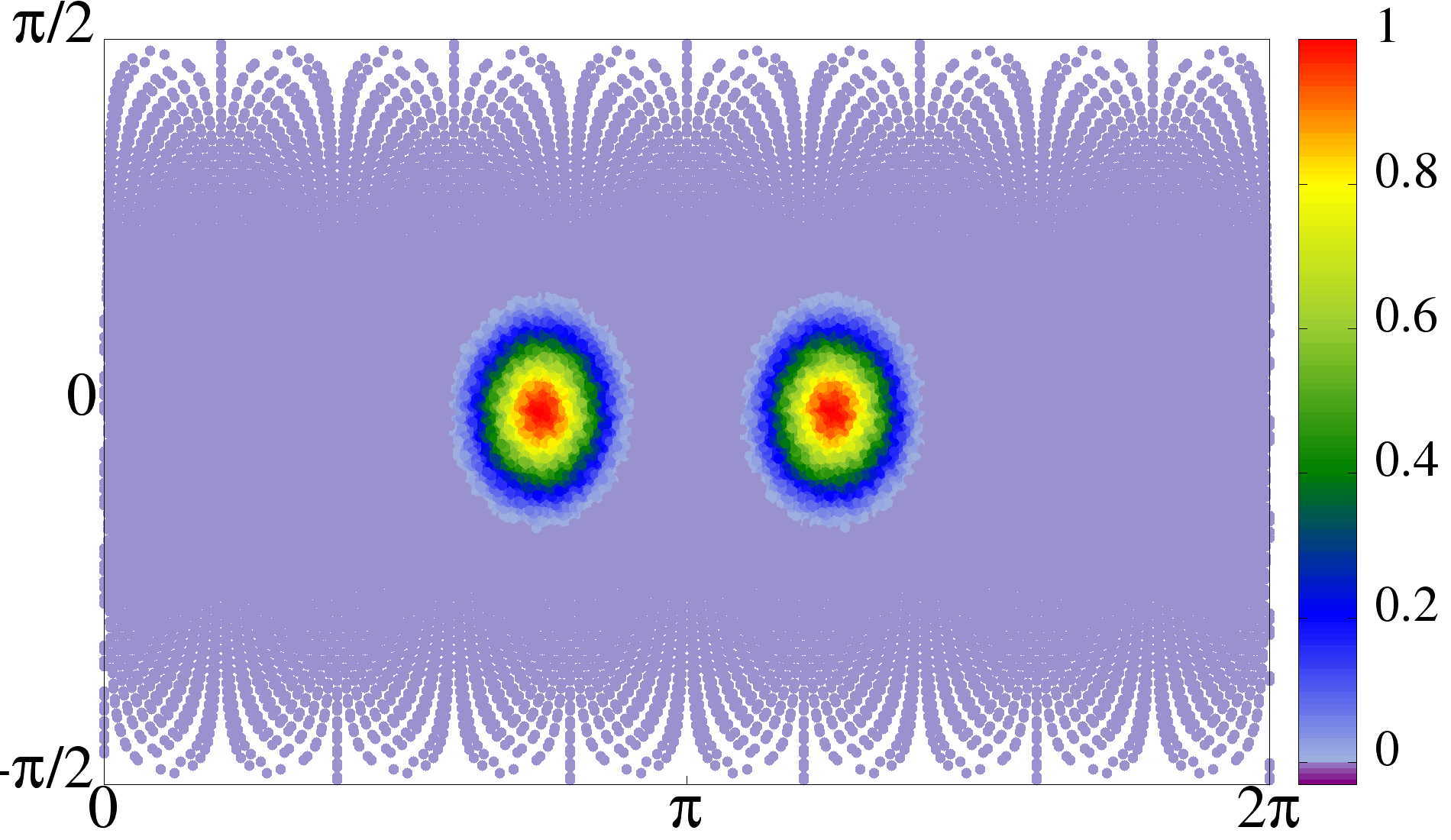}
  \caption{}
  \label{fig:DF3CosineExact}
\end{subfigure}%
\begin{subfigure}{.5\textwidth}
  \centering
  \includegraphics[width=\textwidth]{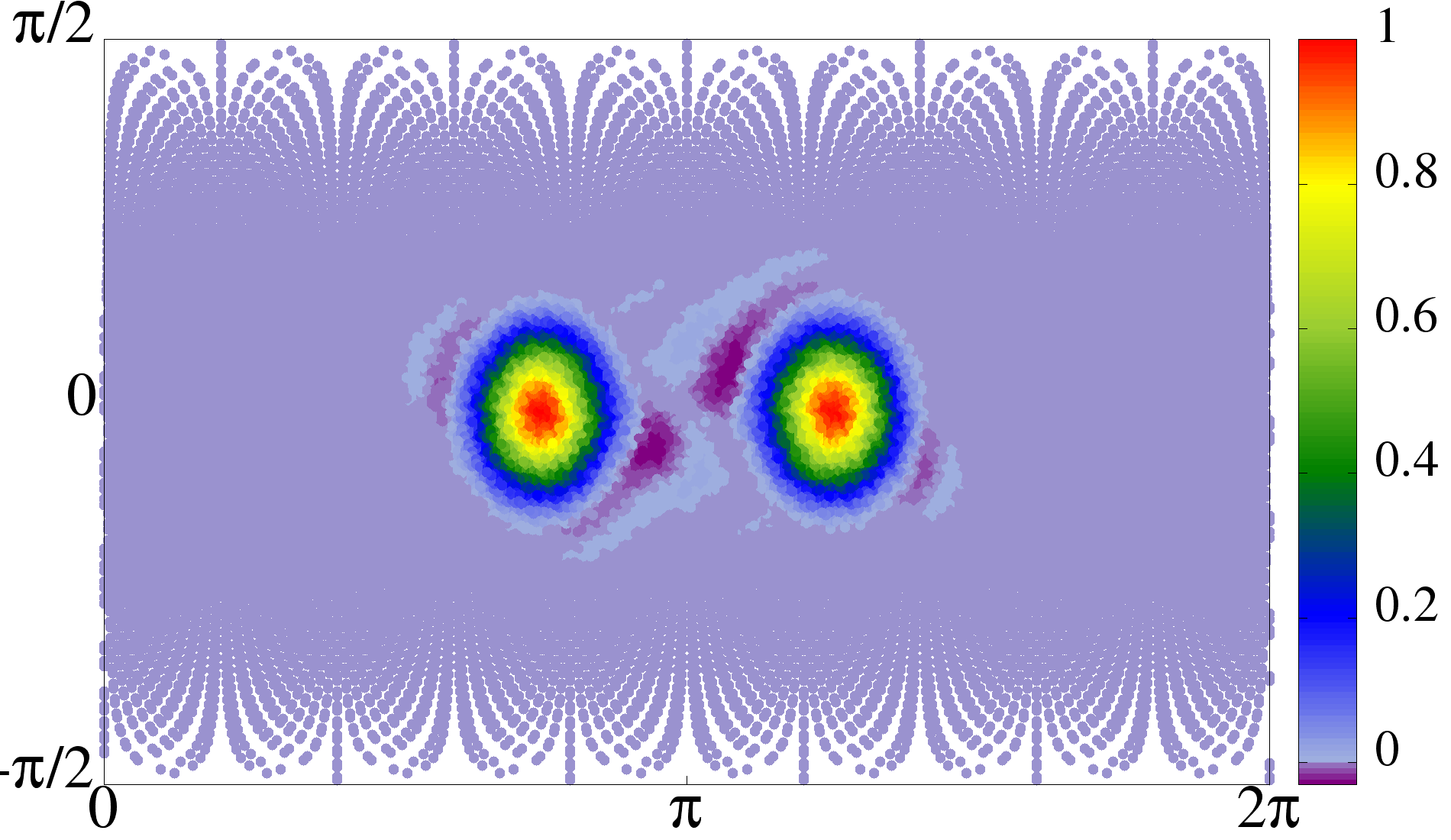}
  \caption{}
  \label{fig:DF3CosineAdjoint}
\end{subfigure}
\begin{subfigure}{.5\textwidth}
  \centering
  \includegraphics[width=\textwidth]{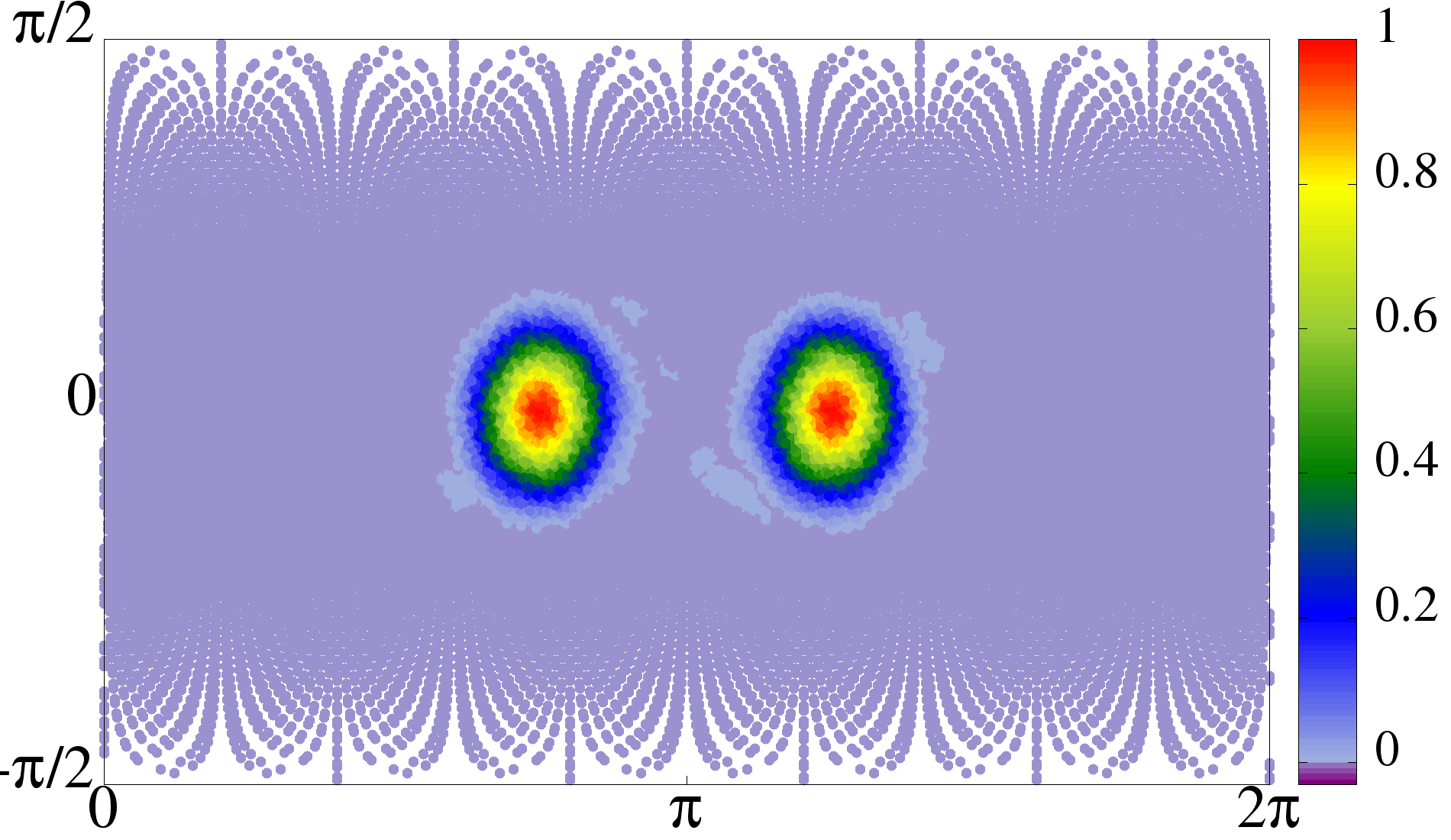}
  \caption{}
  \label{fig:DF3CosineSourceLim3}
\end{subfigure}%
\begin{subfigure}{.5\textwidth}
  \centering
  \includegraphics[width=\textwidth]{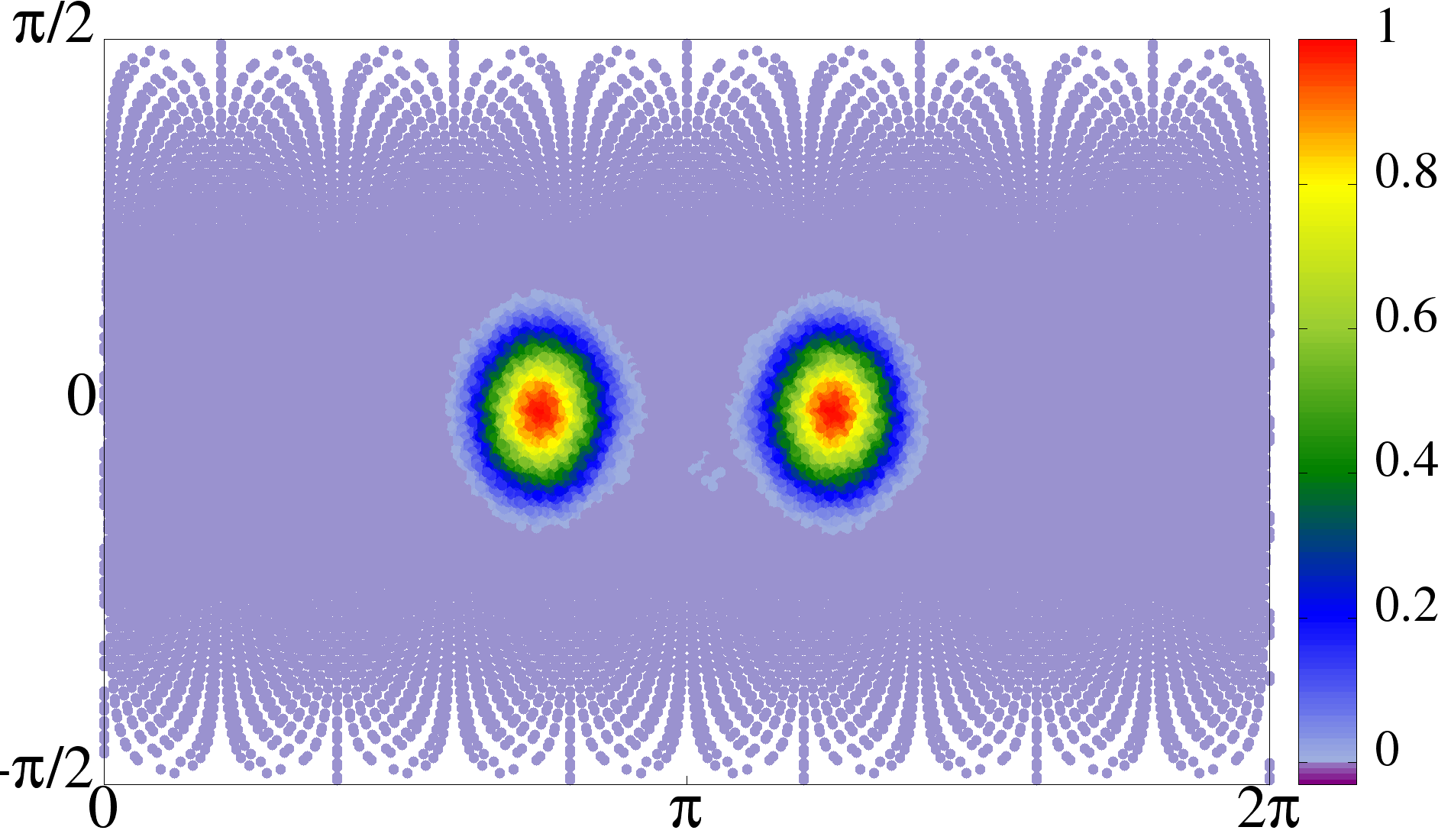}
  \caption{}
  \label{fig:DF3CosineSourceLim4}
\end{subfigure}
\caption{Deformational flow, $\nabla \vec{v}\neq0$, cosine bells, a)-d) - contour plots with color bar: 
                                          a) exact solution
                                          b) standard adjoint, 
                                          c) art. source adjoint, limiter ~ \cite{zalesak1979fully}, \cite{schar1996synchronous}
                                          d) art. source adjoint, limiter ~\cite{zalesak1979fully}, \cite{harris2011flux}     
                                          }
\label{fig:DefFl2:2CosBellContour}
\end{figure}
%
\begin{figure}[h]
\begin{subfigure}{.5\textwidth}
  \centering
  \includegraphics[width=\textwidth]{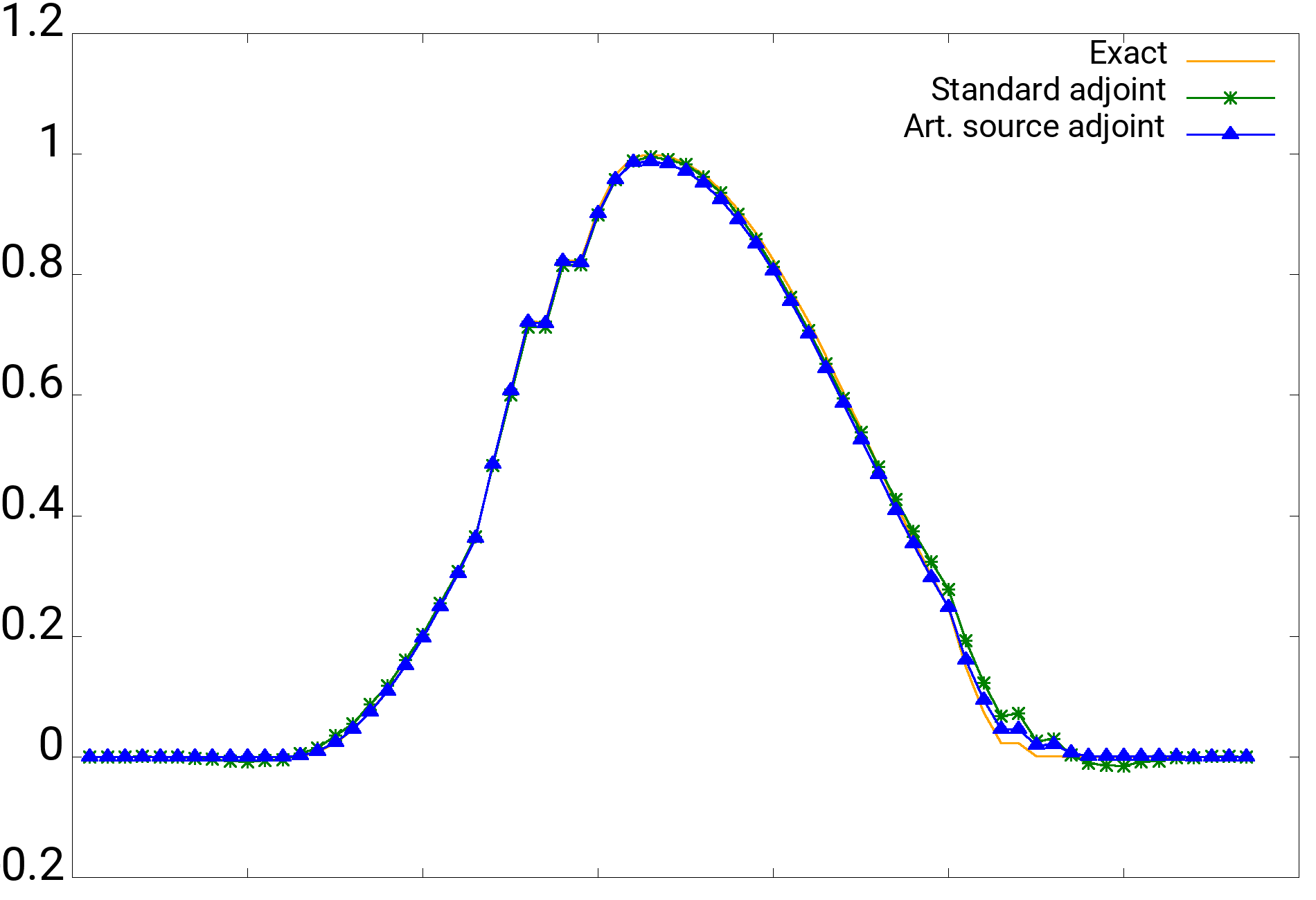}
  \caption{}
  \label{fig:DF3CosineAdjointSourceLim3}
\end{subfigure}%
\begin{subfigure}{.5\textwidth}
  \centering
  \includegraphics[width=\textwidth]{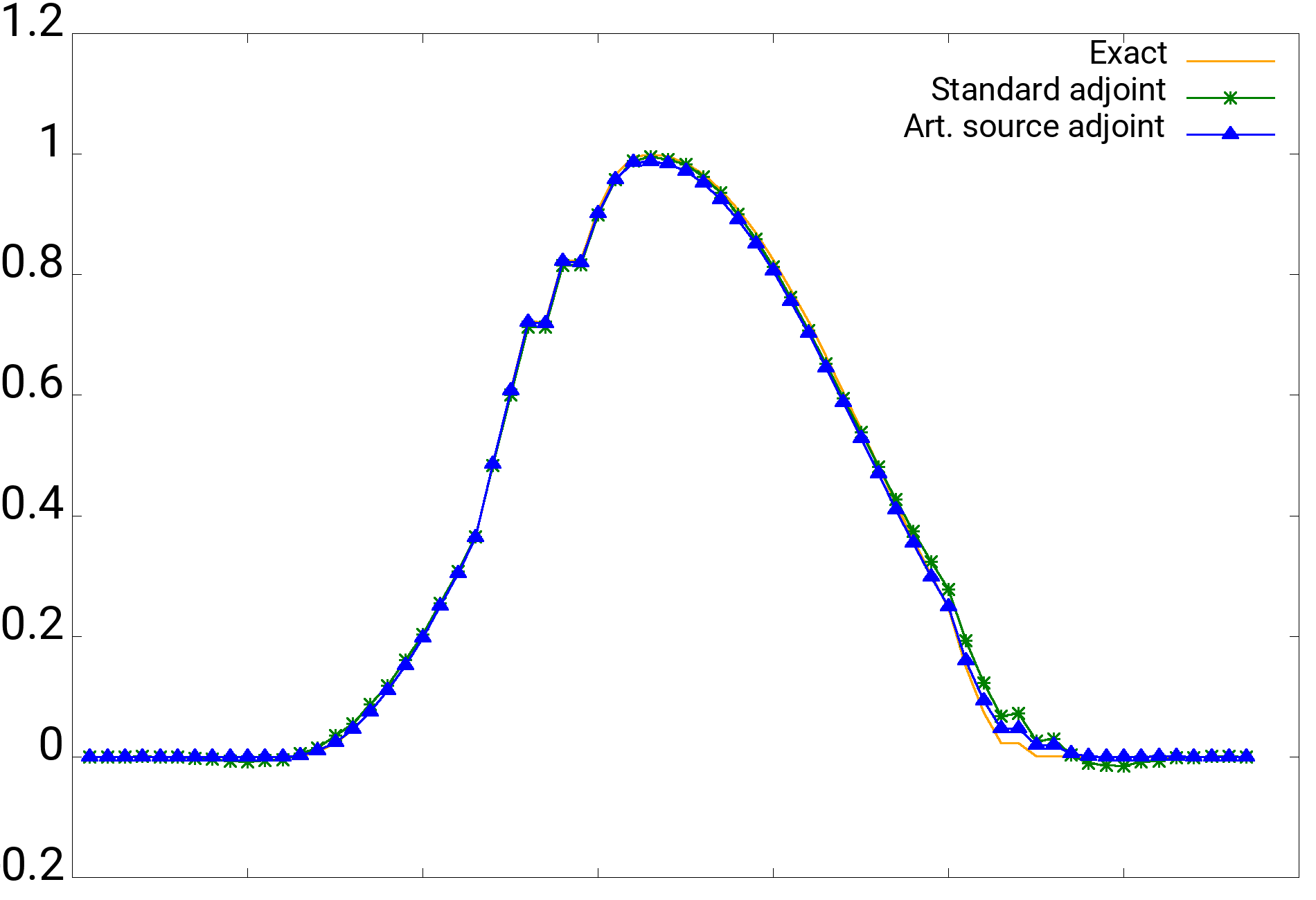}
  \caption{}
  \label{fig:DF3CosineAdjointSourceLim4}
\end{subfigure}
\begin{subfigure}{.5\textwidth}
  \centering
  \includegraphics[width=\textwidth]{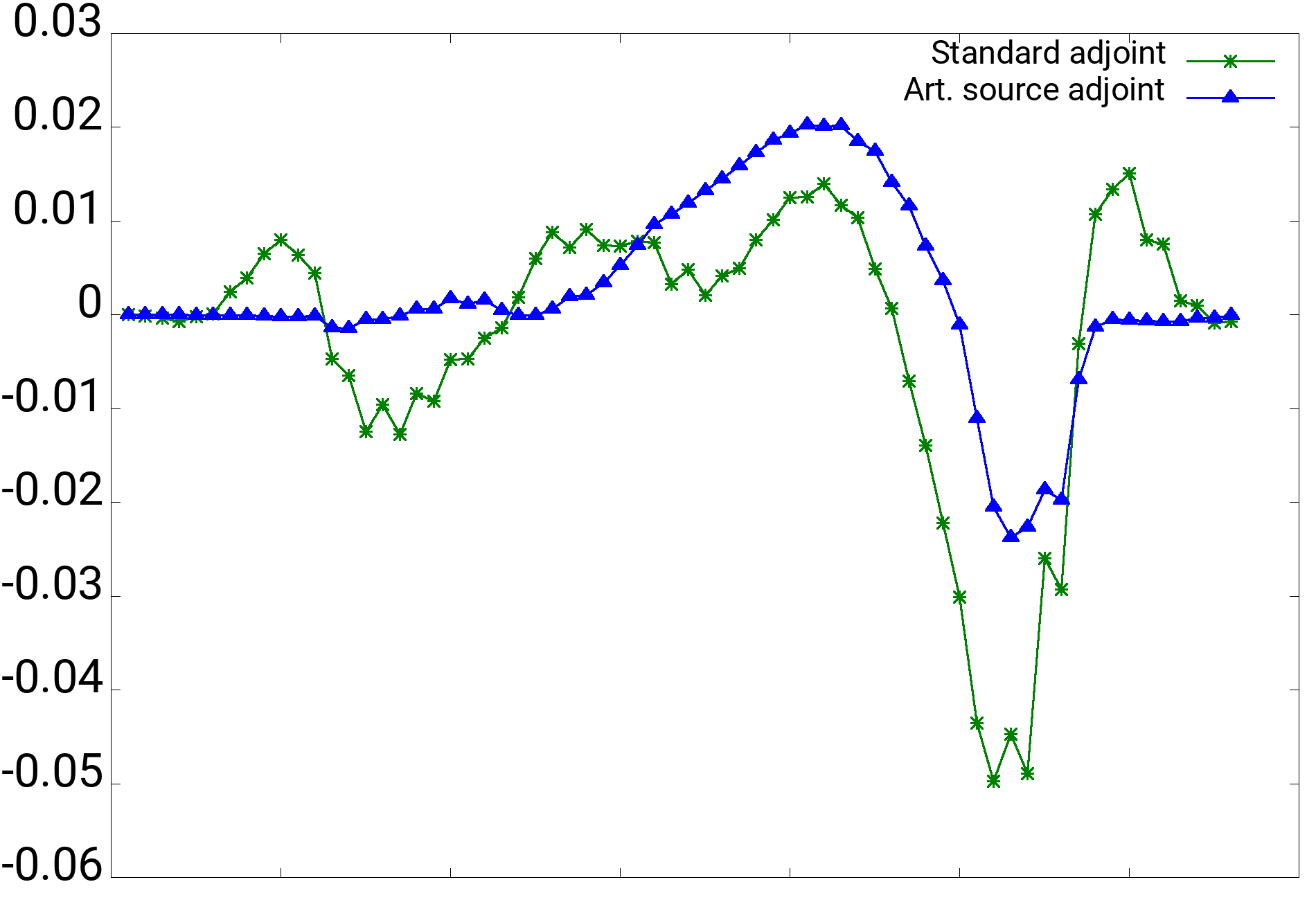}
  \caption{}
  \label{fig:DF3CosineAdjointSourceDiffLim3}
\end{subfigure}%
\begin{subfigure}{.5\textwidth}
  \centering
  \includegraphics[width=\textwidth]{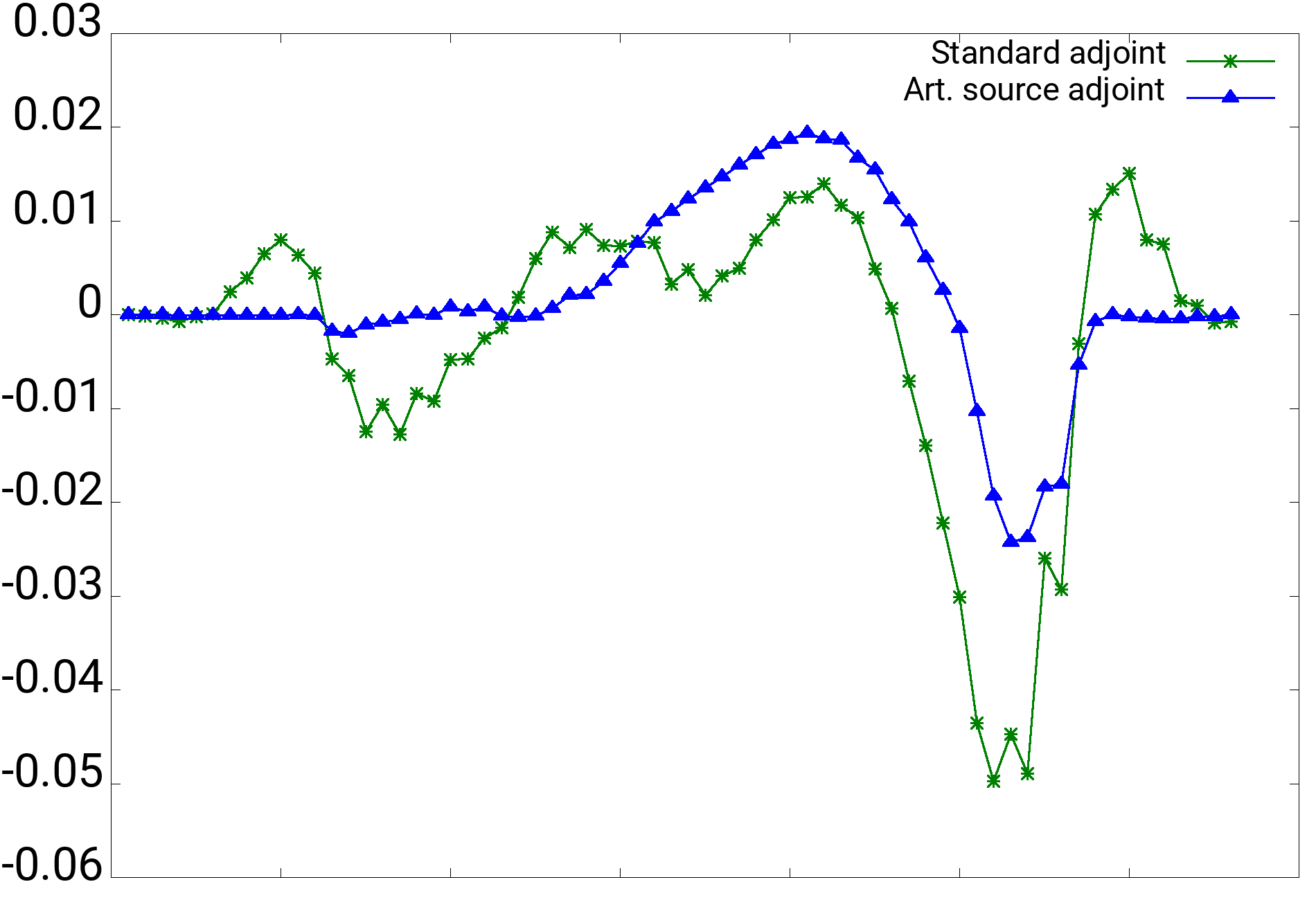}
  \caption{}
  \label{fig:DF3CosineAdjointSourceDiffLim4}
\end{subfigure}
\caption{Deformational flow, $\nabla \vec{v}\neq0$, cosine bells, a)-d) - along the curve exact vs standard adjoint vs art. source adjoint: 
                                                                       a) solutions, art. source with limiter ~ \cite{zalesak1979fully}, \cite{schar1996synchronous} 
                                                                       b) solutions, art. source with limiter ~ \cite{zalesak1979fully}, \cite{harris2011flux} 
                                                                       c) errors, art. source with limiter ~ \cite{zalesak1979fully}, \cite{schar1996synchronous}  
                                                                       d) errors, art. source with limiter ~ \cite{zalesak1979fully}, \cite{harris2011flux}  
                                                                       } 
\label{fig:DefFl2:2CosBellCurve}
\end{figure}
%
\begin{figure}
\begin{subfigure}{.5\textwidth}
  \centering
  \includegraphics[width=\textwidth]{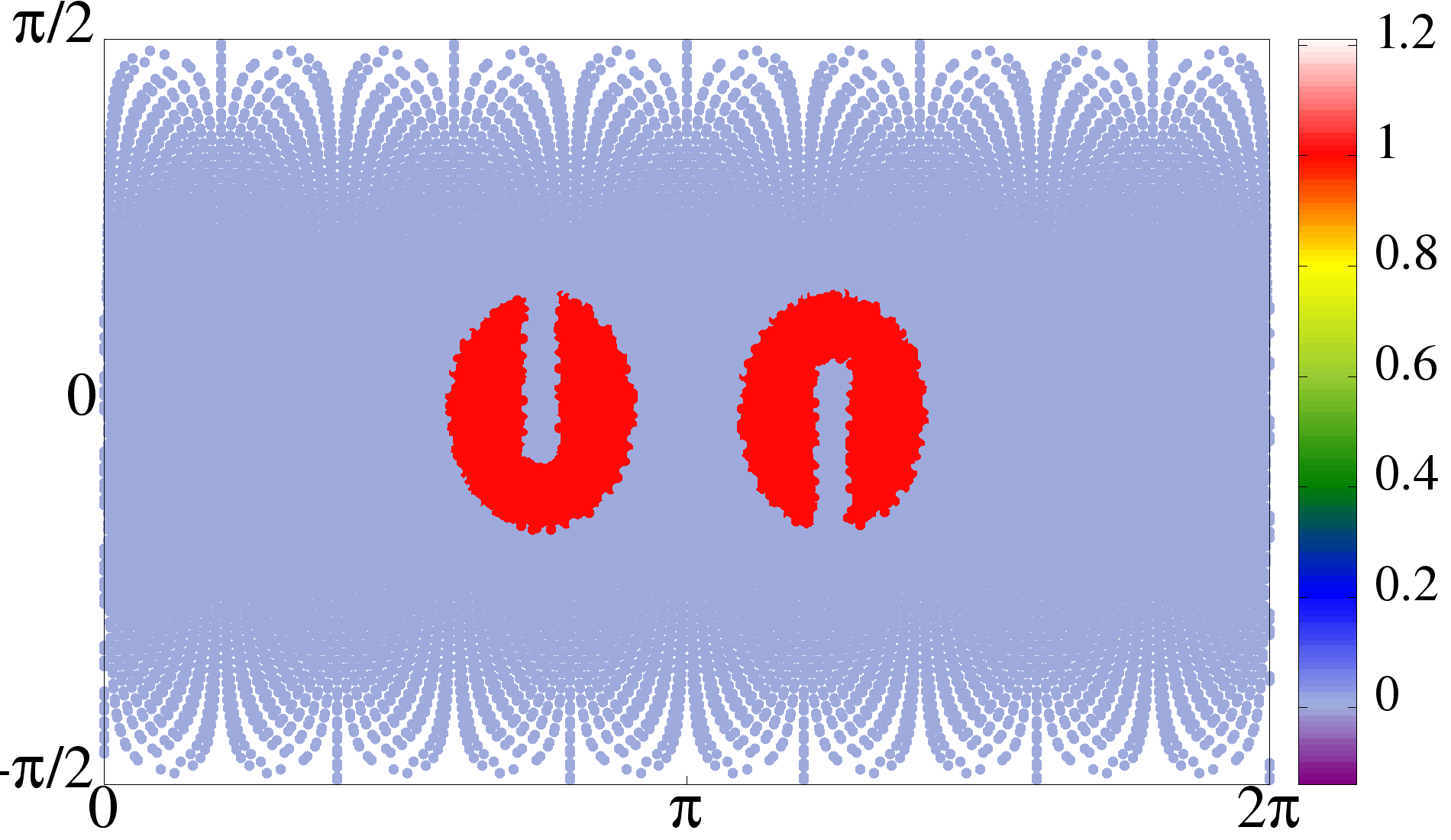}
  \caption{}
  \label{fig:DF3CylinderExact}
\end{subfigure}%
\begin{subfigure}{.5\textwidth}
  \centering
  \includegraphics[width=\textwidth]{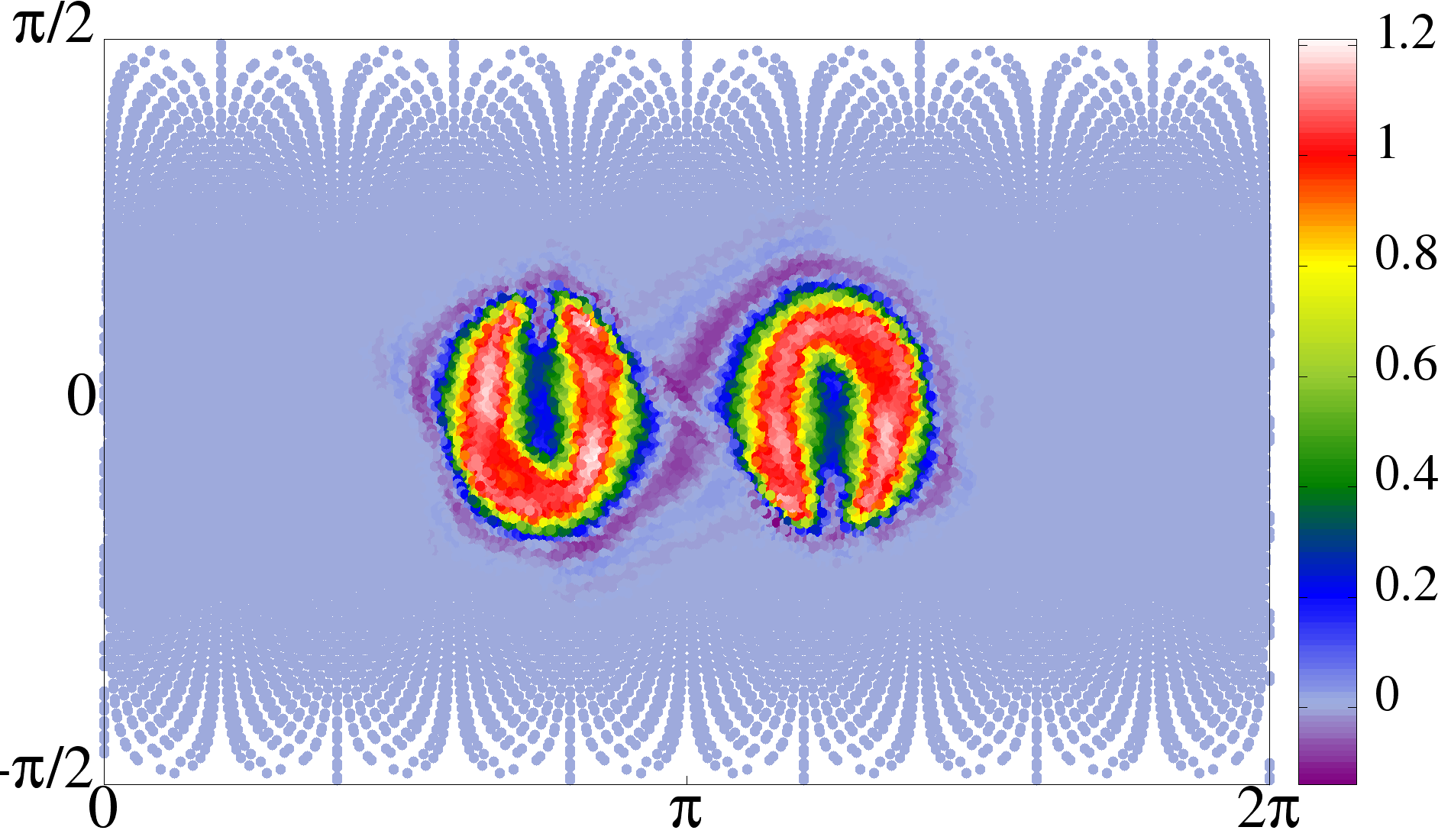}
  \caption{}
  \label{fig:DF3CylinderAdjoint}
\end{subfigure}
\begin{subfigure}{.5\textwidth}
  \centering
  \includegraphics[width=\textwidth]{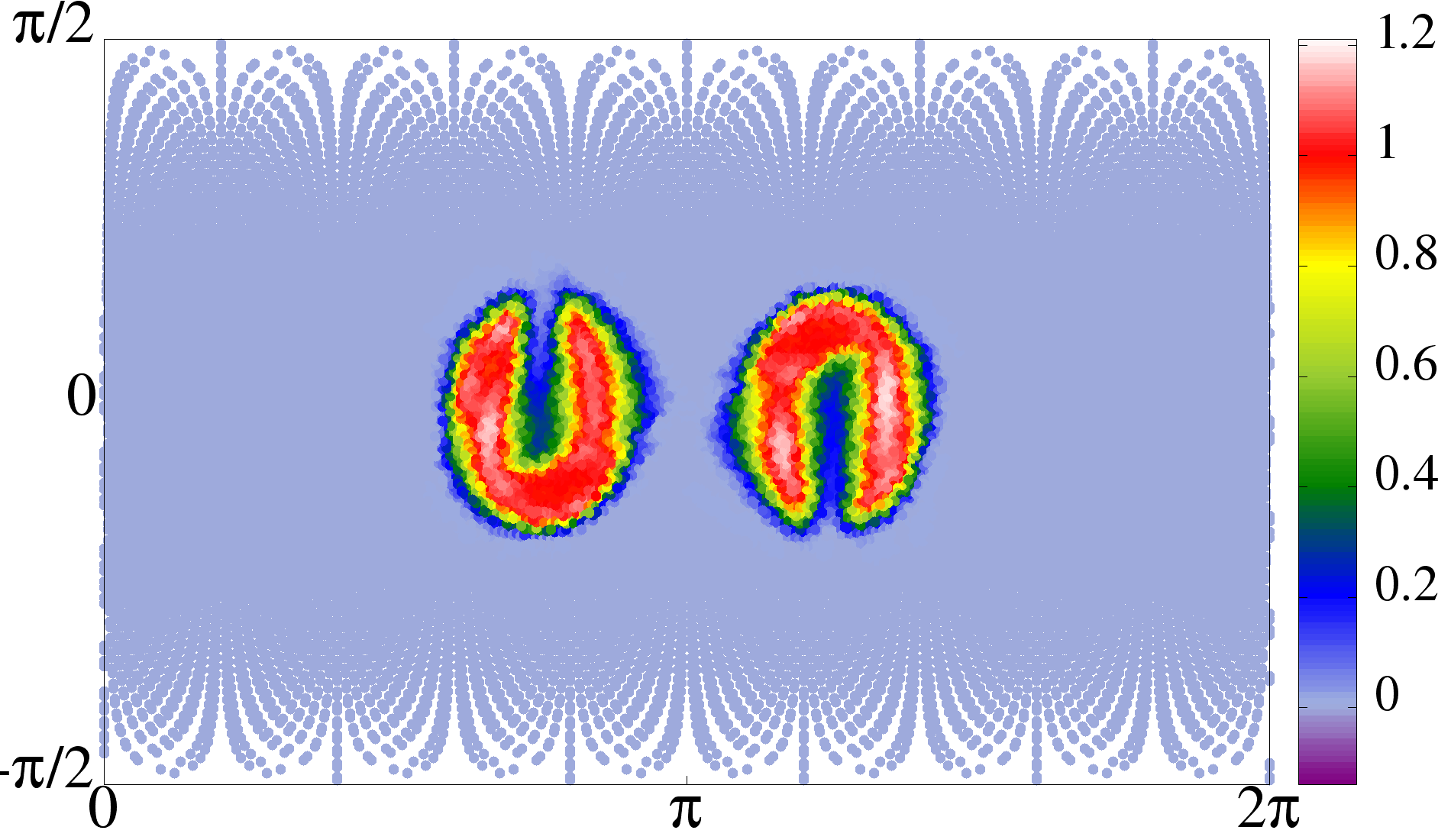}
  \caption{}
  \label{fig:DF3CylinderSourceLim3}
\end{subfigure}%
\begin{subfigure}{.5\textwidth}
  \centering
  \includegraphics[width=\textwidth]{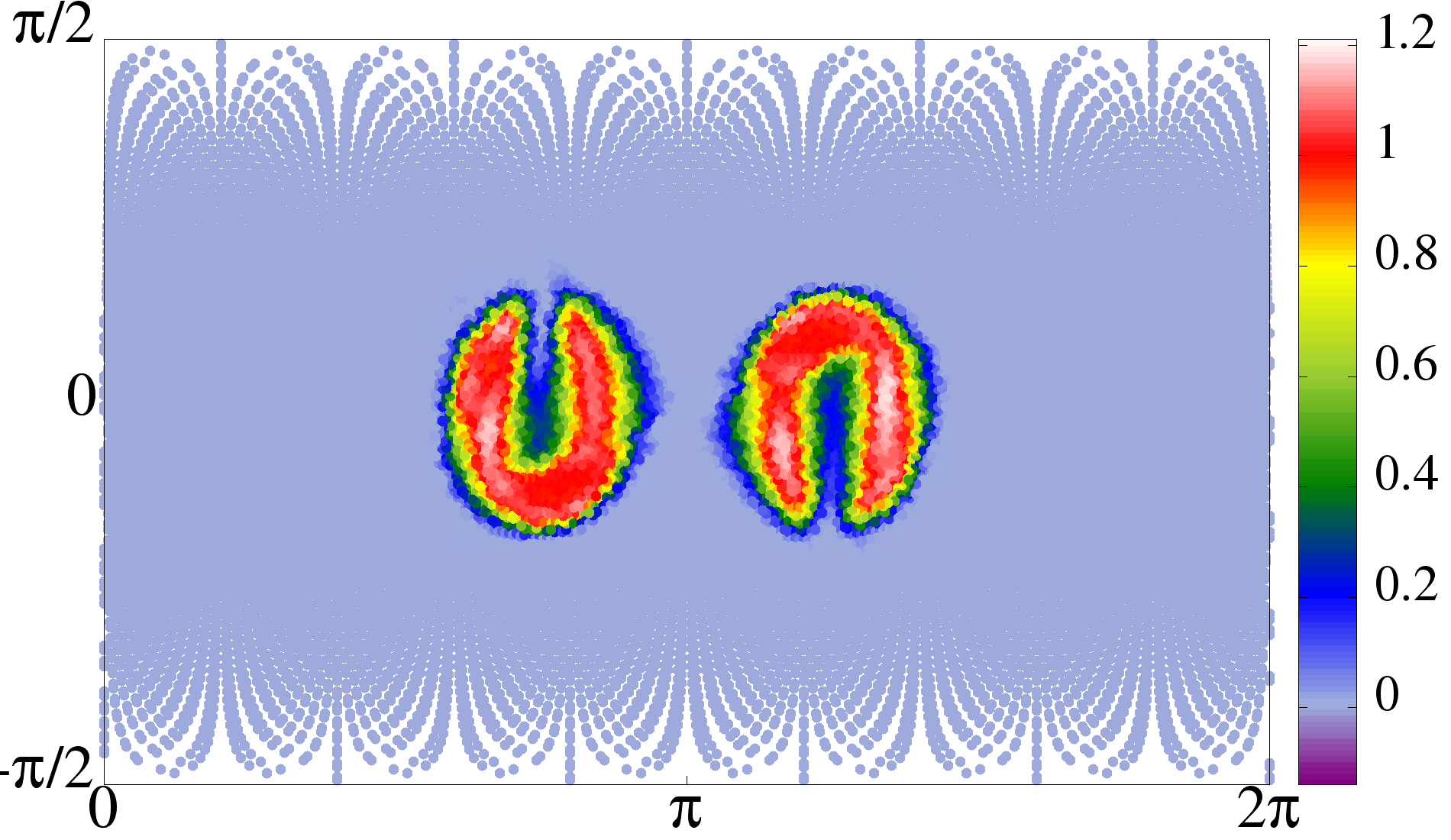}
  \caption{}
  \label{fig:DF3CylinderLim4}
\end{subfigure}
\caption{Deformational flow, $\nabla \vec{v}\neq0$, slotted cylinders, a)-d) - contour plots with color bar: 
                                          a) exact solution
                                          b) standard adjoint, 
                                          c) art. source adjoint, limiter ~ \cite{zalesak1979fully}, \cite{schar1996synchronous}
                                          d) art. source adjoint, limiter ~\cite{zalesak1979fully}, \cite{harris2011flux}    
                                          }
\label{fig:DefFl2:2SlottedCylinderContour}
\end{figure}

\begin{figure}[h]
\begin{subfigure}{.5\textwidth}
  \centering
  \includegraphics[width=\textwidth]{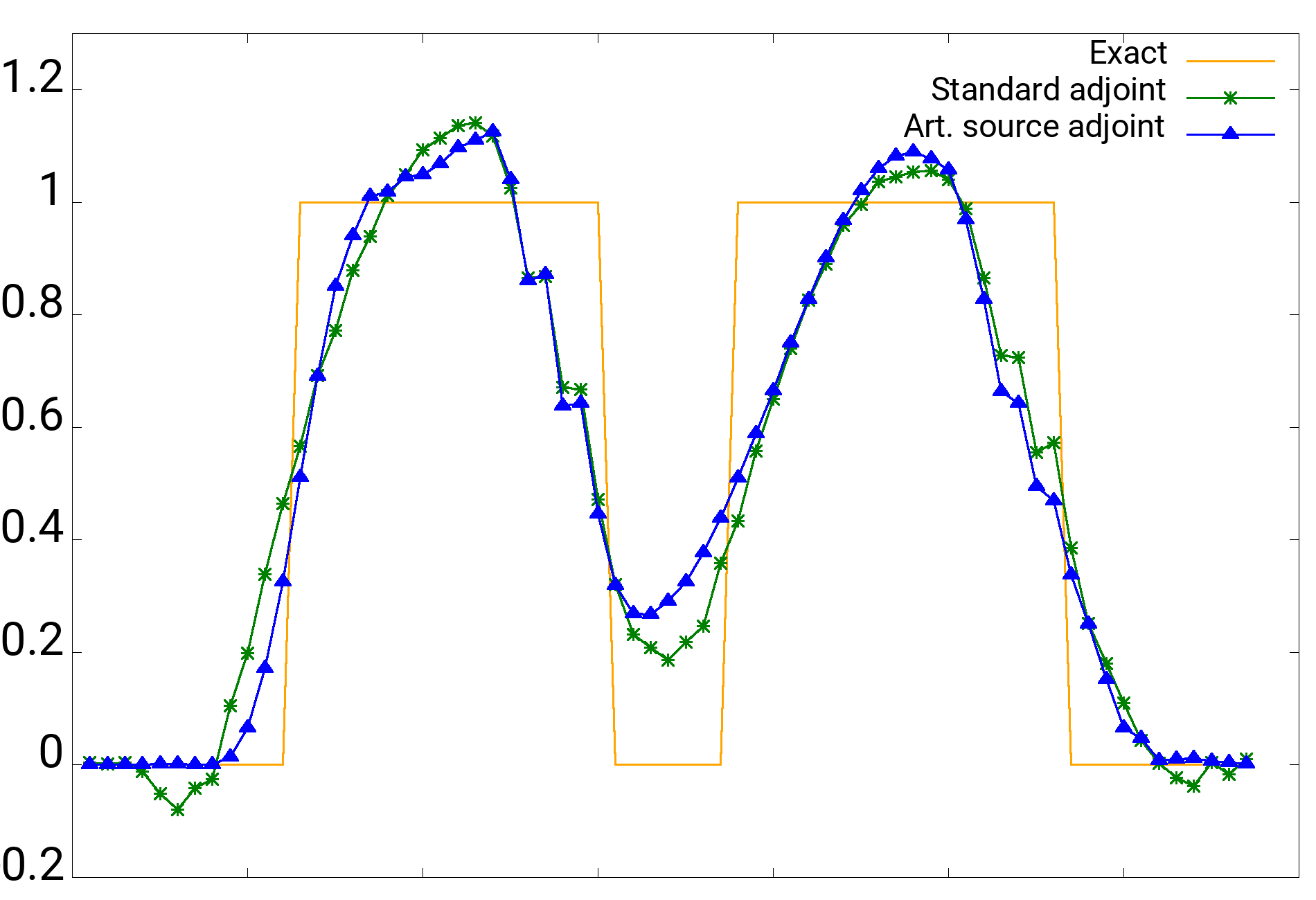}
  \caption{}
  \label{fig:DF3CylinderAdjointSourceLim3}
\end{subfigure}%
\begin{subfigure}{.5\textwidth}
  \centering
  \includegraphics[width=\textwidth]{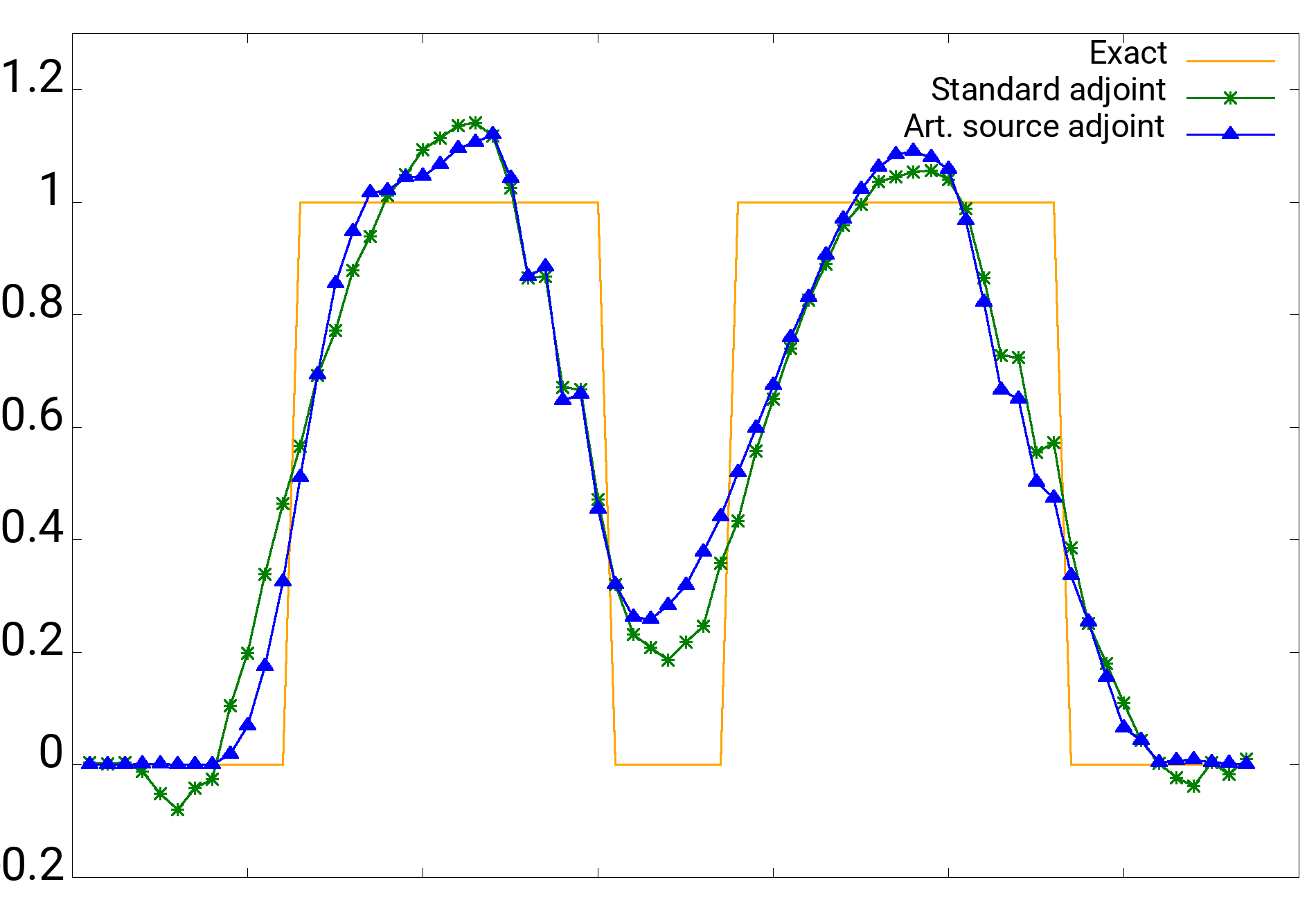}
  \caption{}
  \label{fig:DF3CylinderAdjointSourceLim4}
\end{subfigure}
\begin{subfigure}{.5\textwidth}
  \centering
  \includegraphics[width=\textwidth]{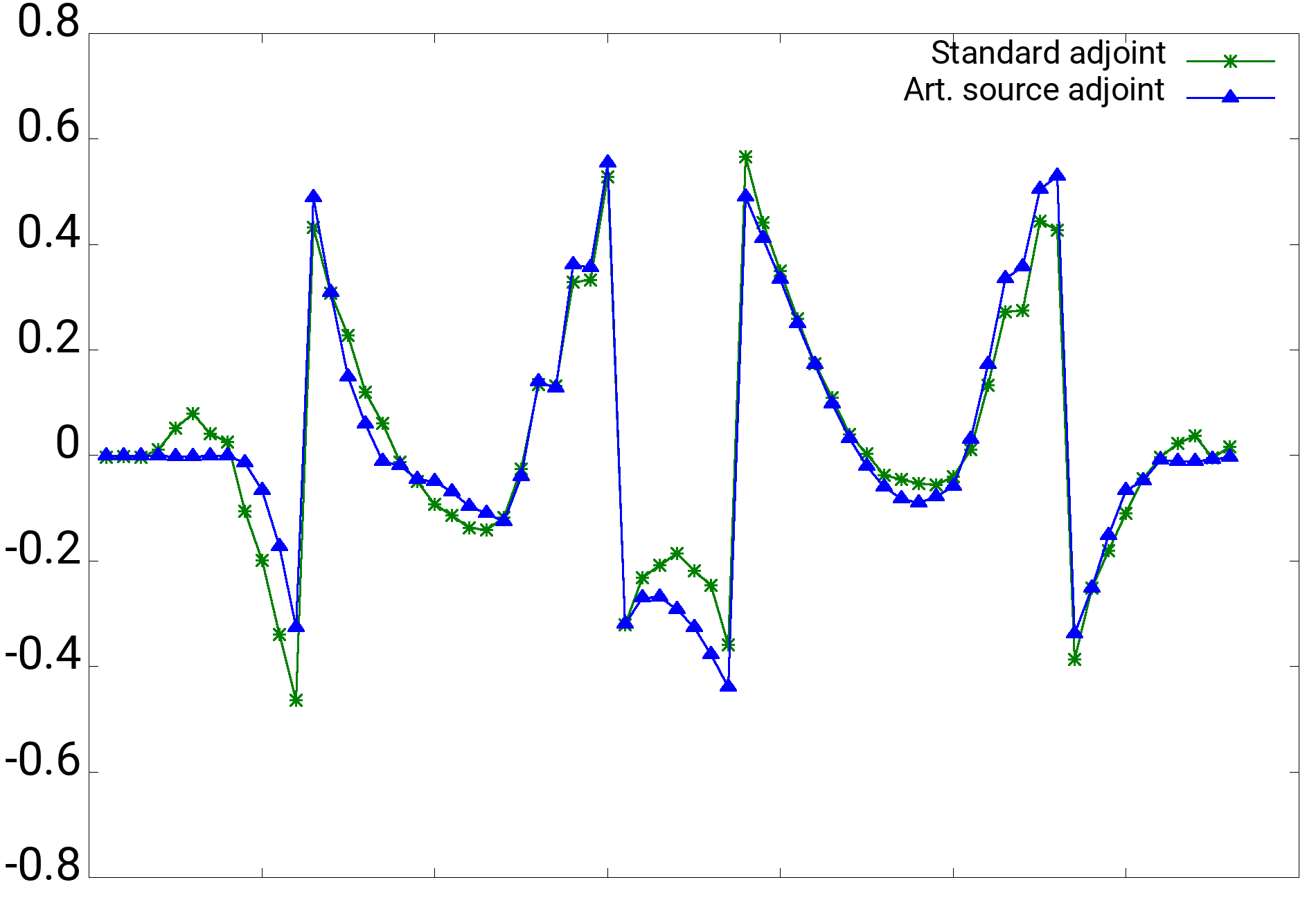}
  \caption{}
  \label{fig:DF3CylinderAdjointSourceDiffLim3}
\end{subfigure}%
\begin{subfigure}{.5\textwidth}
  \centering
  \includegraphics[width=\textwidth]{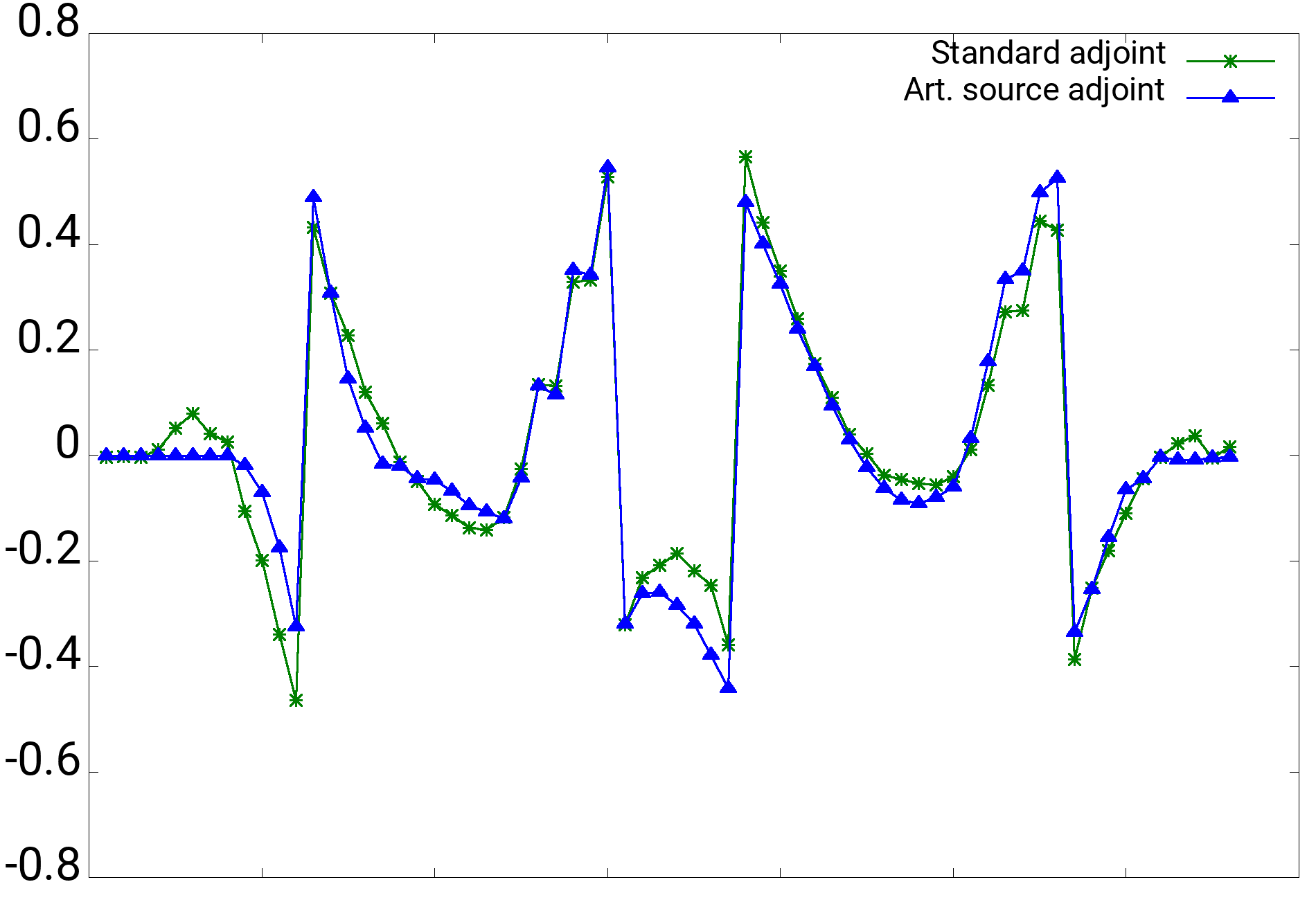}
  \caption{}
  \label{fig:DF3CylinderAdjointSourceDiffLim4}
\end{subfigure}
\caption{Deformational flow, $\nabla \vec{v}\neq0$, slotted cylinders, a)-d) - along the curve  exact vs standard adjoint vs art. source adjoint : 
                                                          a) solutions, art. source with limiter ~ \cite{zalesak1979fully}, \cite{schar1996synchronous} 
                                                          b) solutions, art. source with limiter ~ \cite{zalesak1979fully}, \cite{harris2011flux} 
                                                          c) errors, art. source with limiter ~ \cite{zalesak1979fully}, \cite{schar1996synchronous}  
                                                          d) errors, art. source with limiter ~ \cite{zalesak1979fully}, \cite{harris2011flux}  
                                                          } 
\label{fig:DefFl2:2SlottedCylinderCurve}
\end{figure}

\begin{figure}
\begin{subfigure}{.5\textwidth}
  \centering
  \includegraphics[width=\textwidth]{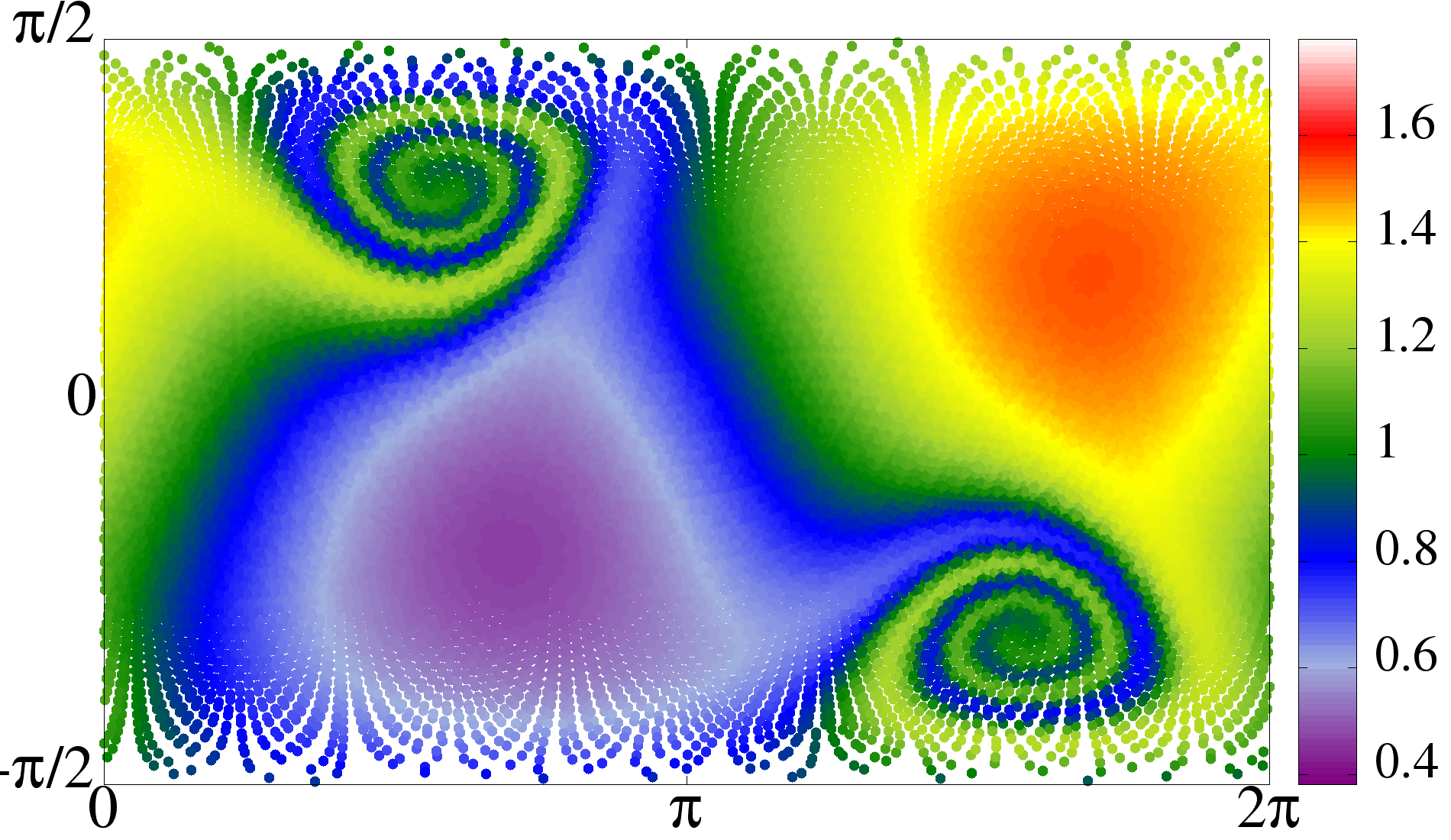}
  \caption{}
  \label{fig:MVExact}
\end{subfigure}%
\begin{subfigure}{.5\textwidth}
  \centering
  \includegraphics[width=\textwidth]{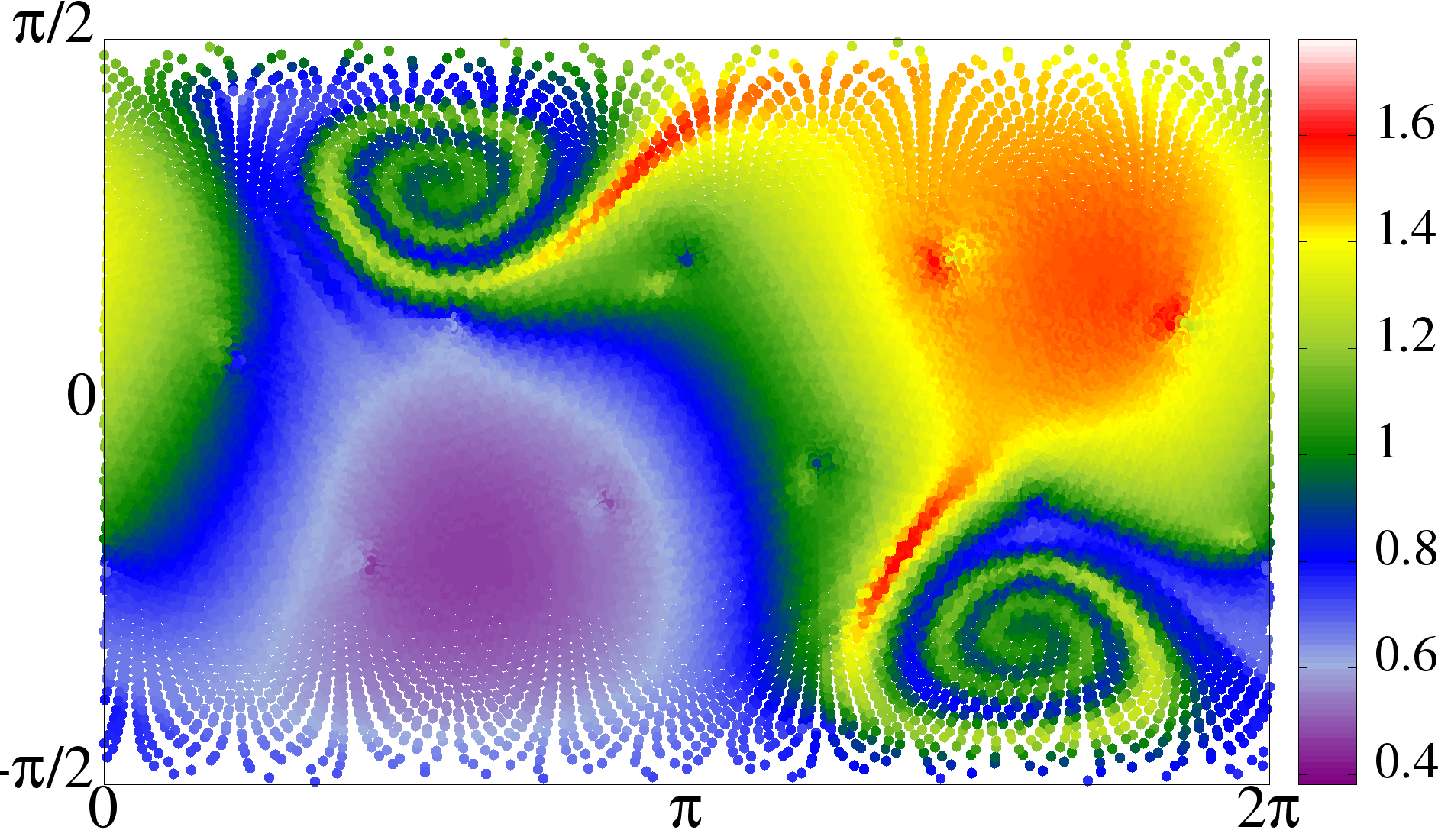}
  \caption{}
  \label{fig:MVAdjoint}
\end{subfigure}
\begin{subfigure}{.5\textwidth}
  \centering
  \includegraphics[width=\textwidth]{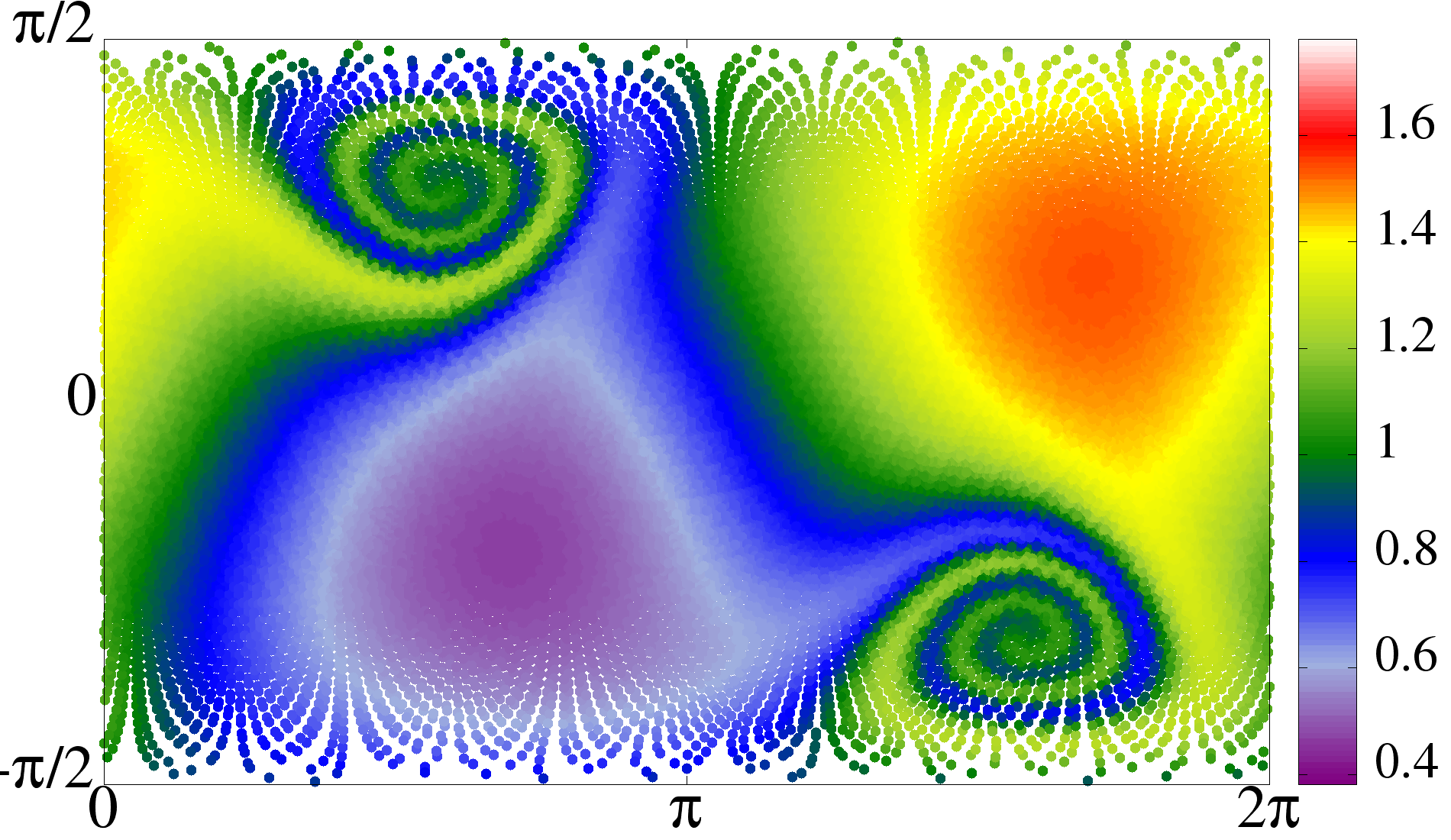}
  \caption{}
  \label{fig:MVSourceLim3}
\end{subfigure}%
\begin{subfigure}{.5\textwidth}
  \centering
  \includegraphics[width=\textwidth]{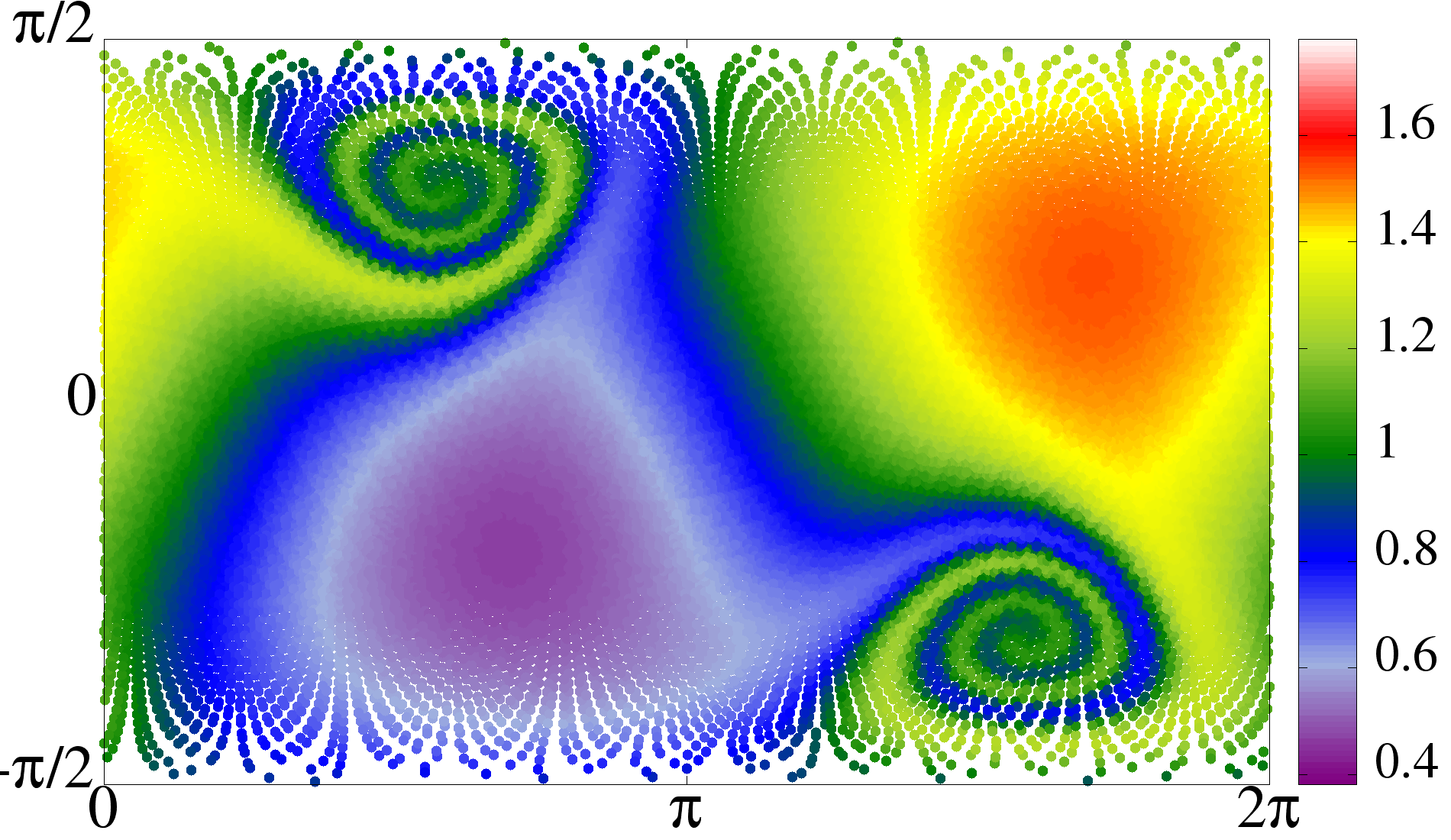}
  \caption{}
  \label{fig:MVSourceLim4}
\end{subfigure}
\caption{Moving vortices, a)-d) - contour plots with color bar: 
                                          a) exact solution
                                          b) standard adjoint, 
                                          c) art. source adjoint, limiter ~ \cite{zalesak1979fully}, \cite{schar1996synchronous}
                                          d) art. source adjoint, limiter ~\cite{zalesak1979fully}, \cite{harris2011flux}     
                                          }
\label{fig:MovingVortexContour}
\end{figure}

\begin{figure}
\begin{subfigure}{.5\textwidth}
  \centering
  \includegraphics[width=\textwidth]{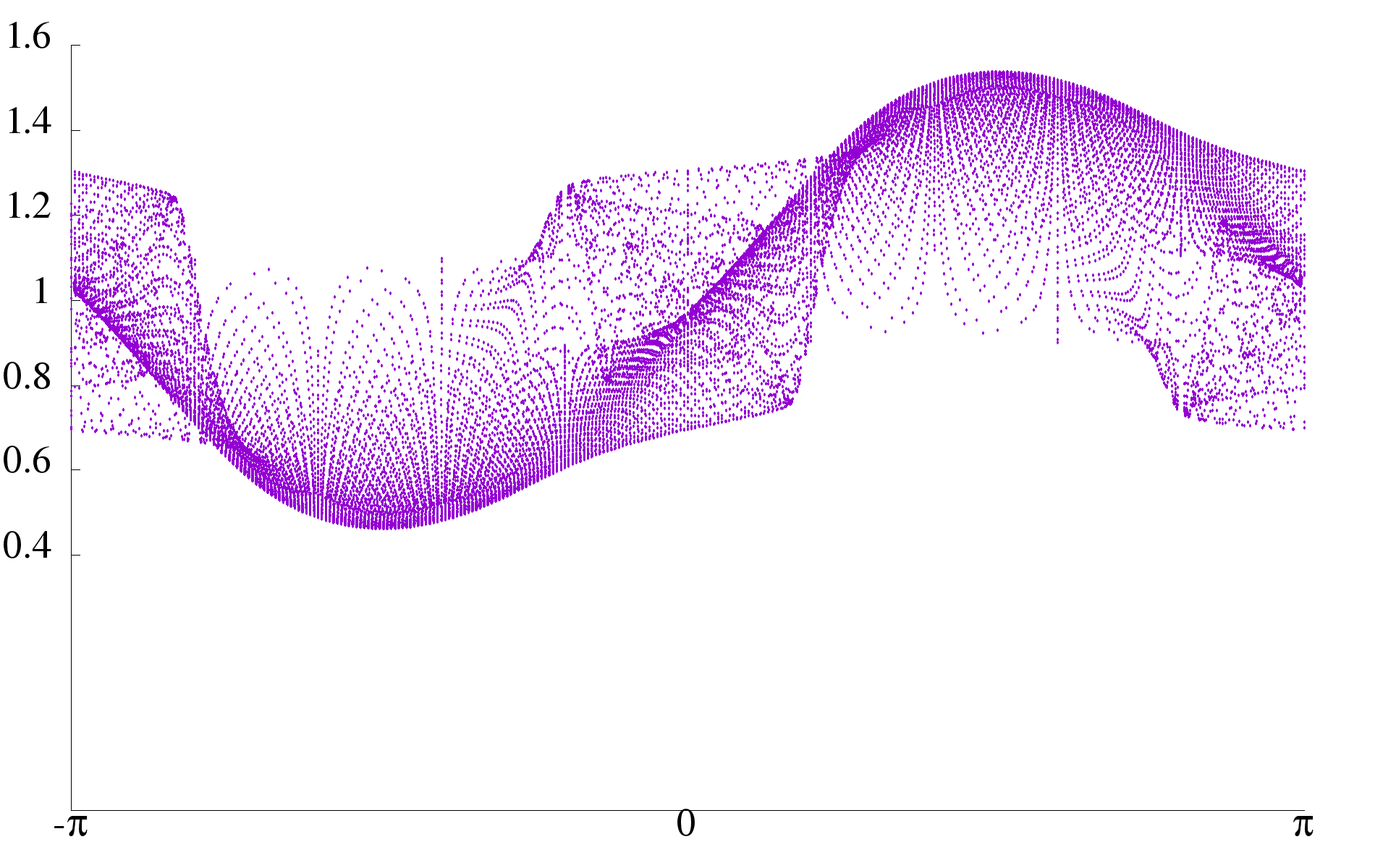}
  \caption{}
  \label{fig:MVExact3D}
\end{subfigure}%
\begin{subfigure}{.5\textwidth}
  \centering
  \includegraphics[width=\textwidth]{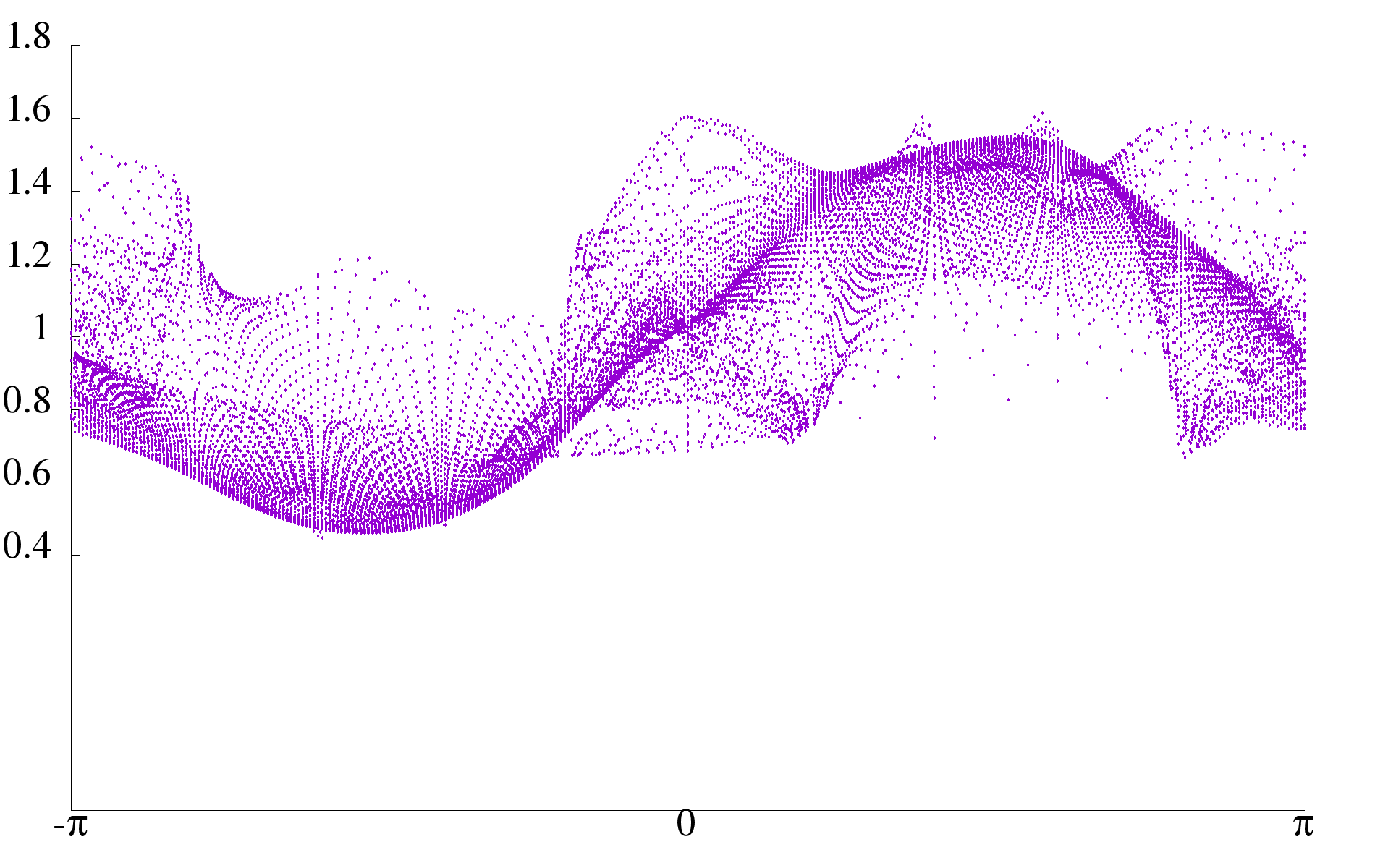}
  \caption{}
  \label{fig:MVAdjoint3D}
\end{subfigure}
\begin{subfigure}{.5\textwidth}
  \centering
  \includegraphics[width=\textwidth]{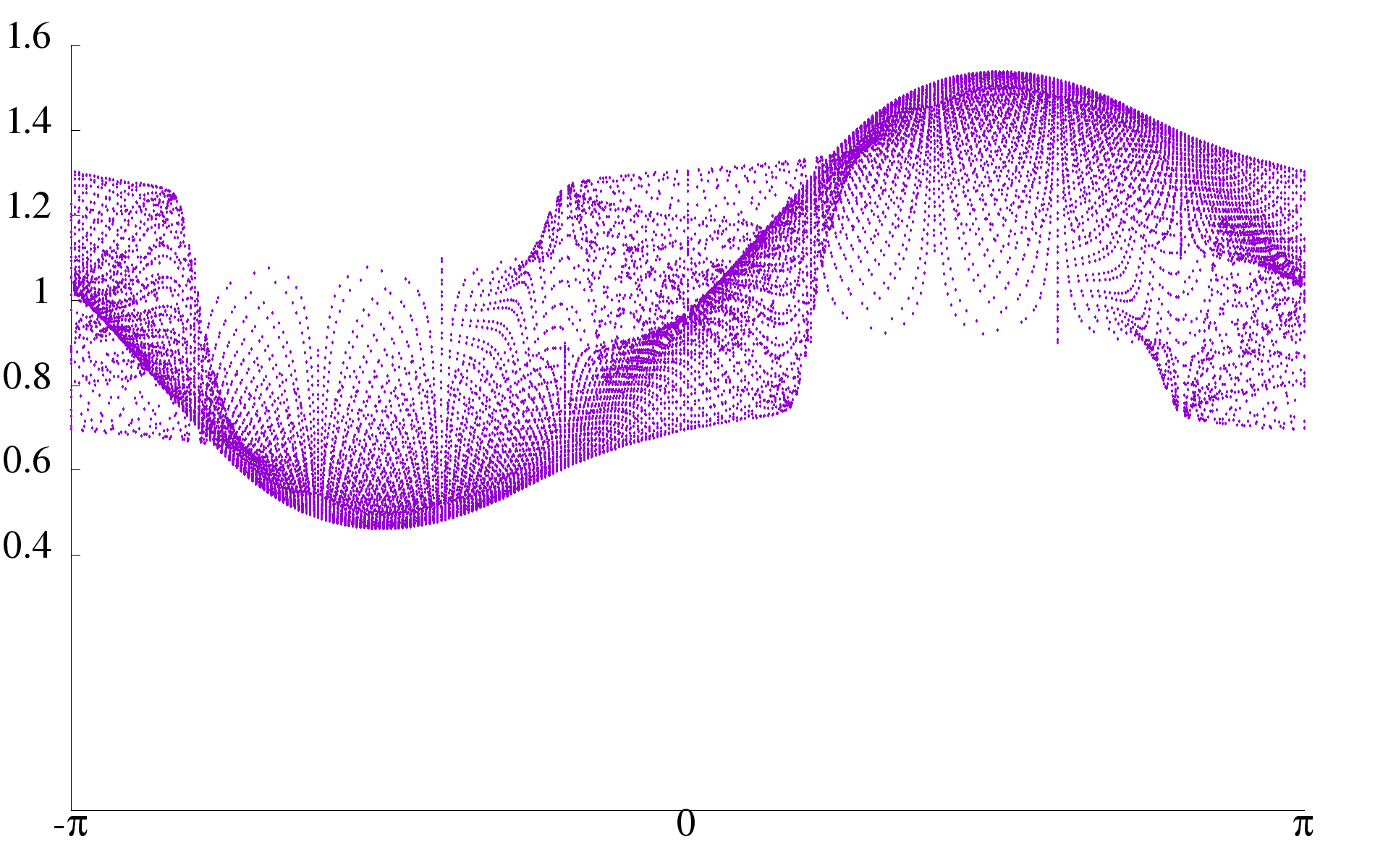}
  \caption{}
  \label{fig:MVSourceLim33D}
\end{subfigure}%
\begin{subfigure}{.5\textwidth}
  \centering
  \includegraphics[width=\textwidth]{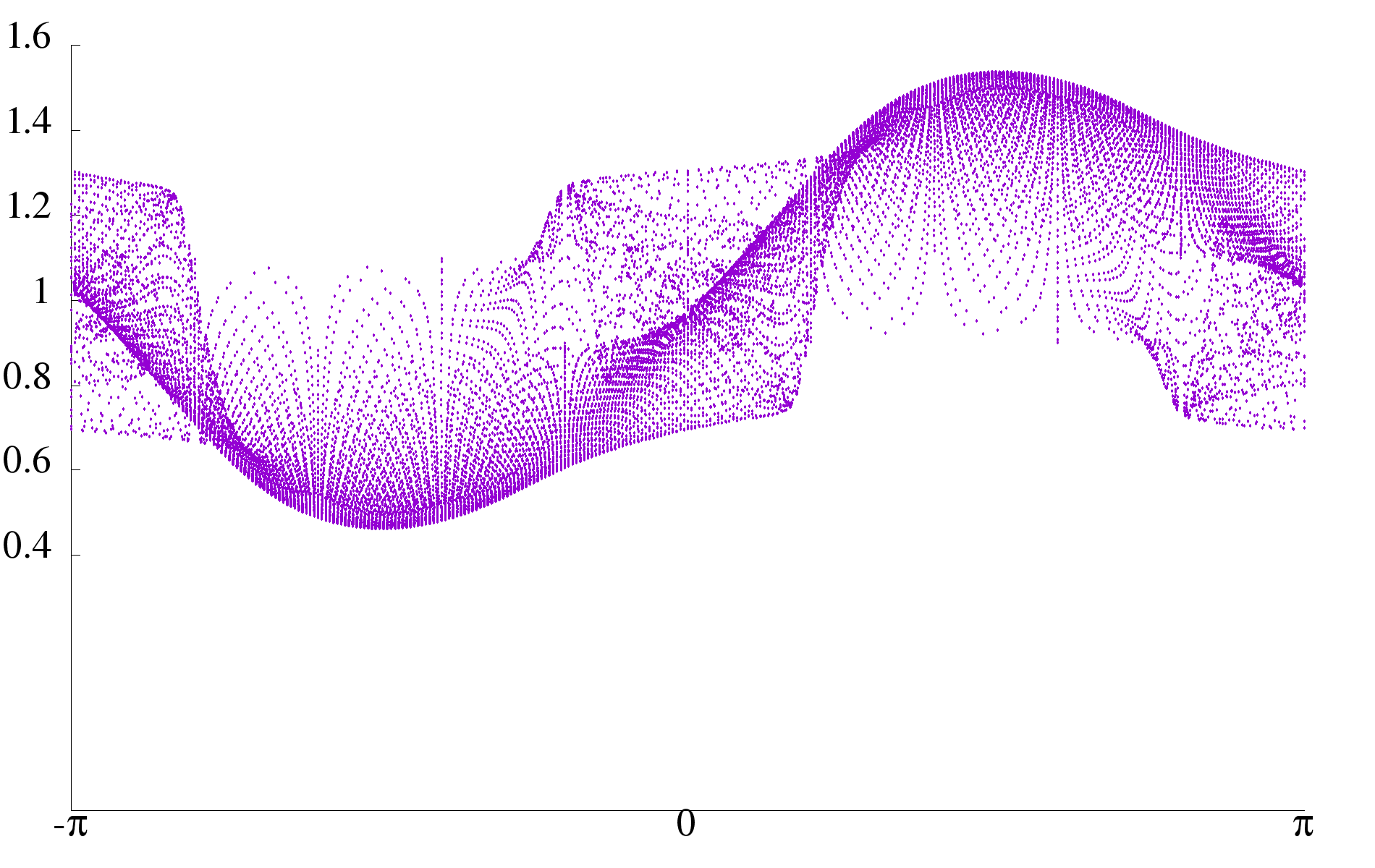}
  \caption{}
  \label{fig:MVSourceLim43D}
\end{subfigure}%
\caption{Moving vortices, a)-d) - unrotated grid $(\lambda, \theta)\in[-\pi; \pi]\times[-\pi/2; \pi/2]$, plot viewing angle $(90^\circ, 0^\circ)$:
                                   a) exact solution
                                   b) standart adjoint
                                   c) art. source adjoint, limiter ~ \cite{zalesak1979fully}, \cite{schar1996synchronous}
                                   d) art. source adjoint, limiter ~\cite{zalesak1979fully}, \cite{harris2011flux}
                                   }
\label{fig:MovingVortexAtT}
\end{figure}

\begin{figure}
\begin{subfigure}{.5\textwidth}
  \centering
  \includegraphics[width=\textwidth]{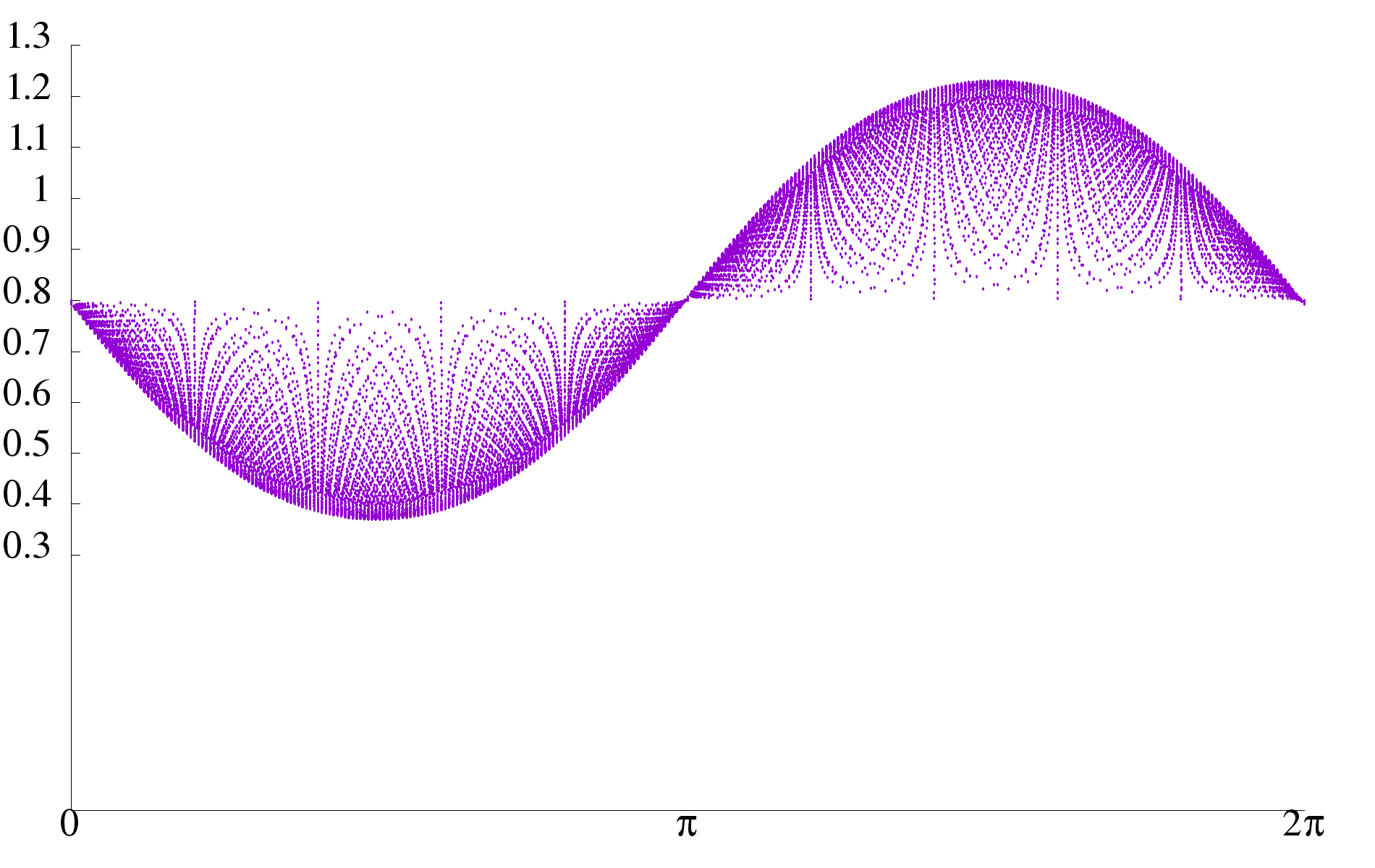}
  \caption{}
  \label{fig:PAMV}
\end{subfigure}%
\begin{subfigure}{.5\textwidth}
  \centering
  \includegraphics[width=\textwidth]{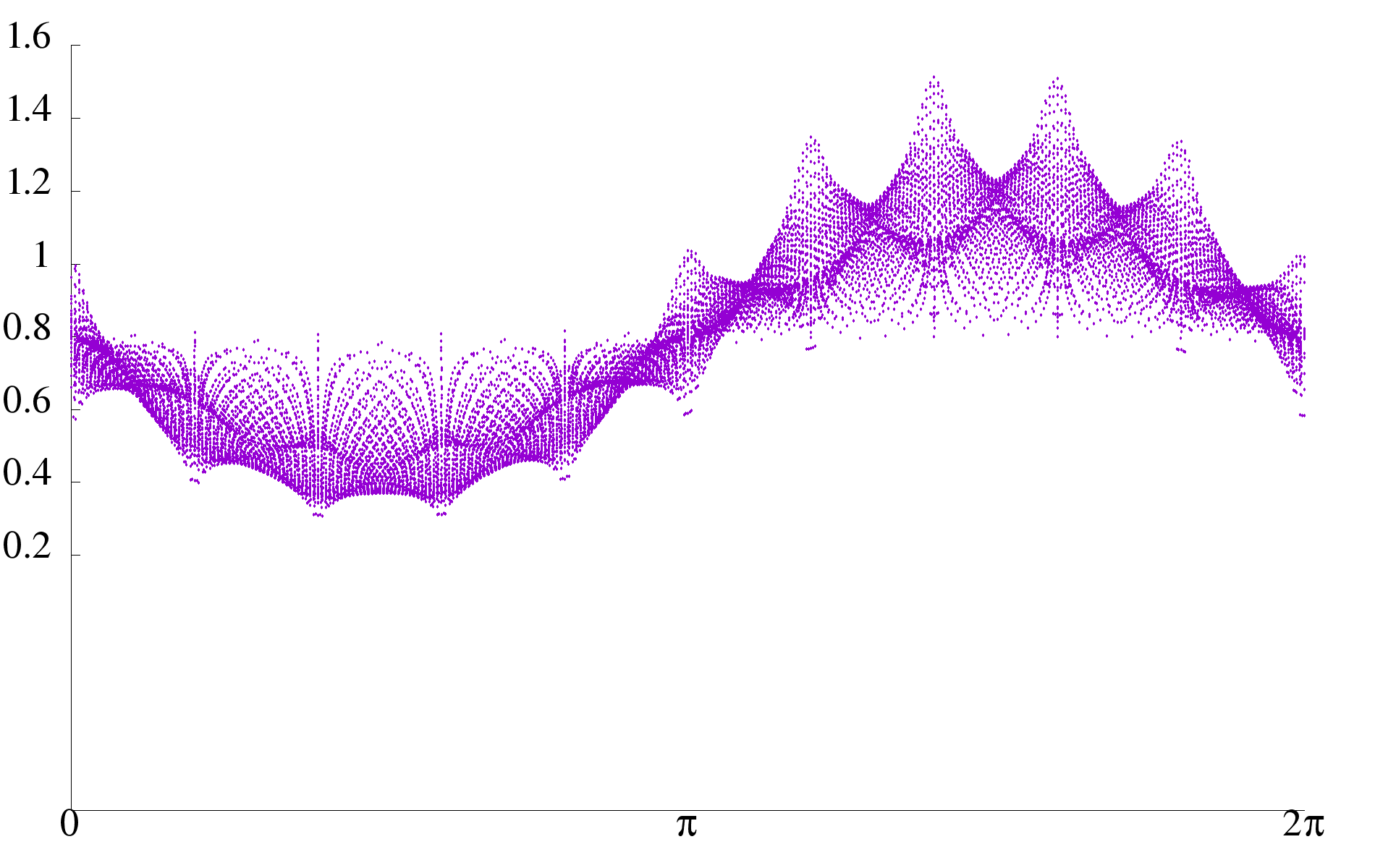}
  \caption{}
  \label{fig:PAMVAdjoint3D}
\end{subfigure}
\begin{subfigure}{.5\textwidth}
  \centering
  \includegraphics[width=\textwidth]{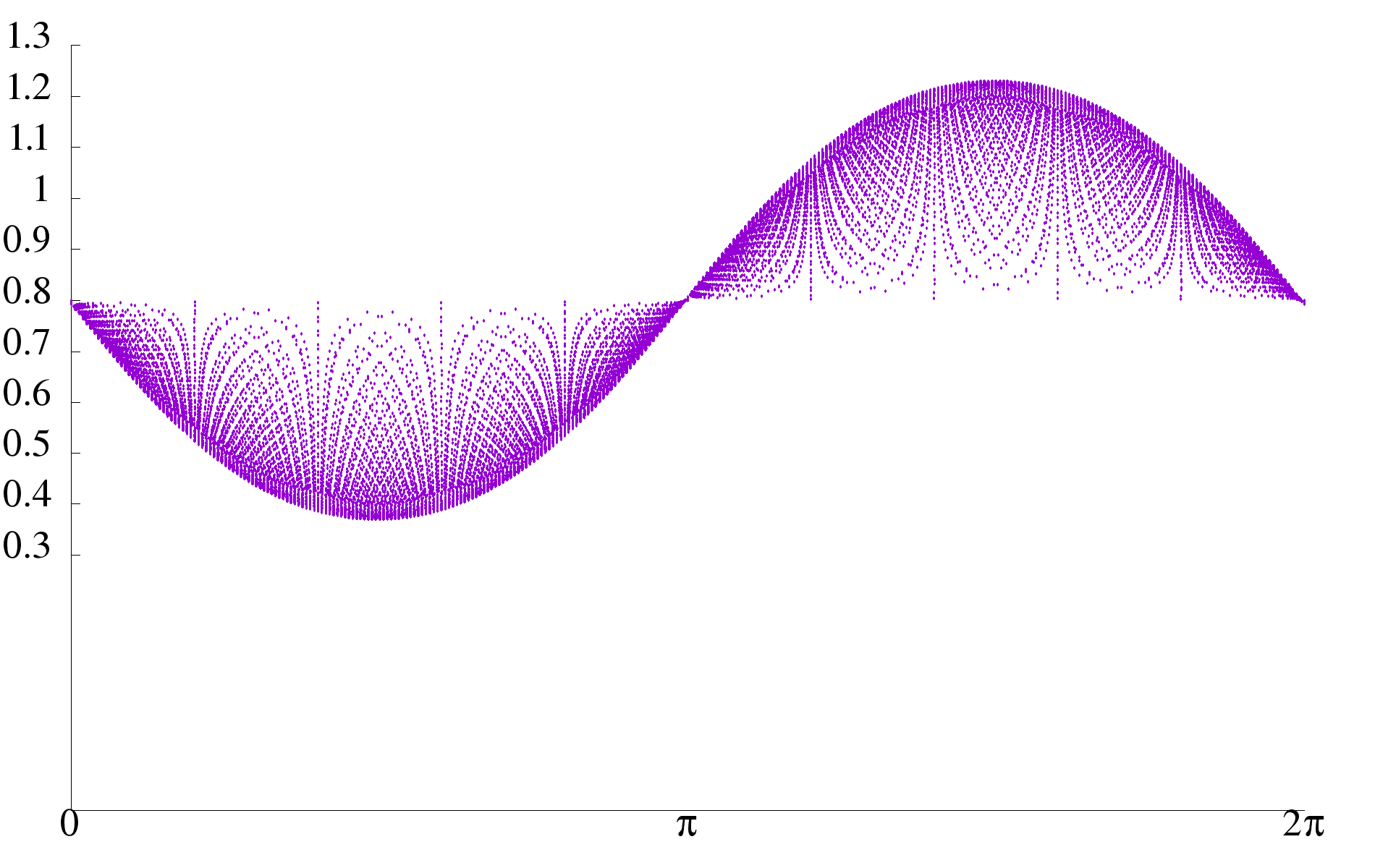}  
  \caption{}
  \label{fig:PAMVSourceLim33D}
\end{subfigure}%
\begin{subfigure}{.5\textwidth}
  \centering
  \includegraphics[width=\textwidth]{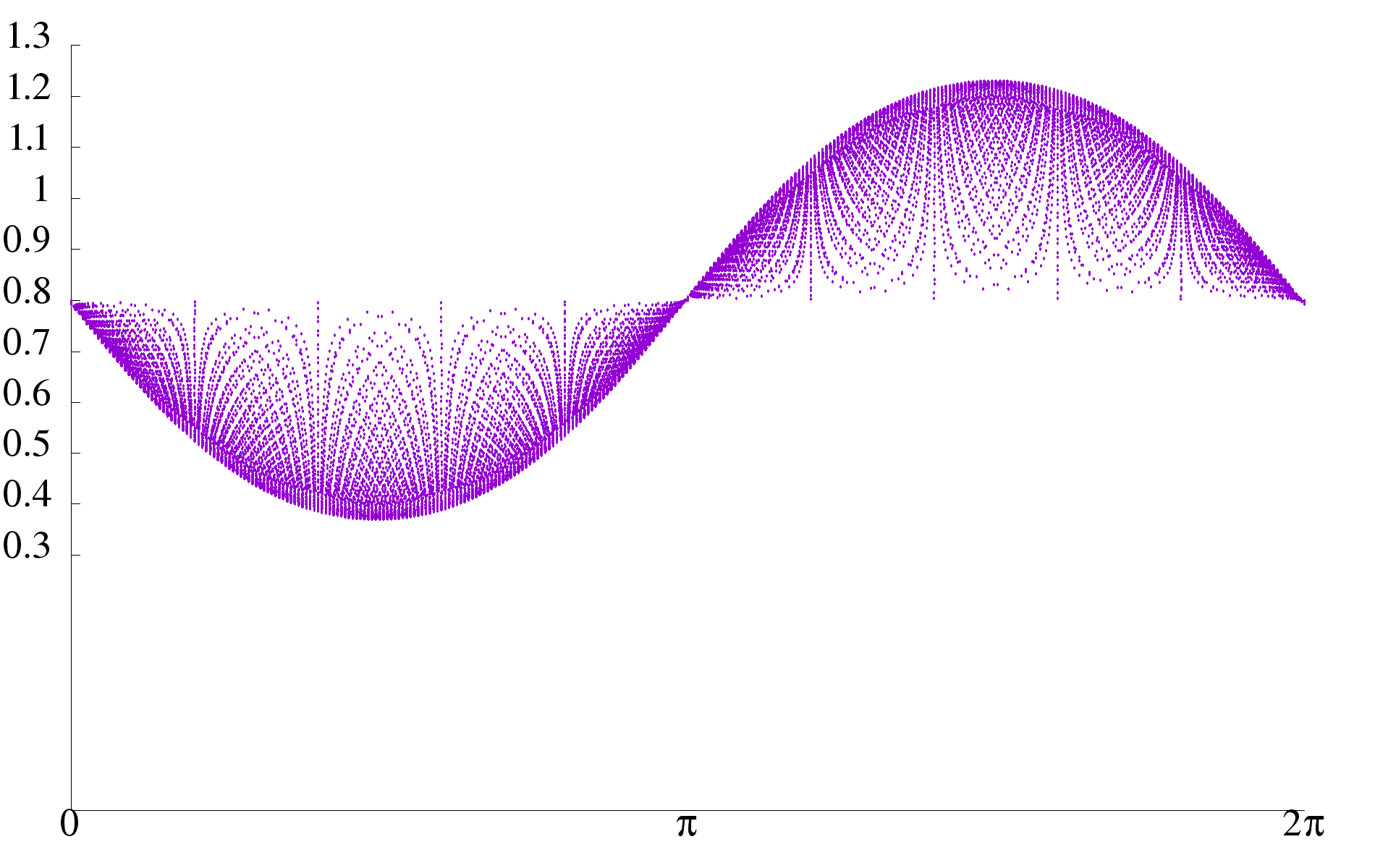} 
  \caption{}
  \label{fig:PAMVSourceLim43D}
\end{subfigure}
\caption{Solid body rotation wind field, moving vortex initial scalar field,a)-d) - plot viewing angle $(90^\circ, 0^\circ)$, $t=T/2$:
                                   a) ICON FFSL, limiter ~ \cite{zalesak1979fully}, \cite{schar1996synchronous}
                                   b) standart adjoint
                                   c) art. source adjoint, limiter ~ \cite{zalesak1979fully}, \cite{schar1996synchronous}
                                   d) art. source adjoint, limiter ~\cite{zalesak1979fully}, \cite{harris2011flux}
                                   }
\label{fig:MovingVortexAtT/2}
\end{figure}

    \clearpage
\subsubsection{Tables for advection tests}
\label{A:AdvTestTables}

\begin{table}[h!]\small
	\centering
	\caption{Solid body rotation, cosine bell, grid R2B04, $T=1036800[s]$}
        \scalebox{0.5}{
	 \begin{tabular}{c c c c c c c c}
			\hline
\textbf{Error norms} &\textbf{ICON-FFSL,}  &\textbf{Standard adjoint,} &\textbf{Art. source adjoint,} &\textbf{ICON-FFSL,}                                                    &\textbf{Art. source adjoint,}                                          &\textbf{ICON-FFSL,}                                              &\textbf{Art. source adjoint,} \\ 
                     &\textbf{no limiter} &\textbf{no limiter}       &\textbf{no limiter}          &\textbf{limiter \cite{zalesak1979fully}, \cite{schar1996synchronous}} &\textbf{limiter \cite{zalesak1979fully}, \cite{schar1996synchronous}} &\textbf{limiter \cite{zalesak1979fully}, \cite{harris2011flux}} &\textbf{limiter \cite{zalesak1979fully}, \cite{harris2011flux}} \\
\hline
\hline
$l_{1,rel}$      & 2.230227E-02  & 4.617754E-02  & 2.230230E-02  & 1.399967E-02  & 1.399982E-02  & 1.358807E-02  & 1.358786E-02 \\
$l_{2,rel}$      & 1.264162E-02  & 2.890380E-02  & 1.264118E-02  & 1.073500E-02  & 1.073535E-02  & 1.297653E-02  & 1.297682E-02 \\
$l_{\infty,rel}$ & 1.145905E-02  & 2.296593E-02  & 1.145946E-02  & 1.516152E-02  & 1.516201E-02  & 2.414110E-02  & 2.414264E-02 \\
$l_{1,abs}$      & 3.737112      & 7.697128      & 3.737118      & 2.330792      & 2.330815      & 2.247145      & 2.247110     \\
$l_{2,abs}$      & 1.245850E-01  & 2.838235E-01  & 1.245807E-01  & 1.055754E-01  & 1.055789E-01  & 1.269004E-01  & 1.269032E-01 \\
$l_{\infty,abs}$ & 1.134216E-02  & 2.273165E-02  & 1.134256E-02  & 1.500686E-02  & 1.500734E-02  & 2.389484E-02  & 2.389635E-02 \\
\hline
Undershoot       & 9993          & 10007         & 9993          &  0            & 0             &  0            & 89           \\
Minimum          & -1.081394E-02 & -1.164300E-02 & -1.081451E-02 &  0            & 0             &  0            & -5.023336E-11\\
Overshoot        & 0             & 2             & 0             &  4            & 4             &  4            & 4            \\
Maximum          & 9.897216E-01  & 9.952338E-01  & 9.897270E-01  &  9.909216E-01 & 9.909271E-01  &  9.905530E-01 & 9.905582E-01 \\
\hline
\hline
		\end{tabular}} 
	\label{tab:SolidBodyRotationcos}	
\end{table}			
\vfill					
\begin{table}[h!]
	\centering
	\caption{Solid body rotation, slotted cylinder, grid R2B04, $T=1036800[s]$}
        \scalebox{0.5}{
	 \begin{tabular}{c c c c c c c c}
			\hline
\textbf{Error norms} &\textbf{ICON-FFSL,}  &\textbf{Standard adjoint,} &\textbf{Art. source adjoint,} &\textbf{ICON-FFSL,}                                                    &\textbf{Art. source adjoint,}                                          &\textbf{ICON-FFSL,}                                              &\textbf{Art. source adjoint,} \\ 
                     &\textbf{no limiter} &\textbf{no limiter}       &\textbf{no limiter}          &\textbf{limiter \cite{zalesak1979fully}, \cite{schar1996synchronous}} &\textbf{limiter \cite{zalesak1979fully}, \cite{schar1996synchronous}} &\textbf{limiter \cite{zalesak1979fully}, \cite{harris2011flux}} &\textbf{limiter \cite{zalesak1979fully}, \cite{harris2011flux}} \\
\hline
\hline					
$l_{1,rel}$       & 3.069264E-01  & 3.171004E-01  & 3.069262E-01  & 2.522601E-01  & 2.522586E-01 & 2.599566E-01  & 2.599549E-01 \\
$l_{2,rel}$       & 2.673495E-01  & 2.693426E-01  & 2.673488E-01  & 2.676666E-01  & 2.676659E-01 & 2.736425E-01  & 2.736418E-01 \\
$l_{\infty,rel}$  & 7.106978E-01  & 6.999076E-01  & 7.107034E-01  & 8.149918E-01  & 8.150041E-01 & 8.057034E-01  & 8.057114E-01 \\ 
$l_{1,abs}$       & 3.157555E+02  & 3.263074E+02  & 3.157552E+02  & 2.593519E+02  & 2.593502E+02 & 2.669324E+02  & 2.669305E+02 \\
$l_{2,abs}$       & 8.582671      & 8.649268      & 8.582646      & 8.586382      & 8.586358     & 8.773138      & 8.773115 \\
$l_{\infty,abs}$  & 7.106978E-01  & 6.999076E-01  & 7.107034E-01  & 8.149918E-01  & 8.150041E-01 & 8.057034E-01  & 8.057114E-01 \\		
\hline
Undershoot      & 9666          & 9649          &  9666         &  0            & 0            & 0             & 117    \\
Minimum           & -2.736402E-01 & -2.657430E-01 & -2.736436E-01 &  0            & 0            & 0             & -8.511675E-10 \\
Overshoot      & 435           & 431           & 435           &  415          & 415          & 414           & 414    \\
Maximum           & 1.2450956     & 1.246481      & 1.245110      & 1.284667      & 1.284693     & 1.327834      & 1.327850  \\
\hline	
\hline
		\end{tabular}} 
	\label{tab:SolidBodyRotationcyl}	
\end{table}
\vfill					
\begin{table}[h!]
	\centering
	\caption{Deformational flow, cosine bells, non-divergent velocity vector, grid R2B04, $T=1036800[s]$}
        \scalebox{0.5}{
	 \begin{tabular}{c c c c c c c c}
			\hline
\textbf{Error norms} &\textbf{ICON-FFSL,}  &\textbf{Standard adjoint,} &\textbf{Art. source adjoint,} &\textbf{ICON-FFSL,}                                                    &\textbf{Art. source adjoint,}                                          &\textbf{ICON-FFSL,}                                              &\textbf{Art. source adjoint,} \\ 
                     &\textbf{no limiter} &\textbf{no limiter}       &\textbf{no limiter}          &\textbf{limiter \cite{zalesak1979fully}, \cite{schar1996synchronous}} &\textbf{limiter \cite{zalesak1979fully}, \cite{schar1996synchronous}} &\textbf{limiter \cite{zalesak1979fully}, \cite{harris2011flux}} &\textbf{limiter \cite{zalesak1979fully}, \cite{harris2011flux}} \\
\hline
\hline
$l_{1,rel}$      & 3.256092E-02  & 4.011542E-02  & 3.256965E-02  & 2.342655E-02  & 2.341489E-02 & 1.895446E-02 & 1.894835E-02 \\
$l_{2,rel}$      & 2.464795E-02  & 2.847823E-02  & 2.463155E-02  & 2.138599E-02  & 2.136528E-02 & 1.901537E-02 & 1.899423E-02 \\
$l_{\infty,rel}$ & 3.967162E-02  & 4.035984E-02  & 3.958678E-02  & 3.179198E-02  & 3.170231E-02 & 2.966709E-02 & 2.957616E-02 \\	
$l_{1,abs}$      & 2.438983E+01  & 2.999117E+01  & 2.439604E+01  & 1.759734E+01  & 1.758865E+01 & 1.427368E+01 & 1.426911E+01 \\
$l_{2,abs}$      & 5.131036E-01  & 5.930658E-01  & 5.127594E-01  & 4.469183E-01  & 4.464925E-01 & 3.979740E-01 & 3.975392E-01 \\
$l_{\infty,abs}$ & 3.966944E-02  & 4.035763E-01  & 3.958461E-02  & 3.179024E-02  & 3.170058E-02 & 2.966546E-02 & 2.957454E-02 \\
\hline
Undershoot       & 9201          & 9209          & 9208          & 0             & 0            & 0            & 1167          \\
Minimum          & -2.355757E-02 & -2.397034E-02 & -2.360643E-02 & 0             & 0            & 0            & -8.922257E-10 \\
Overshoot        & 0             & 1             & 0             & 0             & 0            & 0            & 0             \\
Maximum          & 9.9958314E-01 & 1.004633      & 9.994996E-01  & 9.988715E-01  & 9.987885E-01 & 9.991847E-01 & 9.991020E-01  \\
\hline
\hline

		\end{tabular}} 
	\label{tab:DeformationalFlow1cos}	
\end{table}			
	
\vfill					
\begin{table}[h!]
	\centering
	\caption{Deformational flow, slotted cylinders, non-divergent velocity vector, grid R2B04, $T=1036800[s]$}
        \scalebox{0.5}{
	 \begin{tabular}{c c c c c c c c}
			\hline
\textbf{Error norms} &\textbf{ICON-FFSL,}  &\textbf{Standard adjoint,} &\textbf{Art. source adjoint,} &\textbf{ICON-FFSL,}                                                    &\textbf{Art. source adjoint,}                                          &\textbf{ICON-FFSL,}                                              &\textbf{Art. source adjoint,} \\ 
                     &\textbf{no limiter} &\textbf{no limiter}       &\textbf{no limiter}          &\textbf{limiter \cite{zalesak1979fully}, \cite{schar1996synchronous}} &\textbf{limiter \cite{zalesak1979fully}, \cite{schar1996synchronous}} &\textbf{limiter \cite{zalesak1979fully}, \cite{harris2011flux}} &\textbf{limiter \cite{zalesak1979fully}, \cite{harris2011flux}} \\
\hline
\hline			
$l_{1,rel}$      & 2.890181E-01   & 2.888711E-01  & 2.889820E-01  & 2.780968E-01  & 2.780884E-01  & 2.738582E-01  & 2.738496E-01 \\
$l_{2,rel}$      & 2.940940E-01   & 2.942568E-01  & 2.940659E-01  & 3.042498E-01  & 3.042283E-01  & 3.015322E-01  & 3.015094E-01 \\
$l_{\infty,rel}$ & 9.425960E-01   & 9.190518E-01  & 9.412042E-01  & 9.509492E-01  & 9.504477E-01  & 9.468321E-01  & 9.464949E-01 \\
$l_{1,abs}$      & 6.834999E+02   & 6.832123E+02  & 6.834093E+02  & 6.576075E+02  & 6.575934E+02  & 6.475465E+02  & 6.475301E+02 \\
$l_{2,abs}$      & 1.430822E+01   & 1.431595E+01  & 1.430670E+01  & 1.480470E+01  & 1.480354E+01  & 1.467501E+01  & 1.467376E+01 \\
$l_{\infty,abs}$ & 9.425960E-01   & 9.190518E-01  & 9.412042E-01  & 9.509492E-01  & 9.504477E-01  & 9.468321E-01  & 9.464949E-01 \\
\hline
Undershoot       & 8886           & 8879          & 8890          &  0            & 0             & 0             & 181           \\
Minimum          & -9.581544E-02  & -1.741285E-01 & -9.600165E-02 &  0            & 0             & 0             & -1.739492E-09 \\
Overshoot        & 910            & 910           & 909           & 796           & 796           & 837           & 837            \\
Maximum          & 1.193710       & 1.202044      & 1.193843      & 1.169024      & 1.169178      & 1.183066      & 1.183258       \\
\hline
\hline   
		\end{tabular}} 
	\label{tab:DeformationalFlow1cyl}	
\end{table}
\vfill					
\begin{table}[h!]
	\centering
	\caption{Deformational flow, cosine bells, divergent velocity vector, grid R2B04, $T=1036800[s]$}
        \scalebox{0.5}{
		\begin{tabular}{c c c c c}
			\hline
			\textbf{Error norms} &\textbf{Standard adjoint,} &\textbf{Art. source adjoint,} &\textbf{Art. source adjoint,}                                          &\textbf{Art. source adjoint,} \\
			                     &\textbf{no limiter}       &\textbf{no limiter}          &\textbf{limiter \cite{zalesak1979fully}, \cite{schar1996synchronous}} &\textbf{limiter \cite{zalesak1979fully}, \cite{harris2011flux}} \\
			\hline
			\hline  
			$l_{1,rel}$      & 4.029138E-002  & 3.311189E-002  & 1.753717E-002  & 1.627544E-002 \\
			$l_{2,rel}$      & 3.351494E-002  & 2.831085E-002  & 1.925266E-002  & 1.848417E-002 \\
			$l_{\infty,rel}$ & 5.717197E-002  & 4.793155E-002  & 3.574307E-002  & 3.388867      \\	
			$l_{1,abs}$      & 3.013273E+001  & 2.479282E+001  & 1.305616E+001  & 1.210955E+001 \\
			$l_{2,abs}$      & 6.964130E-001  & 5.885914E-001  & 3.994988E-001  & 3.835561E-001 \\
			$l_{\infty,abs}$ & 5.710458E-002  & 4.787505E-002  & 3.570094E-002  & 3.384873E-001\\		
			\hline
			Undershoot       & 9187           & 9189           & 0              & 7135 \\
			Minimum          & -3.020335E-002 & -2.661911E-002 & 0              & -1.572671E-008 \\
			Overshoot        & 0              & 0              & 0              & 0   \\
			Maximum          & 9.955241E-001  & 9.937946E-001  & 9.933643E-001  & 9.933694E-001 \\
			\hline
			\hline
		\end{tabular}}  
	\label{tab:DeformationalFlow2cos}	
\end{table}	
\vfill				
\begin{table}[h!]
	\centering
	\caption{Deformational flow, slotted cylinders, divergent velocity vector, grid R2B04, $T=1036800[s]$}
        \scalebox{0.5}{
		\begin{tabular}{c c c c c}
			\hline
			\textbf{Error norms} &\textbf{Standard adjoint,} &\textbf{Art. source adjoint,} &\textbf{Art. source adjoint,}                                          &\textbf{Art. source adjoint,} \\
			                     &\textbf{no limiter}       &\textbf{no limiter}          &\textbf{limiter \cite{zalesak1979fully}, \cite{schar1996synchronous}} &\textbf{limiter \cite{zalesak1979fully}, \cite{harris2011flux}} \\
			\hline
			\hline 		 	
			$l_{1,rel}$      & 3.201877E-001  & 3.072660E-01  & 2.959957E-01  & 2.911967E-01 \\
			$l_{2,rel}$      & 3.004045E-001  & 2.950234E-01  & 3.033951E-01  & 3.008232E-01 \\
			$l_{\infty,rel}$ & 8.202367E-001  & 8.689133E-01  & 8.527525E-01  & 8.511831E-01 \\
			$l_{1,abs}$      & 6.792025E+002  & 6.511061E+02  & 6.267315E+02  & 6.163679E+02 \\
			$l_{2,abs}$      & 1.383153E+001  & 1.357617E+01  & 1.396507E+01  & 1.384493E+01 \\
			$l_{\infty,abs}$ & 8.202369E-001  & 8.689133E-01  & 8.527525E-01  & 8.511831E-01 \\
			\hline
			Undershoot       & 8935           & 8867           & 14             & 7445 \\
			Minimum          & -1.334238E-001 & -1.173093E-01  & -4.796964E-15  & -4.019092E-09\\
                        Overshoot        & 858            & 839            & 791            & 801\\
			Maximum          & 1.200032       & 1.185067       & 1.177634       & 1.179568  \\
			\hline
			\hline
		\end{tabular}}  
	\label{tab:DeformationalFlow2cyl}	
\end{table}
\vfill	
\begin{table}[h!]
	\centering
	 \caption{Moving vortices, grid R2B04, $T=1036800[s]$}
         \scalebox{0.5}{
            \begin{tabular}{c c c c c}
	      \hline	
	      \textbf{Error norms} & \textbf{Standard adjoint,} &\textbf{Art. source adjoint,} & \textbf{Art. source adjoint,} & \textbf{Art. source adjoint,} \\ 
	                           & \textbf{no limiter}        & \textbf{no limiter}  & \textbf{limiter \cite{zalesak1979fully}, \cite{schar1996synchronous}} & \textbf{limiter \cite{zalesak1979fully}, \cite{harris2011flux}} \\
	      \hline
	      \hline
	      $l_{1,rel}$       & 2.412546E-02 & 1.411150E-03 & 1.411459E-03 & 1.411150E-03 \\
	      $l_{2,rel}$       & 4.083858E-02 & 3.874989E-03 & 3.875515E-03 & 3.874989E-03 \\
 	      $l_{\infty,rel}$  & 2.048717E-01 & 3.017727E-02 & 3.017801E-02 & 3.017727E-02 \\ 
              $l_{1,abs}$       & 5.084637E+02 & 2.881504E+01 & 2.882144E+01 & 2.881504E+01 \\
	      $l_{2,abs}$       & 6.294867     & 5.789490     & 5.790304     & 5.789490E-01 \\
	      $l_{\infty,abs}$  & 3.148948E-01 & 4.638326E-02 & 4.638440E-02 & 4.638326E-02 \\	  
	      \hline
	      Undershoot        & 44          & 1            & 1            & 1            \\
	      Minimum           & 4.483181E-01 & 4.629725E-01 & 4.629725E-01 & 4.629725E-01 \\
	      Overshoot         & 259          & 1            & 1            & 1            \\
	      Maximum           & 1.613474     & 1.537027     & 1.537027     & 1.537027     \\
	      \hline
	      \hline	
            \end{tabular}} 
	\label{tab:MovingVortices}	
   \end{table}
\vfill	   
   \clearpage
\subsection{Numerical results of data assimilation tests}
\label{A:AssimNumerical}

\vspace{-3cm}
\subsubsection{Figures for data assimilation tests}
\label{A:DataAssimTestPlots}

\begin{figure}[b]
\begin{subfigure}{0.5\textwidth}
  \centering
  \includegraphics[width=\textwidth]{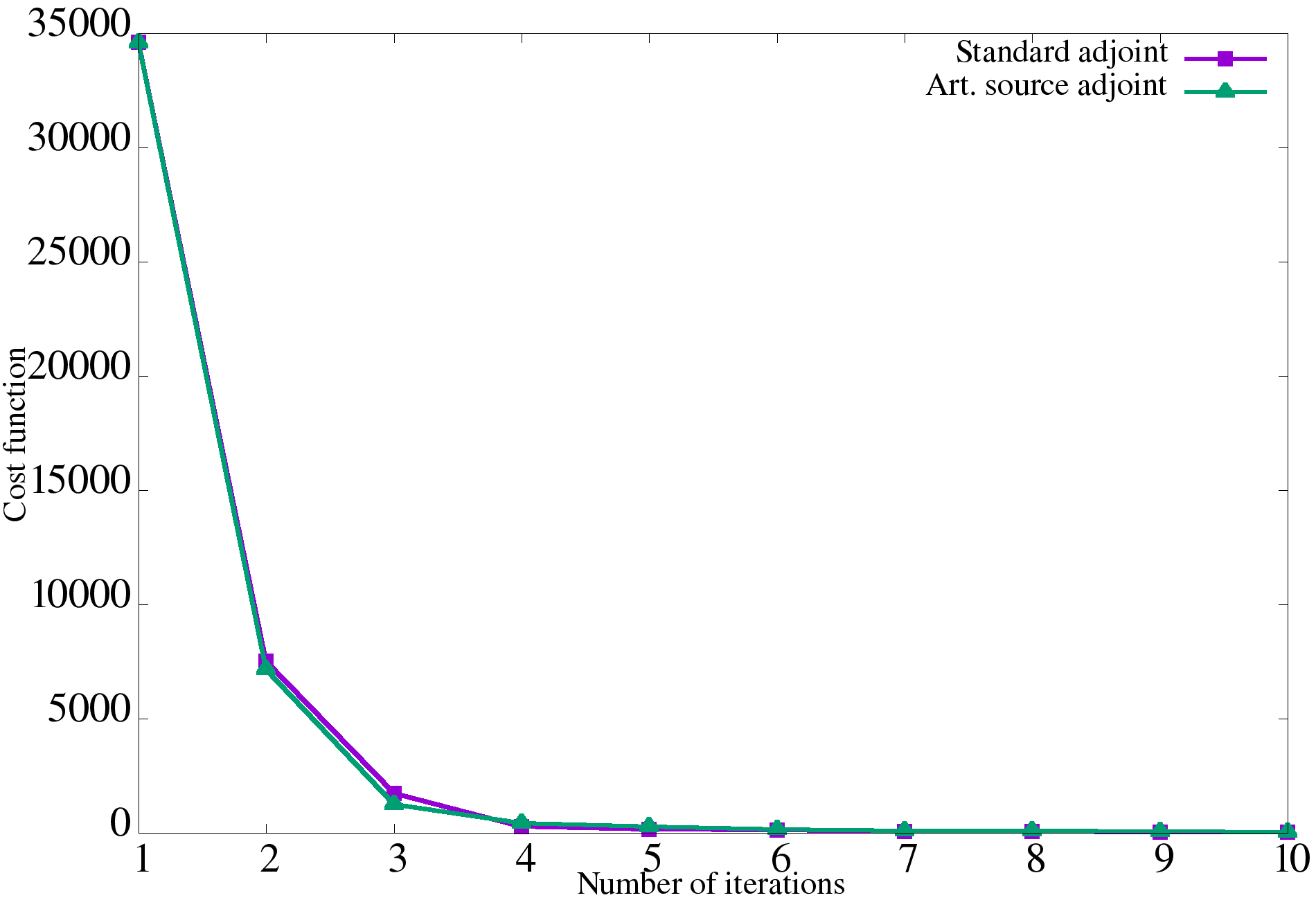}
  \caption{}
  \label{fig:DF3cos5120_cost_it300_1-10}
\end{subfigure}%
\begin{subfigure}{0.5\textwidth}
  \centering
  \includegraphics[width=\textwidth]{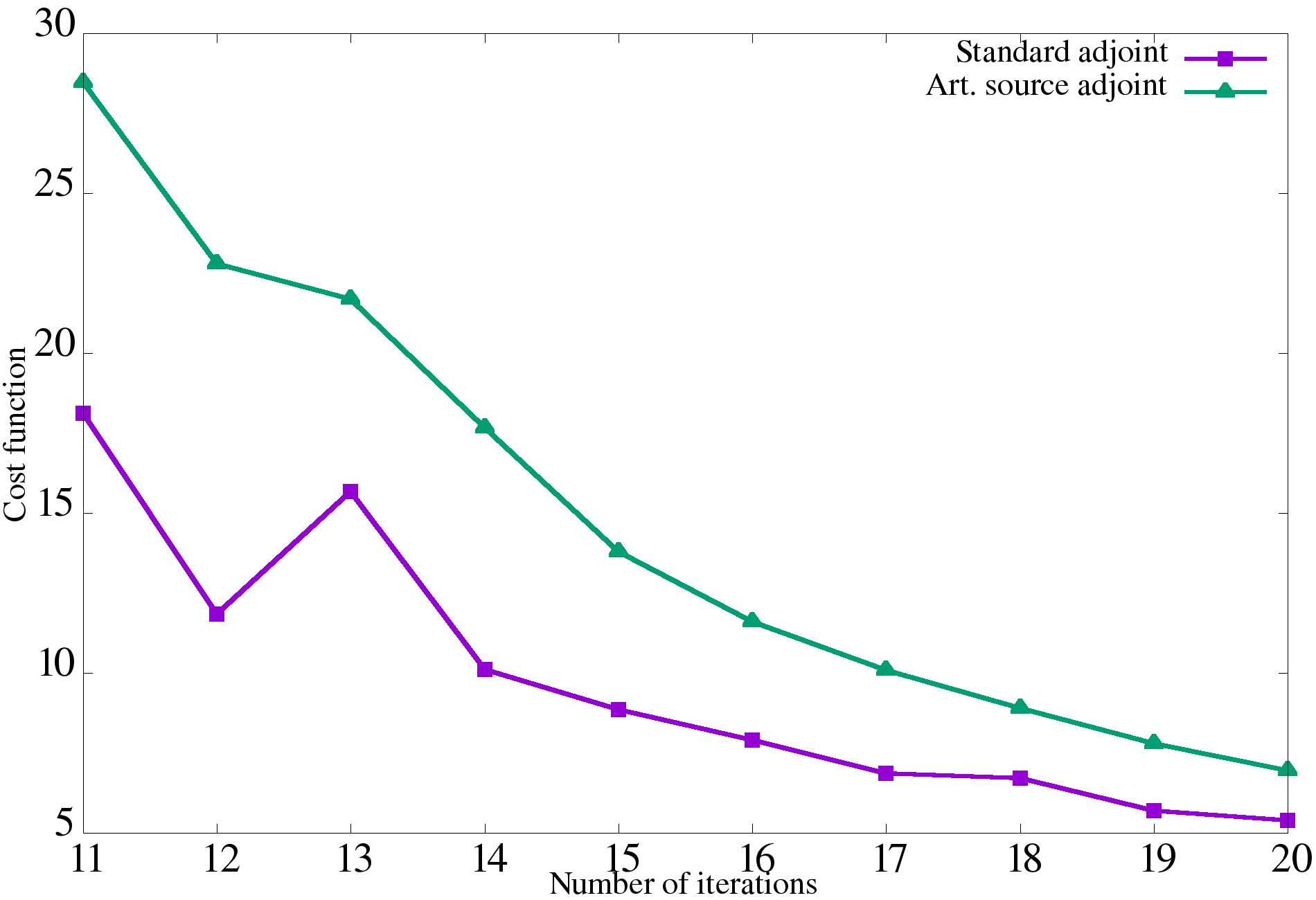}
  \caption{}
  \label{fig:DF3cos5120_cost_it300_11-20}
\end{subfigure}
\begin{subfigure}{0.5\textwidth}
  \centering
  \includegraphics[width=\textwidth]{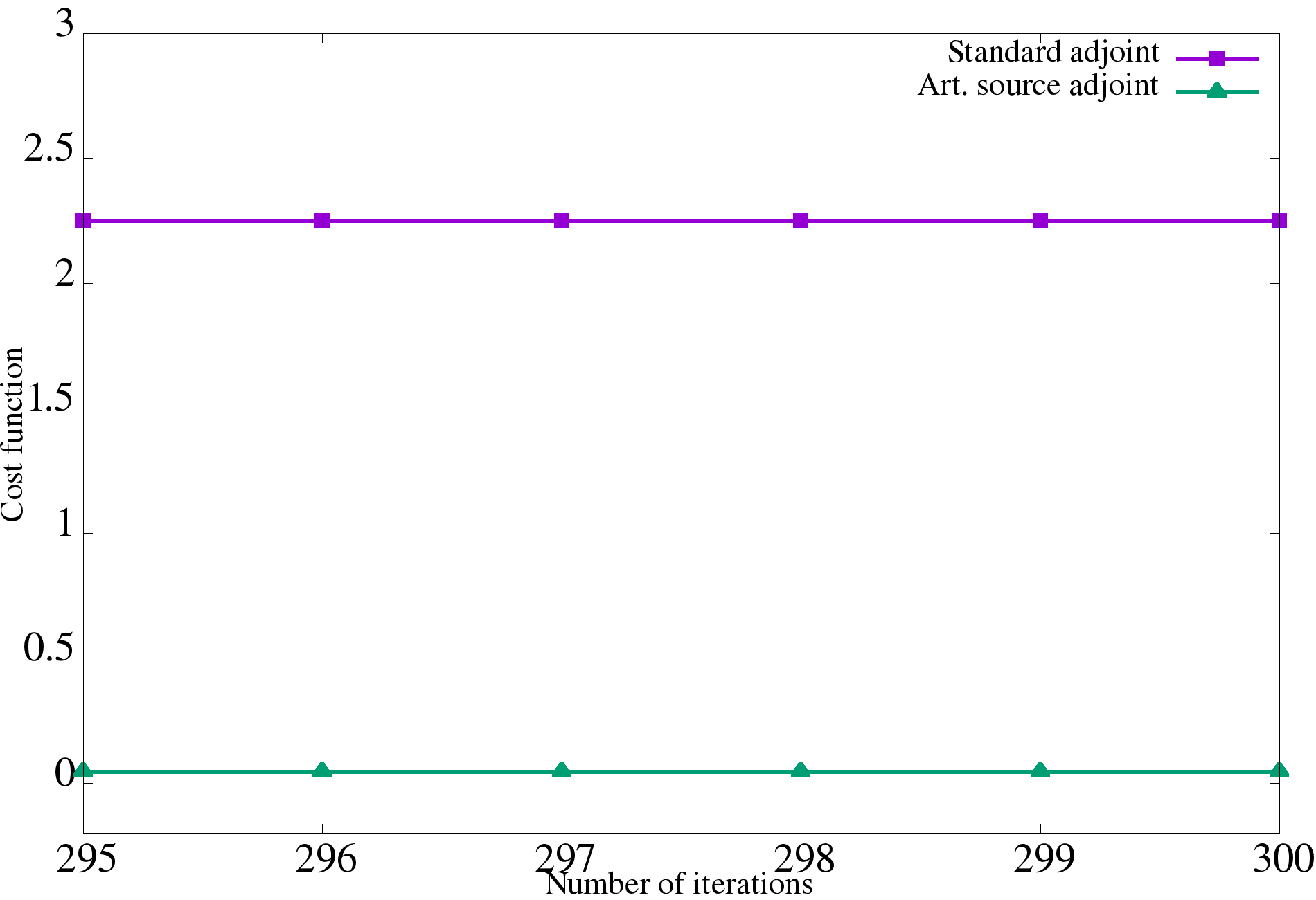}
  \caption{}
  \label{fig:DF3cos5120_cost_it300_295-300}
\end{subfigure}
\caption{Effect of number of iterations, deformational flow with cosine bells, $R2B4$ grid, 5120 observation points, number of iterations on abscissa, cost function on ordinate 
                                           a) iterations from 1 to 10 
                                           b) iterations from 11 to 20
                                           c) iterations from 295 to 300
                                           }
\label{fig:DFcos5120_it300}
\end{figure}

\begin{figure}[h]
\begin{subfigure}{0.5\textwidth}
  \centering
  \includegraphics[width=\textwidth]{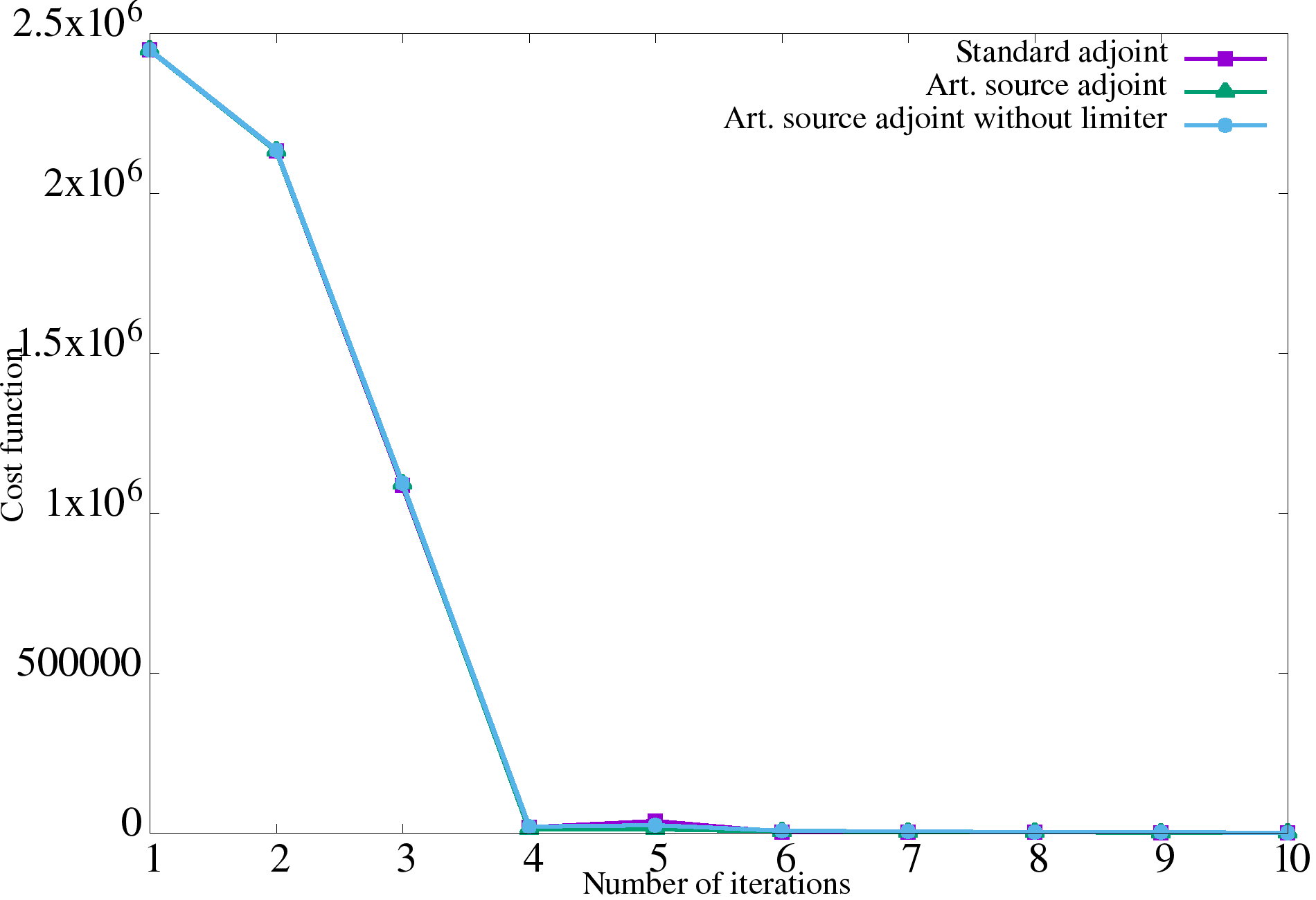}
  \caption{}
  \label{fig:MV5120_cost_it300_1-10}
\end{subfigure}%
\begin{subfigure}{0.5\textwidth}
  \centering
  \includegraphics[width=\textwidth]{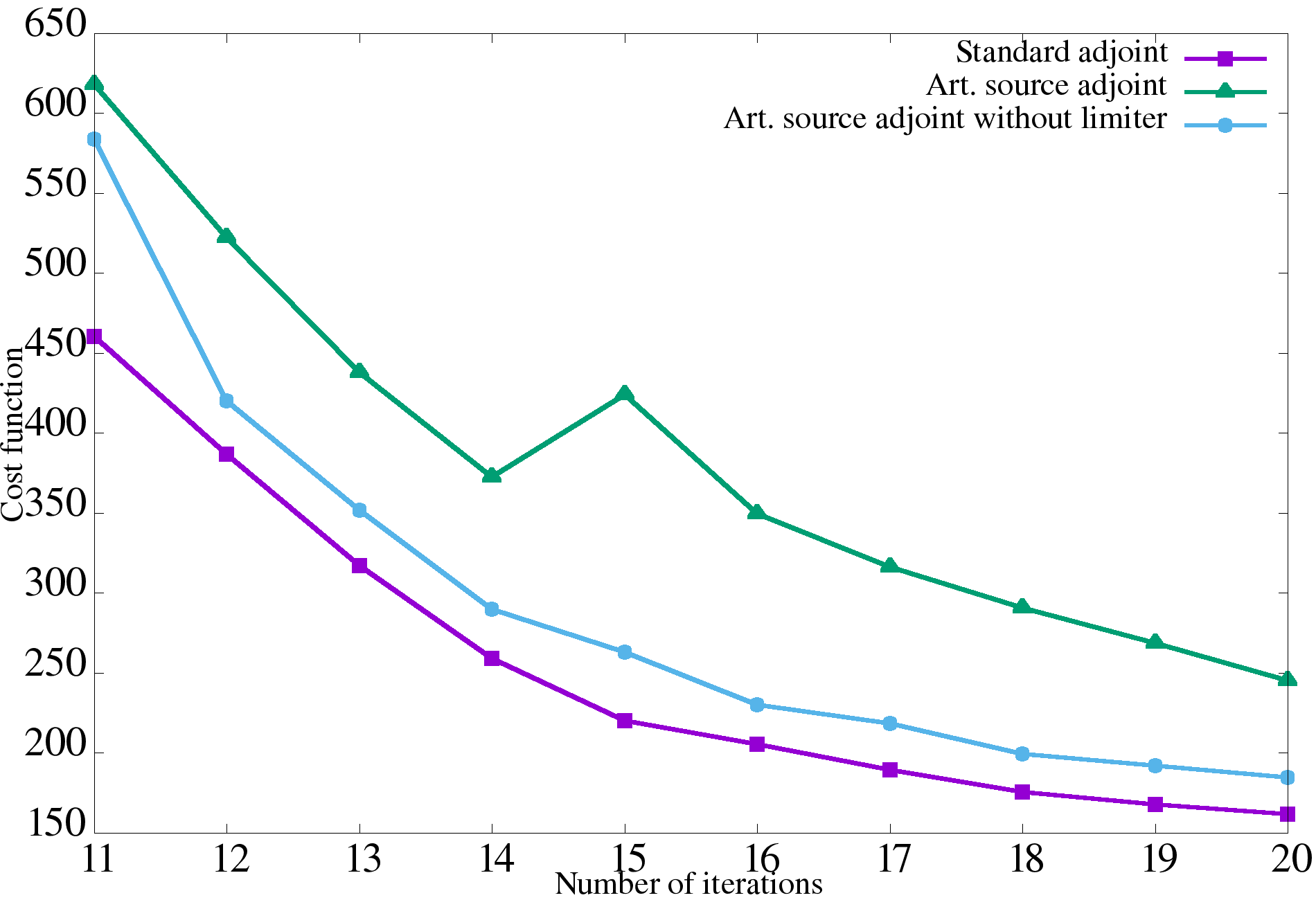}
  \caption{}
  \label{fig:MV5120_cost_it300_11-20}
\end{subfigure}
\begin{subfigure}{0.5\textwidth}
  \centering
  \includegraphics[width=\textwidth]{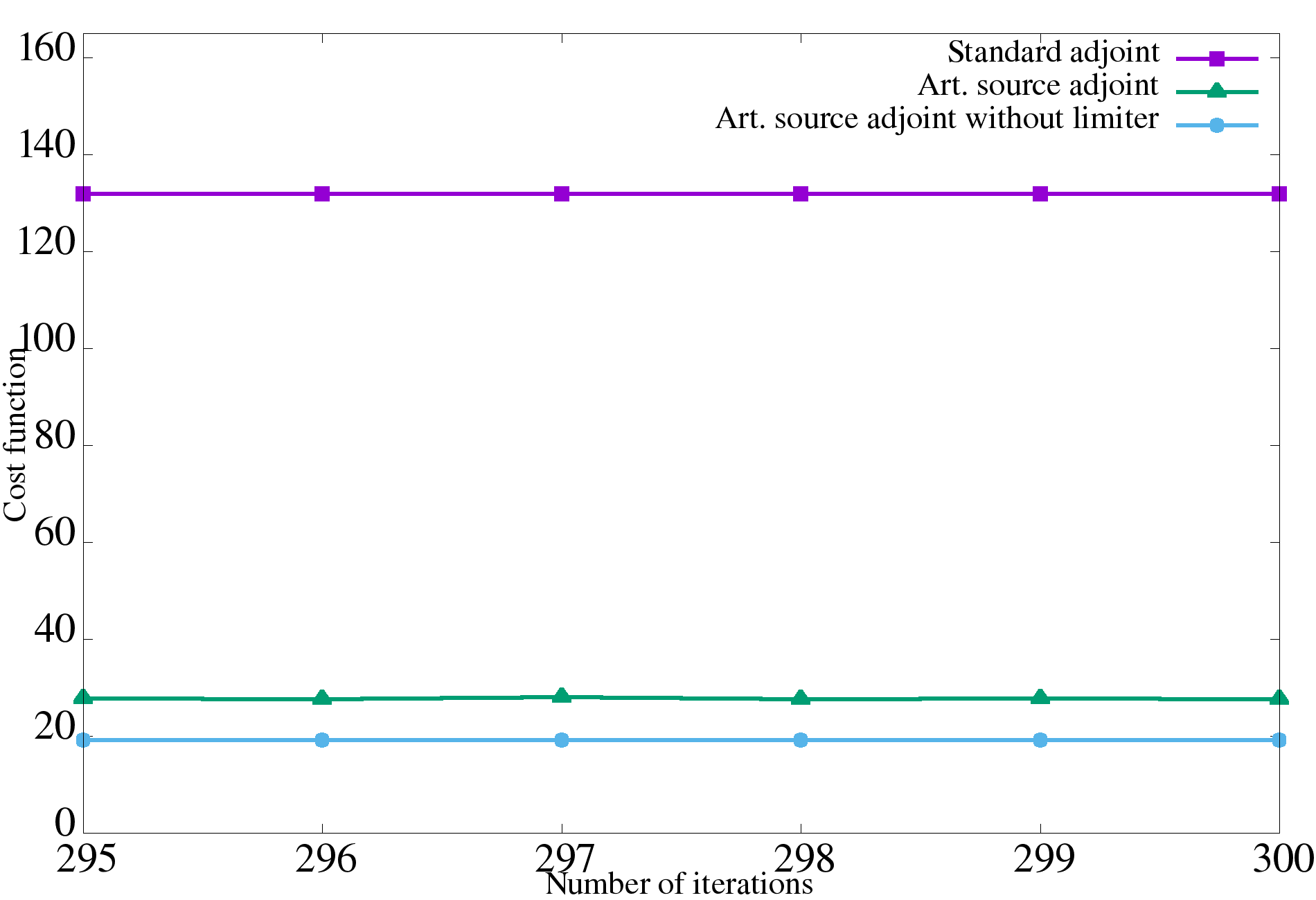}
  \caption{}
  \label{fig:MV5120_cost_it300_295-300}
\end{subfigure}%
\begin{subfigure}{.5\textwidth}
  \centering
  \includegraphics[width=\textwidth]{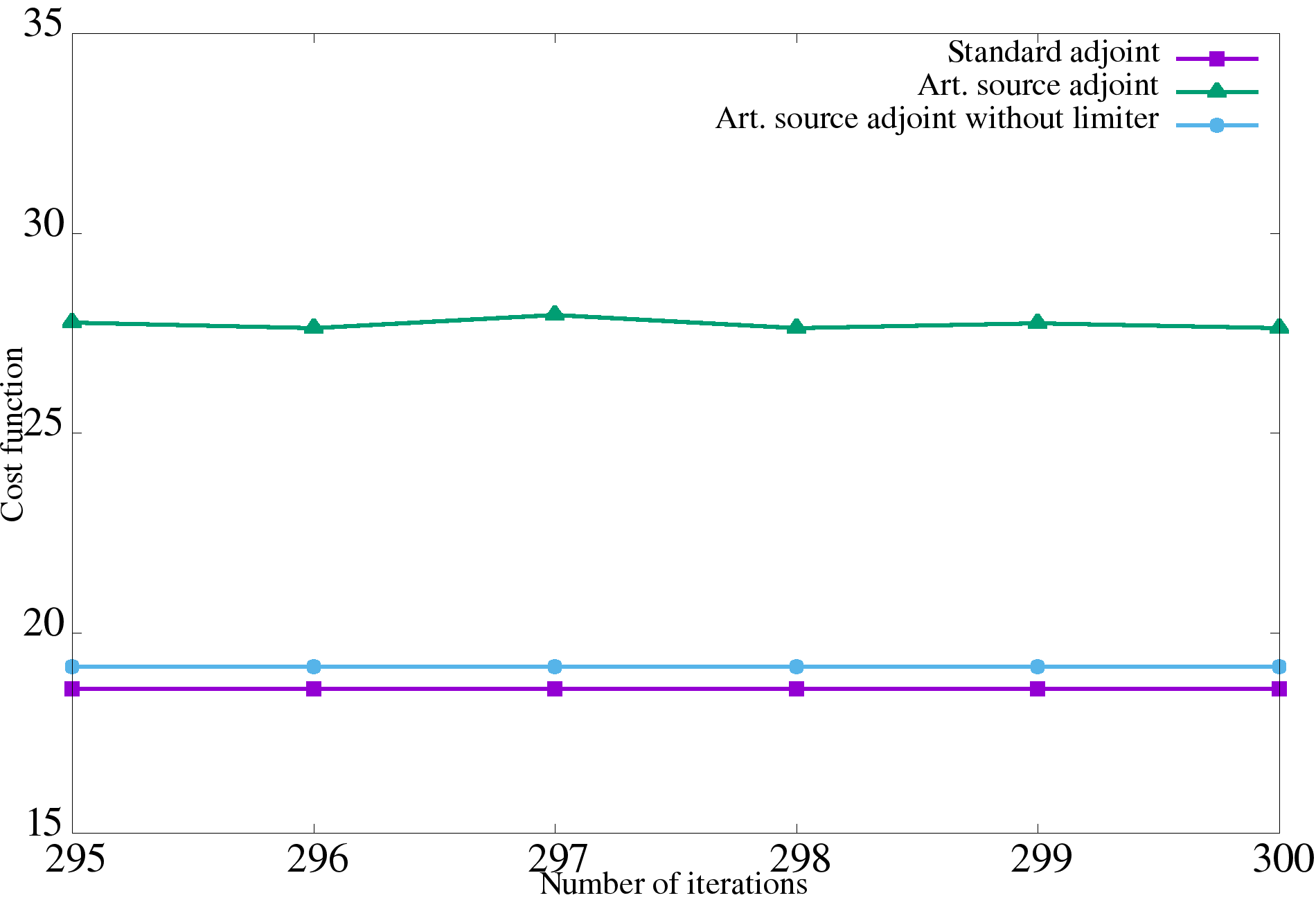}
  \caption{}
  \label{fig:MV5120_obscost_it300_295-300}
\end{subfigure}
\begin{subfigure}{.5\textwidth}
  \centering
  \includegraphics[width=\textwidth]{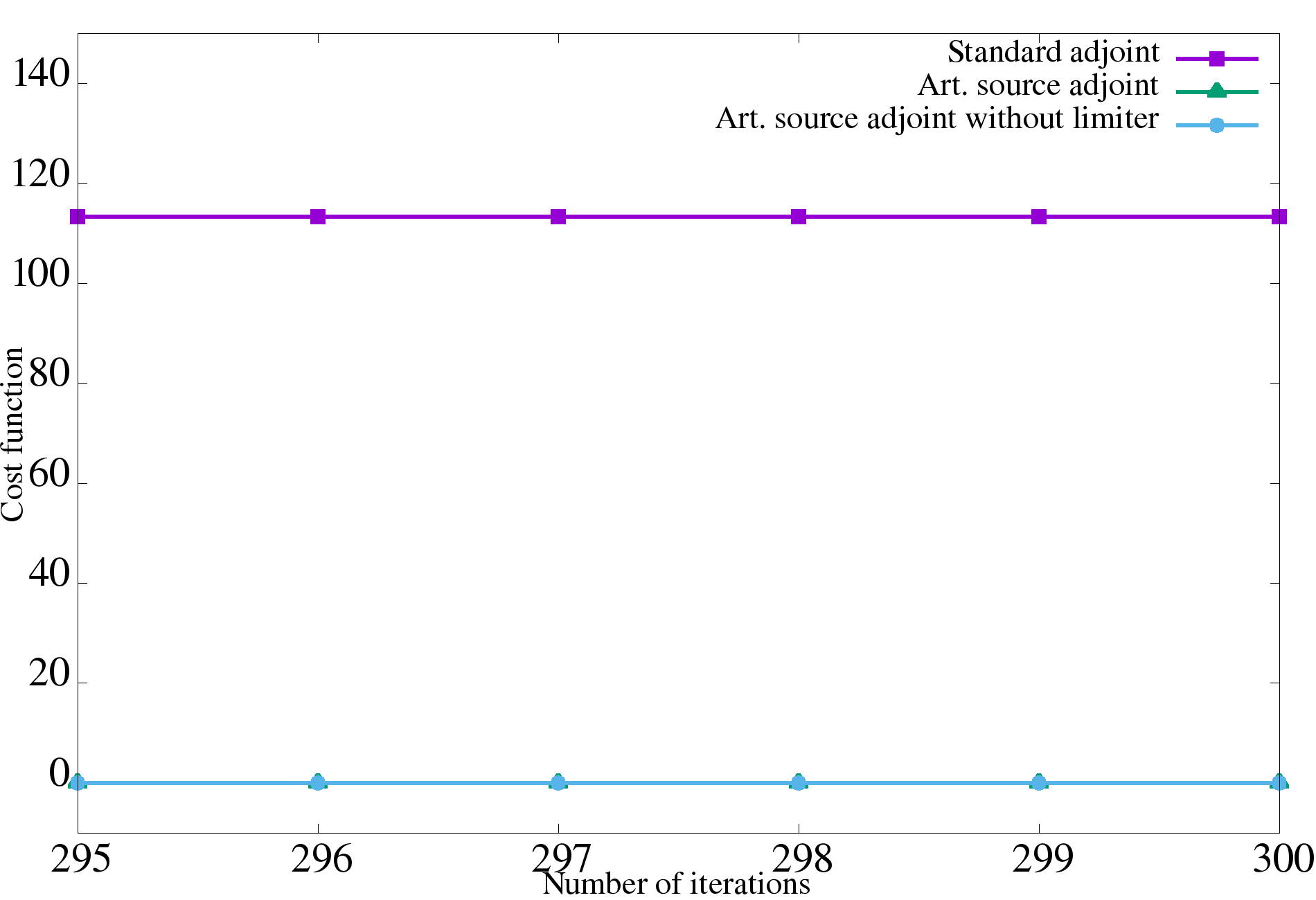}
  \caption{}
  \label{fig:MV5120_bcost_it300_295-300}
\end{subfigure}
\caption{Effect of number of iterations, moving vortex, $R2B4$ grid, 5120 observation points, number of iterations on abscissa, cost function on ordinate  
                                           a) iterations from 1 to 10 
                                           b) iterations from 11 to 20
                                           c) iterations from 295 to 300
                                           d) background term only, iterations from 295 to 300
                                           e) observation term only, iterations from 295 to 300
                                           }
\label{fig:MV5120_it300}
\end{figure}

\begin{figure}[h]
\begin{subfigure}{0.5\textwidth}
  \centering
  \includegraphics[width=\textwidth]{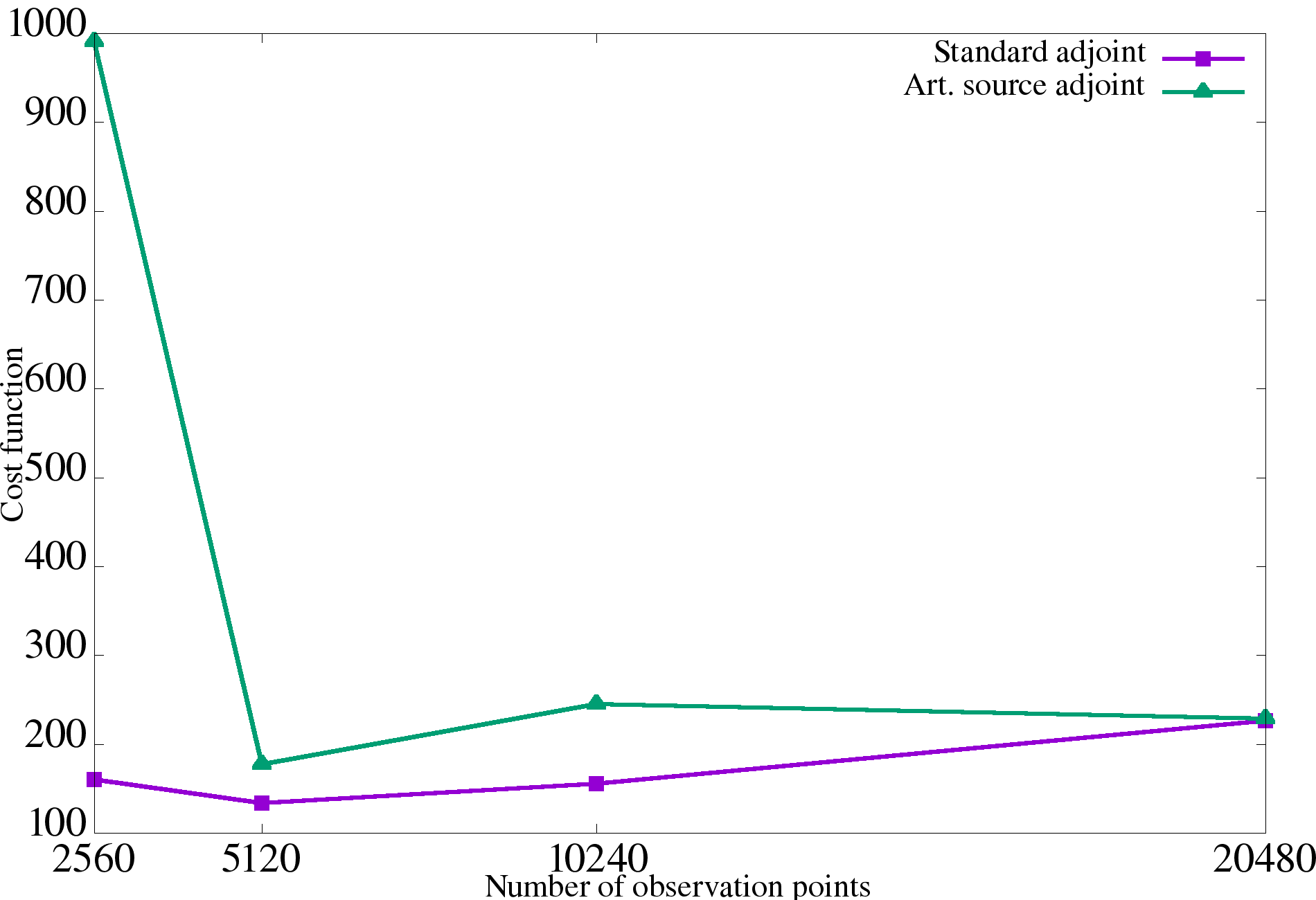}
  \caption{}
  \label{fig:MV_R2B4_cost_it50}
\end{subfigure}%
\begin{subfigure}{0.5\textwidth}
  \centering
  \includegraphics[width=\textwidth]{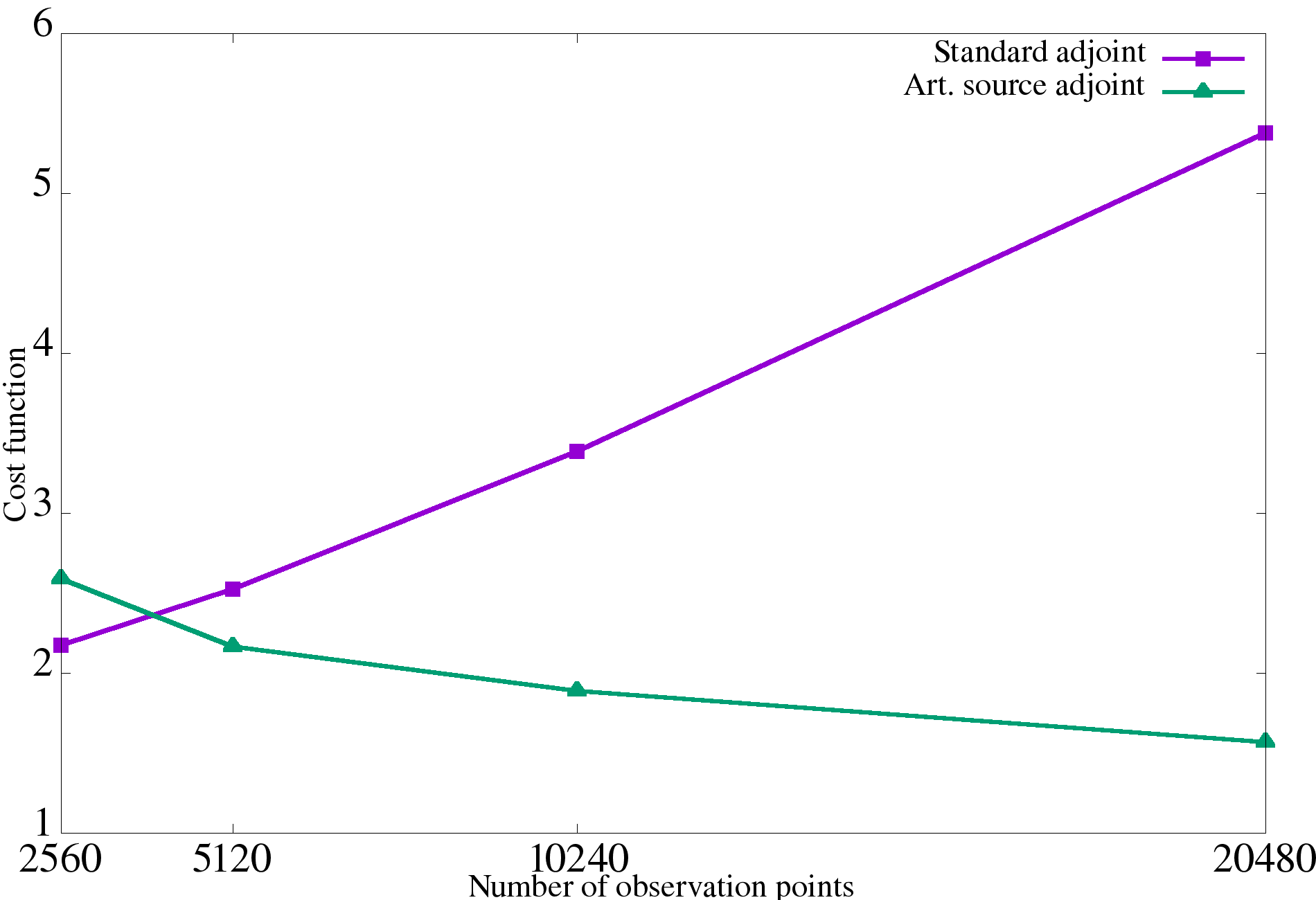}
  \caption{}
  \label{fig:DF3cos_R2B4_cost_it50}
\end{subfigure}
\begin{subfigure}{0.5\textwidth}
  \centering
  \includegraphics[width=\textwidth]{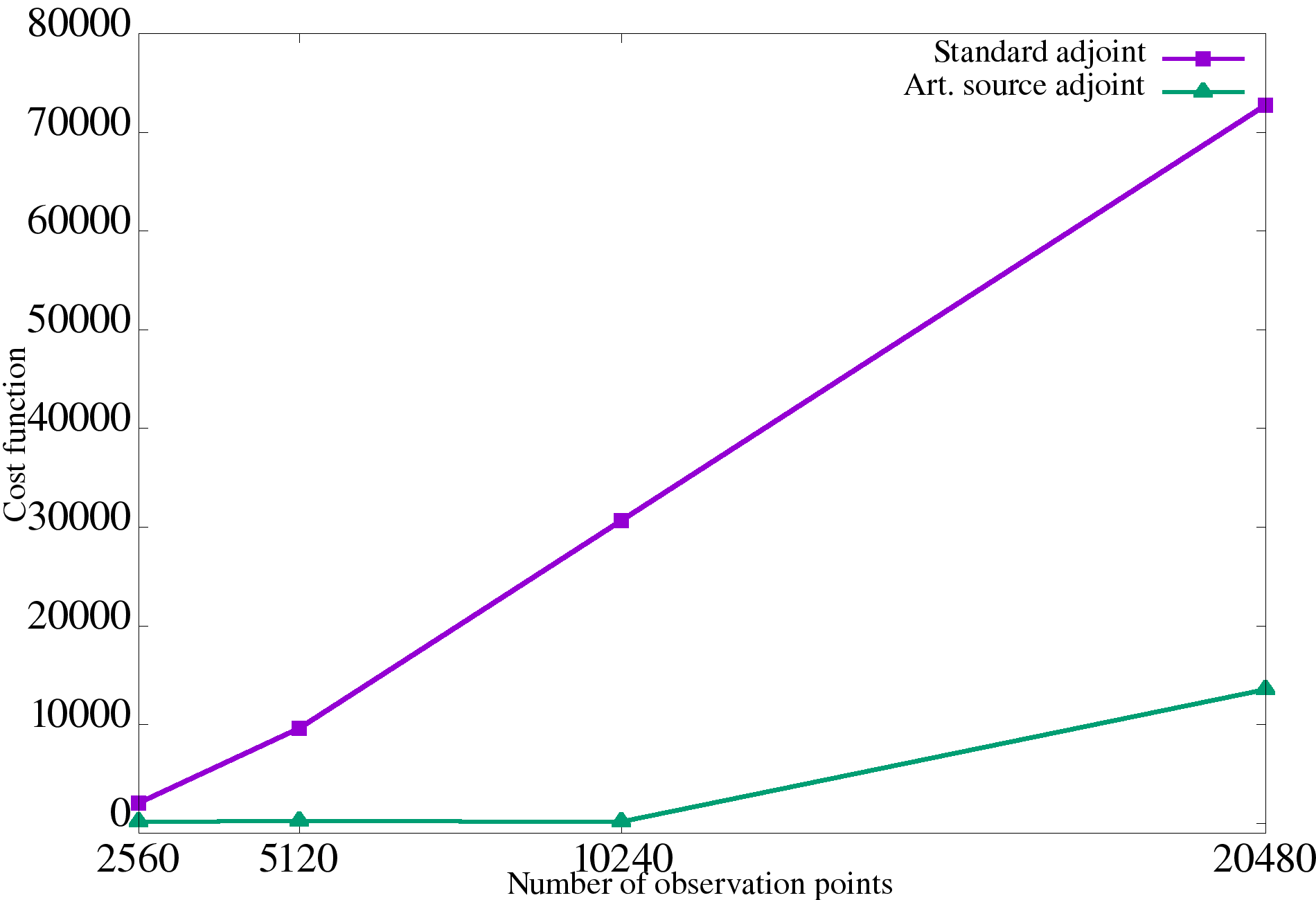}
  \caption{}
  \label{fig:DF3cyl_R2B4_cost_it50}
\end{subfigure}
\caption{Effect of observations, $R2B4$ grid, number of observation points on abscissa, cost function on ordinate 
                                           a) moving vortex
                                           b) deformational flow, cosine bells
                                           c) deformational flow, slotted cylinders
                                           }
\label{fig:R2B4_it50}
\end{figure}

\begin{figure}[h]
\begin{subfigure}{0.5\textwidth}
  \centering
  \includegraphics[width=\textwidth]{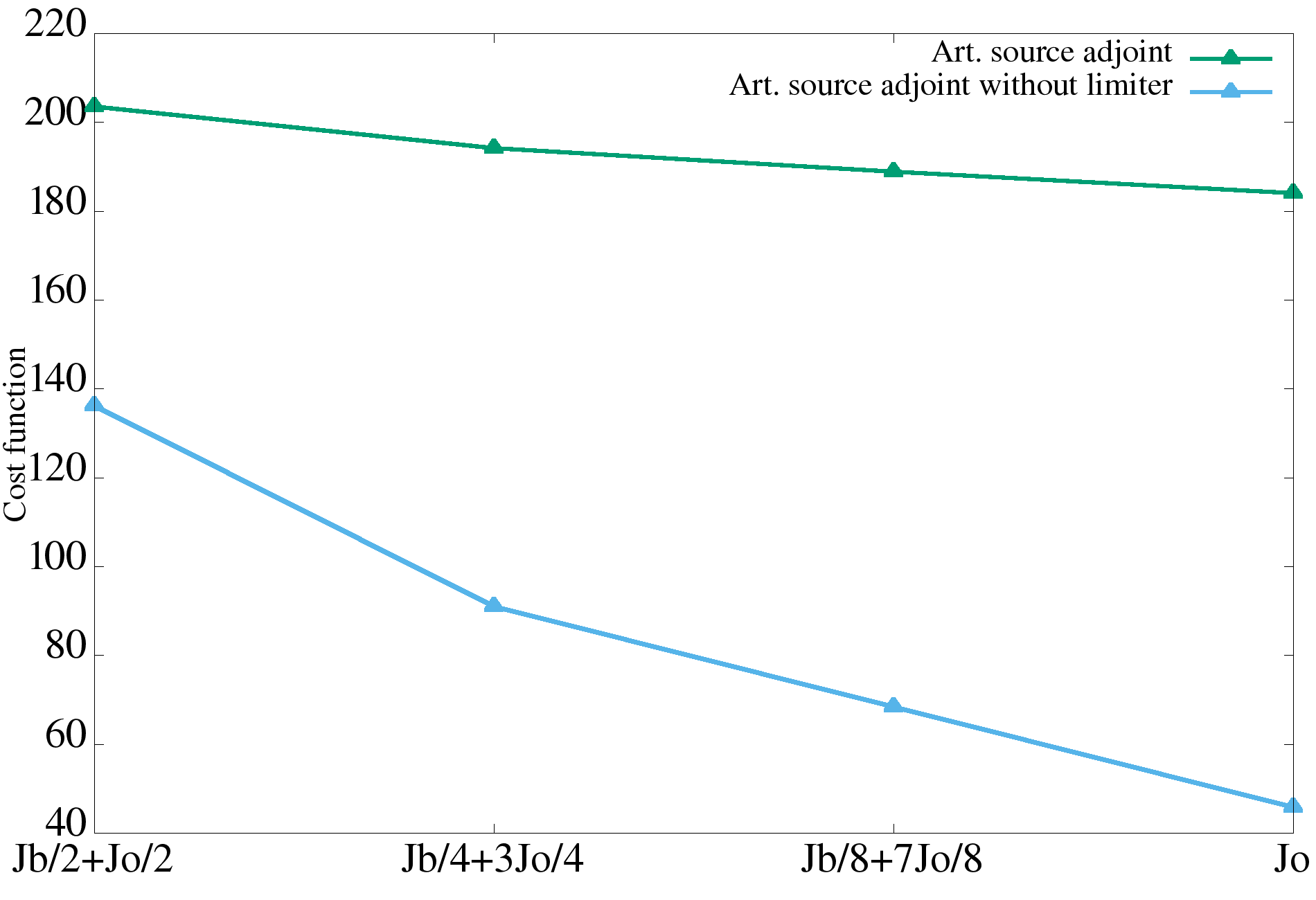}
  \caption{}
  \label{fig:MV_weights}
\end{subfigure}%
\begin{subfigure}{0.5\textwidth}
  \centering
  \includegraphics[width=\textwidth]{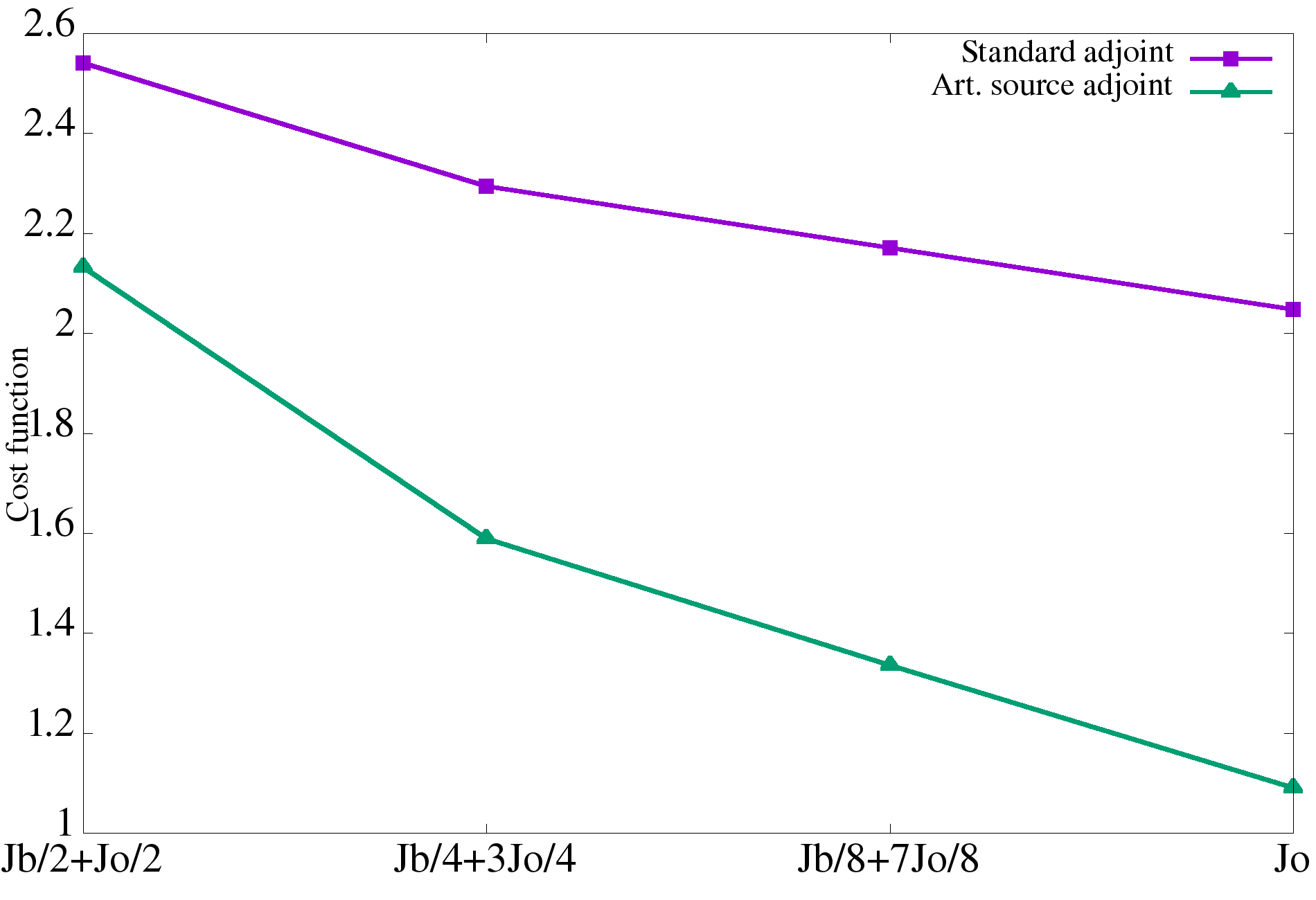}
  \caption{}
  \label{fig:DF3cos_weights}
\end{subfigure}
\begin{subfigure}{0.5\textwidth}
  \centering
  \includegraphics[width=\textwidth]{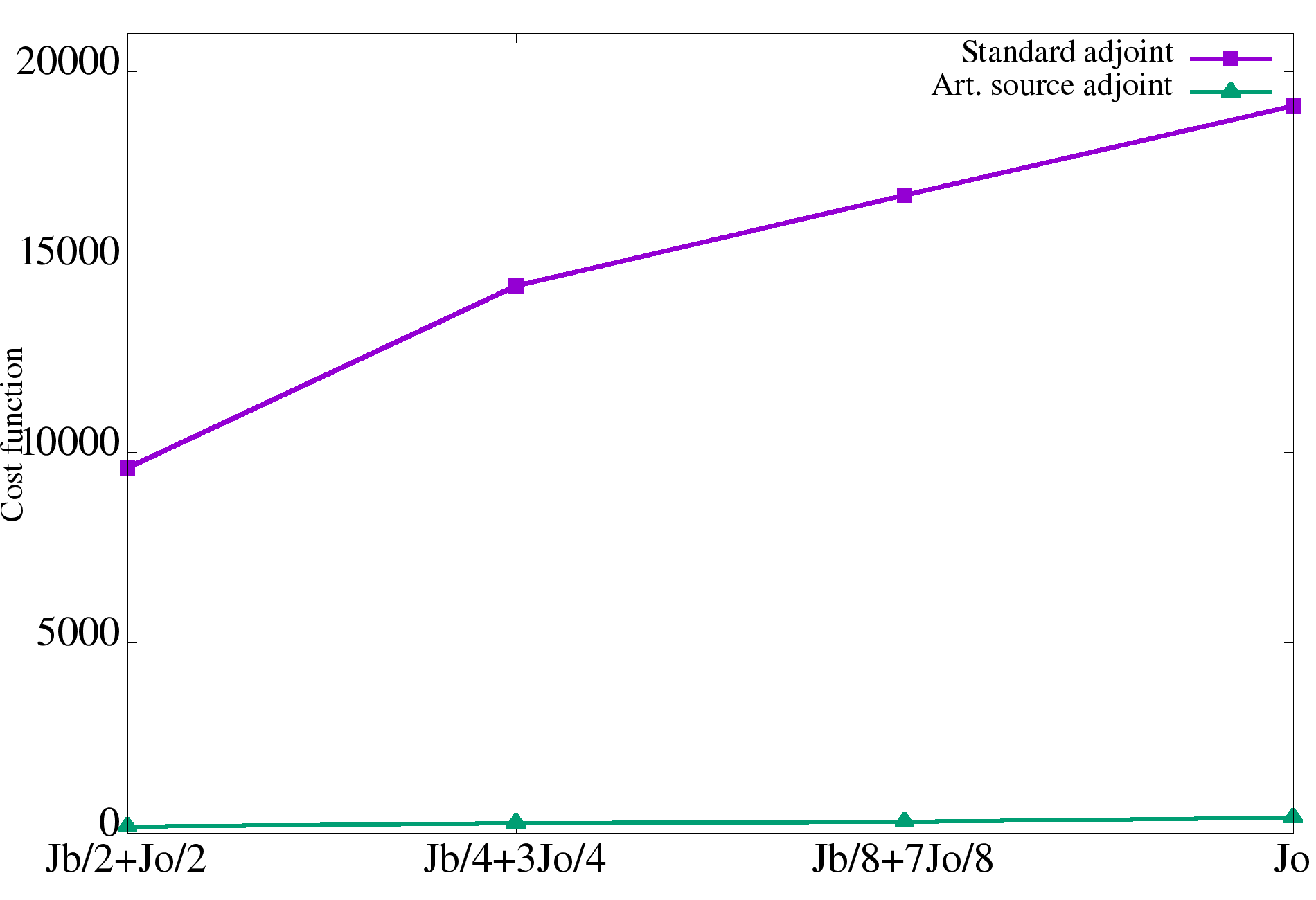}
  \caption{}
  \label{fig:DF3cyl_weights}
\end{subfigure}
\caption{Effect of weights of background and observation terms, $R2B4$ grid, 5120 observation points, cost function with 
         different weights for background and observation term on abscissa, value of cost function on ordinate
                                           a) moving vortex
                                           b) deformational flow, cosine bells
                                           c) deformational flow, slotted cylinders
                                           }
\label{fig:DaWeights}
\end{figure}

\clearpage
\subsubsection{Tables for data assimilation tests}
\label{A:DataAssimTestTables}

\begin{table}[h!]
	\centering
	\caption{Error in initial scalar field before and after assimilation. Cost function $J=J_o$, grid $R2B04$, 5120 observation points, 50 iterations}	
\begin{adjustbox}{width=1\linewidth}		
\begin{tabular}{>{\centering}p{0.1cm} >{\centering}p{0.1cm}  c c c c}
	\hline	
	$\bm{v}$ & $\bm{q_0}$ &\textbf{Error norms} & \textbf{Initial error} & \textbf{Standard adjoint} &\textbf{Art. source adjoint} \\ 
	\hline
	3 & 4 & $l_{1,rel}$       & 1.69E-01   & 1.79E-02 & 1.93E-02 \\
	  &   & $l_{2,rel}$       & 8.39E-02   & 1.45E-02 & 1.94E-02 \\
	  &   & $l_{\infty,rel}$  & 1.0E-01    & 1.64E-02 & 5.45E-02 \\ 
	  &   & $l_{1,abs}$       & 1.27E+02   & 1.36E+01 & 1.48E+01 \\
	  &   & $l_{2,abs}$       & 1.75       & 3.06E-01 & 4.14E-01 \\
	  &   & $l_{\infty,abs}$  & 1.0E-01    & 1.64E-02 & 5.45E-02\\	  
	\hline
	3 & 5 & $l_{1,rel}$       & 9.32E-02  & 1.42E-01 & 1.86E-02 \\
	  &   & $l_{2,rel}$       & 7.37E-02  & 1.58E-01 & 2.0E-02 \\
	  &   & $l_{\infty,rel}$  & 1.0E-01   & 8.7E-01  & 1.38E-01 \\ 
	  &   & $l_{1,abs}$       & 1.98E+02  & 3.02E+02 & 3.98E+01 \\
	  &   & $l_{2,abs}$       & 3.39      & 7.3      & 9.22E-01 \\
	  &   & $l_{\infty,abs}$  & 1.0E-01   & 8.7E-01  & 1.38E-01 \\
	\hline
	\hline	
	$\bm{v}$ & $\bm{q_0}$ &\textbf{Error norms} & \textbf{Initial error} & \textbf{Art. source adjoint} &\textbf{Art. source adjoint} \\ 
	         &            &                     &                        & \textbf{without limiter}     &                             \\
	\hline         
	4 & 3 & $l_{1,rel}$       & 1.0E-01    & 1.09E-03 & 2.4E-03 \\
	  &   & $l_{2,rel}$       & 1.0E-01    & 3.55E-03 & 4.4E-03 \\
	  &   & $l_{\infty,rel}$  & 1.0E-01    & 1.0E-01  & 1.99E-01 \\ 
	  &   & $l_{1,abs}$       & 2.05E+03   & 2.31E+01 & 4.98E+01 \\
	  &   & $l_{2,abs}$       & 1.5E+01    & 5.79E-01 & 7.0E-01 \\
	  &   & $l_{\infty,abs}$  & 1.54E-01   & 1.54E-01 & 3.06E-01 \\  
	\hline
	\hline	
\end{tabular}
\end{adjustbox}	
	\label{tab:Datass_norms_weights_0_1}	
\end{table}

\begin{table}
\centering
\begin{threeparttable}
\caption{Effect of mesh refinement on cost function, 20480 observation points, 50 iterations}
\label{tab:part_3_1_summary}
\begin{tabular}{c c c c c c}
     \hline
     $\bm{v}$ & $\bm{q_0}$ & \textbf{Grid} & \textbf{Cost} & \textbf{Art. source adjoint} & \textbf{Art. source adjoint} \\
              &            &               &               & \textbf{without limiter}     & \textbf{with limiter } \\
     \hline 
      3       & 4          & $R2B4$        & Initial       & 1.38305576E+05      & 1.382259723E+05     \\
              &            &               & Final         & 5.43400710          & 1.568333807         \\
              &            & $R2B5$        & Initial       & 1.38485396E+05      & 1.384239895E+05     \\
              &            &               & Final         & 7.25133893          & 7.422831065         \\
              &            & $R2B6$        & Initial       & 1.38514151E+05      & 1.384673924E+05     \\
              &            &               & Final         & 6.07498723E+04      & 6.838604955E+01\tnote{1} \\  
              &            & $R2B7$        & Initial       & 1.38572302E+05      & 1.384887423E+05     \\
              &            &               & Final         & 1.38570810E+05\tnote{1} & 2.653808533E+02     \\        
     \hline       
      3       & 5          & $R2B4$        & Initial       & 9.28758721E+05      & 4.88312986E+05      \\
              &            &               & Final         & 7.42488109E+04      & 1.24677839E+02      \\
              &            & $R2B5$        & Initial       & 7.74514390E+05      & 4.98244461E+05      \\
              &            &               & Final         & 2.89551111E+04      & 2.49163938E+02      \\
              &            & $R2B6$        & Initial       & 7.22016145E+05      & 5.02445988E+05      \\
              &            &               & Final         & 4.64099399E+05      & 6.98240016E+02      \\   
              &            & $R2B7$        & Initial       & 1.44938992E+06      & 5061026074E+05      \\
              &            &               & Final         & 1.12492223E+06      & 1.38281501E+03      \\         
     \hline  
      4       & 3          & $R2B4$        & Initial       & 9.79289433E+06     & 9.79289444E+06       \\
              &            &               & Final         & 2.26422983E+02     & 2.29100238E+02       \\
              &            & $R2B5$        & Initial       & 9.79276206E+06     & 9.79276221E+06       \\
              &            &               & Final         & 4.79062588E+02     & 2.42988014E+02       \\
              &            & $R2B6$        & Initial       & 9.79280057E+06     & 9.79280054E+06       \\
              &            &               & Final         & 9.79280057E+06     & 3.04970940E+05       \\ 
              &            & $R2B7$        & Initial       &   ---              &  ---                 \\
              &            &               & Final         &   ---              &  ---                 \\         
     \hline   
\end{tabular}
\label{tab:Datass_mesh_ref}
\begin{tablenotes}
\item[1] Iteration 51, as iteration 50 ws the first iteration after restart, where $\alpha=1$ 
\end{tablenotes}
\end{threeparttable}
\end{table}

   \clearpage
\subsection{Coefficients of adjoint scheme}
\label{A:Coef4Adjoint}

\begin{equation*}
\left.\begin{aligned}
&\alpha_{j0}=\sum_{e=1}^3\gamma_e\beta_{j0}^e, \ \ \alpha_{j1}=\sum_{e=1}^3\gamma_e\beta_{j1}^e, \ \
\alpha_{j2}=\sum_{e=1}^3\gamma_e\beta_{j2}^e, \ \ \alpha_{j3}=\sum_{e=1}^3\gamma_e\beta_{j3}^e, \\
&\alpha_{j4}=\gamma_1\beta_{j4}^1+\frac{1}{2}(1+s_2))\gamma_2\beta_{j4}^2
                                            +\frac{1}{2}(1+s_3))\gamma_3\beta_{j4}^3, \\      
&\alpha_{j5}=\gamma_1\beta_{j5}^1+\frac{1}{2}(1+s_2))\gamma_2\beta_{j5}^2
                                            +\frac{1}{2}(1+s_3))\gamma_3\beta_{j5}^3, \\ 
&\alpha_{j6}=\gamma_2\beta_{j6}^2+\frac{1}{2}(1+s_1))\gamma_1\beta_{j6}^1
                                            +\frac{1}{2}(1+s_3))\gamma_3\beta_{j6}^3, \\      
&\alpha_{j7}=\gamma_2\beta_{j7}^2+\frac{1}{2}(1+s_1))\gamma_1\beta_{j7}^1
                                            +\frac{1}{2}(1+s_3))\gamma_3\beta_{j7}^3, \\  
&\alpha_{j8}=\gamma_3\beta_{j8}^3+\frac{1}{2}(1+s_1))\gamma_1\beta_{j8}^1
                                            +\frac{1}{2}(1+s_2))\gamma_2\beta_{j8}^3, \\      
&\alpha_{j9}=\gamma_3\beta_{j9}^3+\frac{1}{2}(1+s_1))\gamma_1\beta_{j9}^1
                                            +\frac{1}{2}(1+s_2))\gamma_2\beta_{j9}^3, \\    
&\alpha_{j10}=\frac{1}{2}(1-s_1))\gamma_1\beta_{j10}^1, \ \ \alpha_{j11}=\frac{1}{2}(1-s_1))\gamma_1\beta_{j11}^1, \\
&\alpha_{j12}=\frac{1}{2}(1-s_1))\gamma_1\beta_{j12}^1 +
             \frac{1}{2}(1-s_2))\gamma_2\beta_{j12}^2 \\
&\alpha_{j13}=\frac{1}{2}(1-s_2))\gamma_2\beta_{j13}^2, \ \ \alpha_{j14}=\frac{1}{2}(1-s_2))\gamma_2\beta_{j14}^2 \\  
&\alpha_{j15}=\frac{1}{2}(1-s_2))\gamma_2\beta_{j15}^2 +
             \frac{1}{2}(1-s_3))\gamma_3\beta_{j15}^3 \\
&\alpha_{j16}=\frac{1}{2}(1-s_3))\gamma_3\beta_{j16}^3, \ \ \alpha_{j17}=\frac{1}{2}(1-s_3))\gamma_3\beta_{j17}^3, \\  
&\alpha_{j18}=\frac{1}{2}(1-s_1))\gamma_1\beta_{j18}^1 +
             \frac{1}{2}(1-s_3))\gamma_3\beta_{j18}^3 \\
\end{aligned}\right.
\end{equation*}
where $\gamma_e=s_ed_el_e\overline{v}_e$. $s_{e}$ indicates the orientation of the edge on the grid and is defined 
by direction of the normal at the edge, $d_{e}$ is the layer thickness at the edge, 
$d_{e}=\frac{\widetilde{l_{e,2}}}{\widetilde{l_{e}}}d_{c,1}+(1-\frac{\widetilde{l_{e,2}}}{\widetilde{l_{e}}})d_{c,2}$,  
where $d_{c,i}, \ \ i=1,2$ is the layer thickness at cell centers on both sides of the edge, $\widetilde{l_{e}}$ is the distance between those cell centers and 
$\widetilde{l_{e,2}}$ is the distance between edge midpoint and cell center.  Layer thickness at cell centers are computed in pressure coodrinates, 
see ~ \cite{wan2009developing}. $|\overline{\Omega_e}|$ is the area of departure region at $e^{th}$ edge, $e=1,2,3$. \\

\begin{equation*}
K_{j1}=\begin{cases}
        \{0,1,2,3,4,5,6,7,8,9 \},   \ \ \ \ \ \ \ \  \text{if } \overline{v}_1>0,\\
        \{1,5,0,4,11,12,2,3,18,10\}, \ \ \text{ otherwise } 
       \end{cases}
\end{equation*}
\begin{equation*}
K_{j2}=\begin{cases}
        \{0,1,2,3,4,5,6,7,8,9 \}, \ \ \ \ \ \ \ \  \text{if } \overline{v}_2>0,\\
        \{2,6,7,0,12,13,14,15,3,1\}, \ \ \text{otherwise } 
       \end{cases}
\end{equation*}
\begin{equation*}
K_{j3}=\begin{cases}
        \{0,1,2,3,4,5,6,7,8,9 \}, \ \ \ \ \ \ \ \  \text{if } \overline{v}_3>0,\\
        \{3,0,8,9,1,2,15,16,17,18\}, \ \ \text{otherwise } 
       \end{cases}
\end{equation*}
$\vec{p}_{e}$ is tracer independent gauss quadrature vector:

\begin{align*}
  &p_{e,1}=\sum_{k=1}^4\omega_{e,k}|J_{e,k}|,\ \ p_{e,2}=\sum_{k=1}^4\omega_{e,k}|J_{e,k}|\lambda_{e,k}, 
  \ \ p_{e,3}=\sum_{k=1}^4\omega_{e,k}|J_{e,k}|\theta_{e,k},\\
  &p_{e,4}=\sum_{k=1}^4\omega_{e,k}|J_{e,k}|\lambda_{e,k}^2, \ \ p_{e,5}=\sum_{k=1}^4\omega_{e,k}|J_{e,k}|\theta_{e,k}^2,
  \ \ p_{e,6}=\sum_{k=1}^4\omega_{e,k}|J_{e,k}|\lambda_{e,k}\theta_{e,k},\\
  &p_{e,7}=\sum_{k=1}^4\omega_{e,k}|J_{e,k}|\lambda_{e,k}^3, p_{e,8}=\sum_{k=1}^4\omega_{e,k}|J_{e,k}|\theta_{e,k}^3, 
  \ \ p_{e,9}=\sum_{k=1}^4\omega_{e,k}|J_{e,k}|\lambda_{e,k}^2\theta_{e,k},\\
  &p_{e,10}=\sum_{k=1}^4\omega_{e,k}|J_{e,k}|\lambda_{e,k}\theta_{e,k}^2. \\
 \end{align*}

$(\lambda_{e,i},\theta_{e,i})$ is the gauss quadrature points, $w_{e,i}$ is the weight and $J_{e,i}$ is the Jacobian of transformation,  $i=\overline{1,4}$. \\
$\vec{c_e}$ is tracer dependent vector of 10 unknown coefficients: 

\begin{align*}
&c_{e,1}=\tilde{q_{j,e}}, \ \ c_{e,2}=\frac{\partial\tilde{q_{j,e}}}{\partial\lambda}, \ \ c_{e,3}=\frac{\partial\tilde{q_{j,e}}}{\partial\theta},
\ \ c_{e,4}=\frac{1}{2}\frac{\partial^2\tilde{q_{j,e}}}{\partial\lambda^2}, \ \ c_{e,5}=\frac{1}{2}\frac{\partial^2\tilde{q_{j,e}}}{\partial\theta^2}, \\
&. . . \\ 
 \end{align*}

Those unknown coefficients are the solution of least-square problem ~\cite{ollivier2002high}. 

\begin{figure}[h]
\centering
\includegraphics[width=0.5\textwidth]{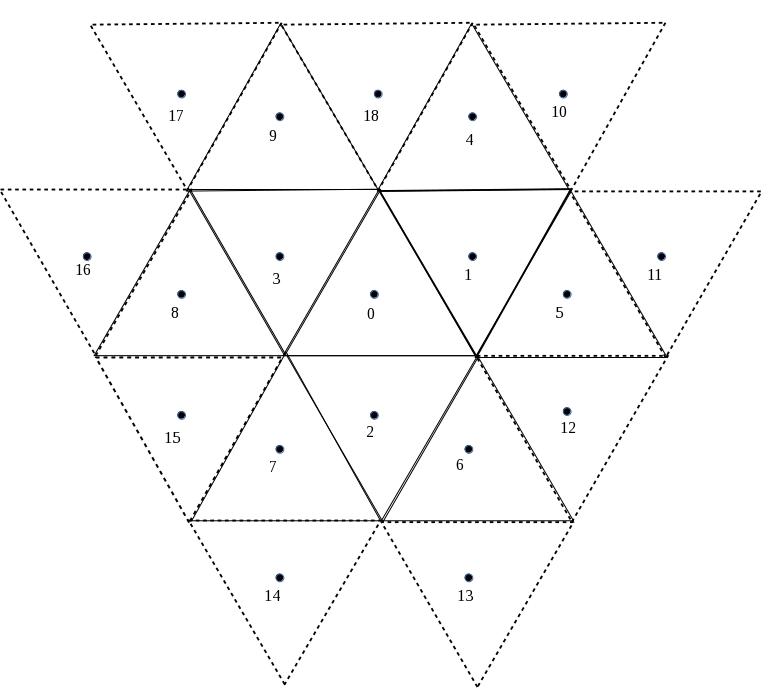}
\caption{Stencil on horizontal grid of ICON FFSL and its descendant adjoint scheme. Only bold triangles for ICON FFSL scheme stencil, bold and dashed triangles for adjoint scheme. }
\label{fig:stencil}
\end{figure}

\bibliography{adjoint4icon_arXiv}

\end{document}